\documentclass[12pt,reqno,oneside]{article}

\usepackage[letterpaper,hmargin=1in,vmargin=1in]{geometry}
\usepackage{amsmath}
\usepackage{amsthm}
\usepackage{amssymb}
\usepackage{mathtools}
\usepackage{fourier}
\usepackage{etoolbox}
\usepackage{tikz}
\usepackage{algorithm2e}

\usetikzlibrary{matrix,arrows,cd,calc}
\usepackage[hidelinks]{hyperref}
\usepackage{xcolor}
\definecolor{darkgreen}{rgb}{0,0.45,0}
\hypersetup{colorlinks,urlcolor=blue,citecolor=darkgreen,linkcolor=darkgreen,linktocpage}
\usepackage[capitalize]{cleveref}
\usepackage[style=numeric,maxnames=10,backend=bibtex,hyperref=true]{biblatex}
\usepackage{hypcap}
\usepackage{caption}
\usepackage{subcaption}

\mathtoolsset{mathic=true}

\makeatletter
\renewcommand\subsubsection{\@startsection{subsubsection}{3}{\z@}%
                                     {-3.25ex\@plus -1ex \@minus -.2ex}%
                                     {-1.5ex \@plus -.2ex}
                                     {\normalfont\normalsize\bfseries}}
\makeatother

\usepackage{tocbasic}
\DeclareTOCStyleEntry[
  beforeskip=.5em plus 1pt,
  pagenumberformat=\textbf
]{tocline}{section}
\let\LaTeXStandardTableOfContents\tableofcontents
\renewcommand{\tableofcontents}{%
\begingroup%
\renewcommand{\bfseries}{\relax}%
\LaTeXStandardTableOfContents%
\endgroup%
}%

\numberwithin{equation}{section}

\newtheorem{defn}{Definition}[section]
\newtheorem{lem}[defn]{Lemma}
\newtheorem{prop}[defn]{Proposition}
\newtheorem{cor}[defn]{Corollary}
\newtheorem{thm}[defn]{Theorem}

\newtheorem{thmx}{Theorem}

\theoremstyle{remark}
\newtheorem{rmk}[defn]{Remark}

\newtheorem{nota}[defn]{Notation}
\newtheorem{constr}[defn]{Construction}
\newcommand{\addQEDstyle}[2]{\AtBeginEnvironment{#1}{\pushQED{\qed}\renewcommand{\qedsymbol}{#2}}\AtEndEnvironment{#1}{\popQED}}
\addQEDstyle{constr}{$\lhd$}

\def\noteson{\gdef\note##1{\noindent{\color{blue}[##1]}}}

\noteson

\newcommand{\define}[1]{\textbf{\boldmath{#1}}}
\newcommand{\leng}[1]{{\scriptscriptstyle #1}}
\newcommand{\Z}{\mathbb{Z}}

\newcommand{\R}{\mathbb{R}}
\newcommand{\B}{\mathcal{B}}

\newcommand{\N}{\mathbb{N}}

\newcommand{\M}{\mathcal{M}}
\newcommand{\Mob}{\mathsf{Mob}}

\newcommand{\U}{\mathcal{U}}
\newcommand{\bO}{\mathcal{O}}
\newcommand{\V}{\mathcal{V}}
\renewcommand{\P}{\mathsf{Proj}}

\newcommand{\VR}{\mathsf{VR}}
\newcommand{\cl}{\mathsf{Cl}}

\newcommand{\Ten}{\mathsf{T}}

\newcommand{\Met}{\mathbf{Met}}
\newcommand{\Set}{\mathbf{Set}}

\newcommand{\id}{\mathsf{id}}

\newcommand{\GR}[1]{\mathsf{Gr}(#1)}
\newcommand{\ST}[1]{\mathsf{V}(#1)}
\renewcommand{\phi}{\varphi}

\DeclareMathOperator{\trace}{tr}
\DeclareMathOperator{\med}{med}
\DeclareMathOperator{\reach}{reach}
\DeclareMathOperator*{\colim}{colim}
\newcommand{\OP}{\mathsf{op}}
\newcommand{\vv}{\mathsf{cl}}
\newcommand{\Pin}{\mathsf{Pin}}
\newcommand{\Spin}{\mathsf{Spin}}
\newcommand{\ad}{\rho}
\newcommand{\av}{\mathsf{av}}
\newcommand{\atl}{\mathsf{triv}}
\newcommand{\bw}{\mathsf{w}}
\newcommand{\st}{\mathsf{st}}
\newcommand{\sw}{\mathsf{sw}}
\newcommand{\eu}{\mathsf{eu}}

\newcommand{\Prin}{\mathsf{Prin}}
\newcommand{\Vect}{\mathsf{Vect}}
\newcommand{\diam}{\mathsf{diam}}
\newcommand{\cov}{\mathsf{cov}}
\newcommand{\so}{\mathfrak{so}}
\newcommand{\sysp}{\mathsf{sys}}

\newcommand{\cech}{\mathsf{\check{C}}}
\newcommand{\cechm}{\mathsf{\check{C}}}
\newcommand{\cont}[2]{\mathsf{Maps}\left( {#1}, {#2} \right)}
\newcommand{\hcont}[2]{\left\lbrack {#1}, {#2} \right\rbrack}

\newcommand{\apprH}[4]{\check{\mathsf{H}}^{#1}_{#2}\left( #3 ; #4 \right)}
\newcommand{\aH}[1]{\apprH{1}{#1}{\U}{O(d)}}

\newcommand{\apprZ}[4]{\mathsf{Z}^{#1}_{#2}\left( #3 ; #4 \right)}
\newcommand{\aZ}[1]{\apprZ{1}{#1}{\U}{O(d)}}

\newcommand{\dapprH}[4]{\mathsf{DH}^{#1}_{#2}\left( #3 ; #4 \right)}

\newcommand{\dapprZ}[4]{\mathsf{DZ}^{#1}_{#2}\left( #3 ; #4 \right)}
\newcommand{\daZ}[1]{\dapprZ{1}{#1}{K}{O(d)}}

\newcommand{\apprA}[3]{\mathsf{T}_{#1}\left( #2 ; #3 \right)}
\newcommand{\aA}[1]{\apprA{#1}{\U}{d}}

\newcommand{\dapprA}[3]{\mathsf{DT}_{#1}\left( #2 ; #3 \right)}
\newcommand{\daA}[1]{\dapprA{#1}{K}{d}}

\newcommand{\dFrsa}{\mathsf{d}_{\mathsf{Fr}}}
\newcommand{\disHsa}{\mathsf{d}_{\mathsf{\check{H}}}}
\newcommand{\dZsa}{\mathsf{d}_{\mathsf{Z}}}
\newcommand{\dAsa}{\mathsf{d}_{\mathsf{T}}}
\newcommand{\dqsa}{\mathsf{d}_{\mathsf{q}}}
\newcommand{\dCsa}{\mathsf{d}_{\mathsf{C}}}

\newcommand{\dFr}[2]{\dFrsa\left( #1, #2 \right)}
\newcommand{\disH}[2]{\disHsa\left( #1, #2 \right)}
\newcommand{\dZ}[2]{\dZsa\left( #1, #2 \right)}
\newcommand{\dA}[2]{\dAsa\left( #1, #2 \right)}
\newcommand{\dq}[2]{\dqsa\left( #1, #2 \right)}
\newcommand{\dC}[2]{\dCsa\left( #1, #2 \right)}

\renewcommand{\epsilon}{\varepsilon}

\begin{document}

\title{Approximate and discrete Euclidean vector bundles\thanks{This work was partially supported by the National Science Foundation through grants CCF-2006661
and CAREER award  DMS-1943758.}}
\date{\vspace*{-0.5cm}}

\author{Luis Scoccola\thanks{Department of Mathematics, Northeastern University}\,\, and Jose A. Perea$^{\dagger}$\thanks{Department of Mathematics and Khoury College of Computer Sciences, Northeastern University}}

\maketitle

\begin{abstract}
    We introduce $\epsilon$-approximate versions of the notion of Euclidean vector bundle for $\epsilon \geq 0$, which recover the classical notion of Euclidean vector bundle when $\epsilon = 0$.
    In particular, we study \v{C}ech cochains with coefficients in the orthogonal group that satisfy an approximate cocycle condition.
    We show that $\epsilon$-approximate vector bundles can be used to represent classical vector bundles when $\epsilon > 0$ is sufficiently small.
    We also introduce distances between approximate vector bundles and use them to prove that sufficiently similar approximate vector bundles represent the same classical vector bundle.
    This gives a way of specifying vector bundles over finite simplicial complexes using a finite amount of data, and also allows for some tolerance to noise when working with vector bundles in an applied setting.
    As an example, we prove a reconstruction theorem for vector bundles from finite samples.
    We give algorithms for the effective computation of low-dimensional characteristic classes of vector bundles directly from discrete and approximate representations and illustrate the usage of these algorithms with computational examples.
\end{abstract}

\renewcommand{\thefootnote}{\fnsymbol{footnote}}
\footnotetext{2020 \emph{Mathematics Subject Classification.} Primary 55R99, 55N31, 68W05; Secondary 55U99.}
\renewcommand{\thefootnote}{\arabic{footnote}}

\vspace*{-0.5cm}
\tableofcontents

\newpage
\section{Introduction}

\subsection{Context and problem statement}

The notion  of \textit{fiber bundle} is fundamental in Mathematics and Physics (\cite{HJJS, VBR2, VBR1,Bott}).
Informally, a fiber bundle with fiber $F$ consists of a continuous function $p: Y \to X$ from the \textit{total space} $Y$ to the \textit{base space} $X$,
that, locally, looks like a projection $U \times F \to U$, in the following sense: $X$ can be covered by open sets $U \subseteq X$, each equipped with a homeomorphism $i : U \times F \to p^{-1}(U)$ such that $(p\circ i)(x,f) = x$ for every $(x,f) \in U\times F$.
In particular, $p^{-1}(x)$ is homeomorphic to $F$ for every $x \in X$.
\textit{Vector bundles} are fiber bundles for which $F$ is a vector space, and a key example is the tangent bundle $T\M \to \M$ of a  $d$-dimensional differentiable manifold $\M$.
The fiber of this bundle is $\R^d$, as the tangent space at each point of $\M$ is $d$-dimensional.
The M\"obius band $\Mob \to S^1$ is another example of a vector bundle,  interpreted as a collection of $1$-dimensional real vector spaces that change orientation as one goes around the equatorial circle $S^1$.
Of particular interest are fiber bundles for which the fiber $F$ is only allowed to ``twist'' according to a certain group $G$ of automorphisms of $F$; these correspond to \textit{principal $G$-bundles}.
Vector bundles can be identified with principal $G$-bundles with $G$ the general linear group, while   \textit{Euclidean vector bundles}
arise when $G$ is the orthogonal group.

Many problems in Mathematics and Physics can be reduced to finding  sections of a fiber bundle---i.e., maps $s : X\to Y$ for which $p \circ s = \id_X$---satisfying certain properties.
For this reason, one is interested in finding computable obstructions to the existence of certain sections, and, more generally, in defining computable invariants of fiber bundles that can aid in their classification up to isomorphism.

\textit{Characteristic classes}
are   examples of such invariants  (\cite{MS}).
Indeed, any principal bundle determines a collection of elements in the cohomology of its base space, called the characteristic classes of the bundle.
This is done in such a way that isomorphic bundles have the same characteristic classes.
The \textit{Stiefel--Whitney classes} of a vector bundle
are a particular type of characteristic class,
which  provide obstructions to solving  several geometric problems.
For instance, the first Stiefel--Whitney class  determines whether or not the vector bundle is orientable,
while, for a differentiable manifold $\M$,   the Stiefel--Whitney classes of its tangent bundle   provide obstructions to embedding $\M$ in $\R^n$ for  small $n$.
Similarly, the \textit{Euler class} of an oriented vector bundle
is yet another characteristic class, which provides an obstruction to the existence of a nowhere vanishing section.

Part of the ubiquity of principal bundles stems from the fact that they can be defined in several, a posteriori, equivalent ways.
Of particular interest to us are: (1) the definition of principal $G$-bundles by means of $G$-valued \textit{\v{C}ech cocycles}, which, roughly speaking, consist of local data on the base space specifying how the fibers must be glued  to reconstruct the total space; and
(2) the definition of principal $G$-bundles by means of \textit{classifying maps}, which are continuous functions from the base space to a certain \textit{classifying space} $\B G$.

Principal bundles and their characteristic classes appear also in practical applications.
Many synchronization problems, in which independent, local measurements need to be assembled into a global quantity, can be interpreted as the problem of trivializing a  \v{C}ech cochain of pairwise alignments.
Dimensionality reduction problems, where a high-dimensional point cloud concentrated around a low-dimensional manifold needs to be represented with low distortion in a low-dimensional space, can be interpreted
as an embedding problem for which estimates of the tangential characteristic classes
can provide obstructions.
Although it is informally clear that vector bundles are relevant for these kinds of problems,   the discrete and noisy nature of the input data makes it unclear whether the data actually determine a true vector bundle, and whether topological information of this bundle can be extracted from the noisy and incomplete input.

We identify two main difficulties for working with vector bundles in a practical setting.
One comes from a discrete aspect: mathematically, vector bundles are continuous entities specified, for instance, by (continuous) cocycles or by classifying maps.
How can one specify arbitrary vector bundles on, say, the geometric realization of a finite simplicial complex using a finite amount of data?
The other difficulty comes from the fact that, in practical applications, nothing is exact (e.g., cocycle conditions from noisy pairwise alignments) so one needs a notion of vector bundle that is robust to some degree of noise.
Although the results in this paper are mainly theoretical, they are motivated by problems in applied topology.
In the rest of this introduction, we describe some of these problems in more detail.

\subsubsection{Synchronization and cocycles.}
\label{synchronization-subsection}
Consider a synchronization problem with input a set $V$ of local measurements  that are pairwise aligned by elements $\{g_{ij}\}_{i,j \in V}$ of a group $G$.
One instance of this problem arises in cryogenic electron microscopy (cryo-EM), where, broadly speaking, one seeks to reconstruct the 3D shape of a molecule from several 2D projections taken from unknown viewing directions (\cite{F,vHea}).
Here the measurements are the 2D pictures, which are pairwise aligned by elements of the rotation group $SO(2)$.
Figure \ref{fig:example1} below shows   examples
of the input data for this kind of problem.

\begin{figure}[!htb]
    \centering
    \includegraphics[width=0.28\textwidth]{"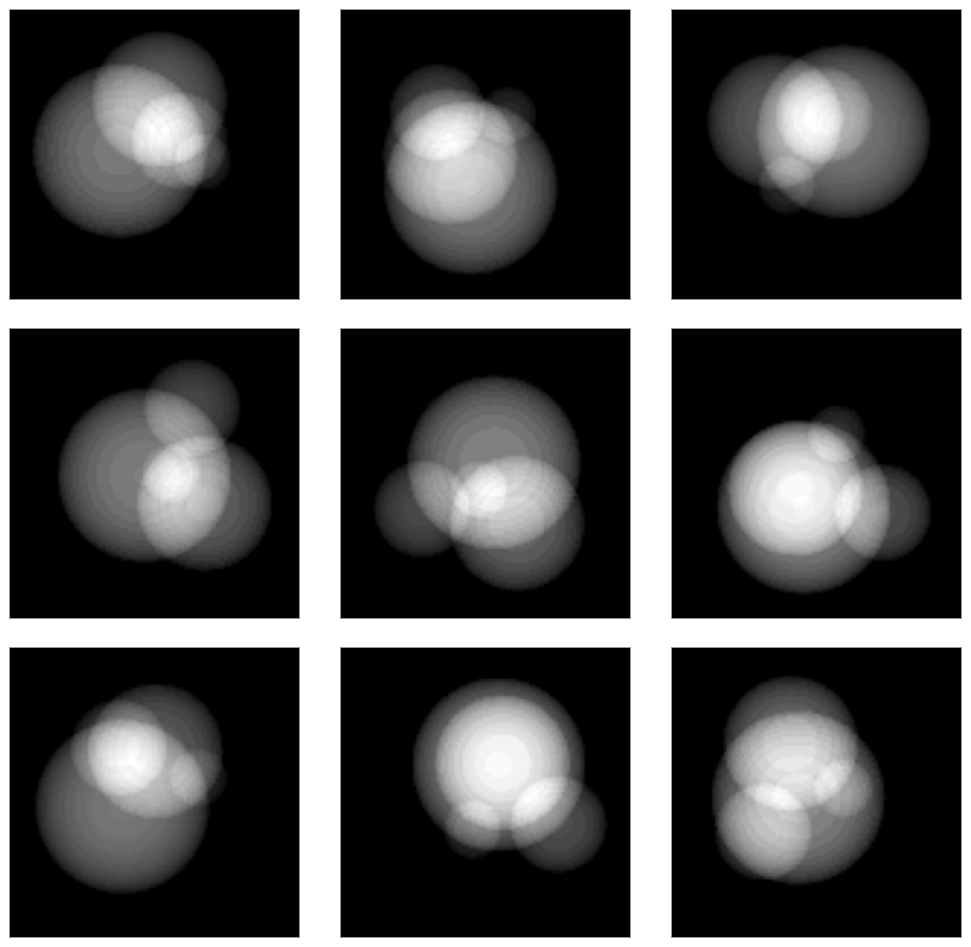"}
    \caption{2D projections of an unknown 3D shape.}\label{fig:example1}
\end{figure}

Letting $K$ be the graph with vertex set $V$ and an edge $(ij)$ if $g_{ij}$ aligns $i$ and $j$ sufficiently well, one expects the family of elements $\{g_{ij}\}_{(ij) \in K}$ to satisfy an \textit{approximate cocycle condition}.
This means that if $(ij)$, $(jk)$, and $(ik)$ are edges of $V$, then $\|g_{ij} g_{jk} - g_{ik}\| < \epsilon$ for some sufficiently small $\epsilon > 0$.
This condition indicates that the measurements can be approximately locally synchronized: in the case of cryo-EM, this implies that sets of 2D projections with very similar viewing directions can  be aligned simultaneously; this is key for averaging and thus denoising images with similar viewing directions.
Global synchronization is a different matter: in the case of cryo-EM we do not expect to be able to align all 2D projections simultaneously.
This is justified in \cite{SZSH} by observing that the pairwise alignments $\{g_{ij}\}$ approximate an $SO(2)$-cocycle  representing the tangent bundle of the $2$-sphere, which is non-trivial.
We will demonstrate in \cref{example-3} how a
data-derived estimate
of the associated Euler class yields a non-trivial obstruction
to globally aligning the data from \cref{fig:example1}.

In the general case, if the approximate cocycle $\{g_{ij}\}$ does indeed determine a true principal $G$-bundle, then global synchronization corresponds to the bundle being trivial.
This view of synchronization is studied in-depth in \cite{GBM} for the case of flat principal $G$-bundles (i.e., bundles determined by a group morphism from $\pi_1$ of the base to $G$).
However, as it is observed in \cite{YL}, the bundles underlying the cryo-EM problem are not flat.

To our knowledge, the problem of determining when an arbitrary approximate $O(d)$-valued cocycle defines a true vector bundle has not yet been addressed in the literature.
This is one of the goals of this paper.
Possible applications of a general theory of approximate $O(d)$-cocycles are the usage of characteristic classes of the underlying true vector bundle for model validation,   detecting non-synchronizability in data, and for guiding local synchronization and averaging algorithms.

\subsubsection{Local trivializations and dimensionality reduction.}
\label{motivation-local-pca}
Consider a point cloud   $X \subseteq \R^D$ concentrated around a $d$-dimensional embedded manifold $\M$,
and the ensuing problem of extracting information about the tangent bundle $T \M \to \M$ from the sample $X$.
Possible applications of this include using the tangential characteristic classes as computable obstructions for
local-to-global dimensionality reduction on $X$.

Here is an  example of such an approach. Fix $k \in \N$ and let $N_x$ denote the set of the $k$ closest neighbors of $x$ in $X$.
The application of Principal Components Analysis (PCA) to each $N_x$ with target dimension $d$
yields an ordered orthonormal basis $\Phi(x)$ of the $d$-dimensional linear subspace that best approximates $N_x$.
This method, or variations thereof, is sometimes known as \textit{local PCA} (\cite{S2}) and can be used, for instance, to estimate the local dimension of the data (\cite{LLJM}).

The aforementioned process defines a function
$\Phi : X \to \ST{d,D}$ from the data to the Stiefel manifold of orthonormal $d$-frames in $\R^D$,
which can be interpreted as an approximate vector bundle given by local trivializations.
One expects this construction to approximate the tangent bundle of the manifold $\M$. Our goal here is to formalize
this intuition, and, in particular,
to quantify the extent to which an approximate local trivialization induces a true vector bundle, and whether topological information of the bundle can be extracted from the approximation.

We apply these ideas in \cref{example-1}
to the problem of distinguishing non-homeomorphic but homotopy equivalent attractors
in the double-gyre dynamical system.
Indeed, we do so by combining
local trivializations and data-driven estimates
of the first tangential Stiefel--Whitney class.
Moreover, we also demonstrate in \cref{example-2} how the first two Stiefel--Whitney
classes can be estimated from data to yield
non-trivial obstructions  to dimensionality reduction
and parallelization.

%

\subsection{Contributions}
In this paper, we introduce relaxations and discretizations of the notion of vector bundle, we determine the extent to which these approximate representations determine true vector bundles, and we give algorithms for the effective computation of low-dimensional characteristic classes of the true vector bundle directly from approximate and discrete representations.
We run these algorithms on examples to illustrate their usage.

The notions of vector bundle we focus on correspond to Euclidean vector bundles, that is, to vector bundles endowed with a compatible, fiberwise inner product, or equivalently, vector bundles whose structure group is an orthogonal group.
To avoid cumbersome nomenclature, we drop the modifier ``Euclidean'' in our main definitions.

\paragraph{Approximate notions of vector bundle.}
Let $\epsilon \geq 0$.
For a topological space $B$ with an open cover $\U$ we define the set of $\epsilon$-approximate $O(d)$-valued $1$-cocycles $\aZ{\epsilon}$ (\cref{definition:approximate-cocycle}), which we use to define the $\epsilon$-approximate cohomology set $\aH{\epsilon}$ (\cref{definition:approximate-cohomology-set}).
We also introduce the set of $\epsilon$-approximate classifying maps $\cont{B}{\GR{d}^\epsilon}$ (\cref{definition:approximate-classifying-map}), which consists of continuous maps with codomain a thickened Grassmannian (\cref{thickenings-grassmannians-background}), and the set of $\epsilon$-approximate classifying maps up to homotopy $\hcont{B}{\GR{d}^\epsilon}$.
In order to relate these notions, we define the set of $\epsilon$-approximate local trivializations $\aA{\epsilon}$ (\cref{approx-local-triv-def}).
We introduce metrics on these sets, the most relevant of which being the metric $\disHsa$ on $\aH{\epsilon}$, which we use to state stability results, below.

When $B$ is a paracompact topological space and $\U$ is a countable open cover, we define maps between the sets, as follows:
\[
    \begin{tikzpicture}
      \matrix (m) [matrix of math nodes,row sep=2em,column sep=8em,minimum width=2em,nodes={text height=1.75ex,text depth=0.25ex}]
        { \aZ{\epsilon} & \aA{\epsilon} & \cont{B}{\GR{d}^{\epsilon}}\\
          \aH{\epsilon} &  & \hcont{B}{\GR{d}^{\epsilon}} . \\};
        \path[line width=0.75pt, -{>[width=8pt]}]
        (m-1-1) edge [bend left=10] node [above] {$\atl$} node [at end, above] {$\leng{1}$} (m-1-2)
                edge node [left] {} node [at end, right] {$\leng{1}$} (m-2-1)
        (m-1-2) edge node [above] {$\av$} node [at end, above] {$\leng{\sqrt{2}}$} (m-1-3)
                edge [bend left=10] node [below] {$\bw$} node [at end, below] {$\leng{3}$} (m-1-1)
        (m-1-3) edge node [right] {} node [at end, left] {$\leng{1}$} (m-2-3)
        (m-2-1) edge node [above] {$\vv$} node [at end, below] {$\leng{\sqrt{2}}$} (m-2-3)
        ;
    \end{tikzpicture}
\]
Here, a sub- or superscript with a constant $c$ at an arrow tip indicates that the map sends one kind of $\epsilon$-approximate vector bundle to another kind of $c\epsilon$-approximate vector bundle.
The maps $\atl$ and $\bw$ are defined in \cref{section-cocycles-and-local-trivializations}, the map $\av$ is defined in \cref{atlas-to-class-construction}, and the map $\vv$ is defined in \cref{main-thm-class-map}, which, in particular, also implies that the diagram above commutes when going from top left to bottom right.
In \cref{approximate-equivalence} we describe a way in which $\atl$ and $\bw$ are approximate inverses of each other.

By a result of Tinarrage (\cite[Lemma~2.1]{tinarrage}), there is a bijection $\pi_* : \hcont{B}{\GR{d}^{\epsilon}} \to \hcont{B}{\GR{d}}$ whenever $\epsilon \leq \sqrt{2}/2$.
In particular, when $\epsilon \leq 1/2$, the composite
\[\pi_* \circ \vv \,\,:\,\, \aH{\epsilon} \to \hcont{B}{\GR{d}}
\]
assigns a true (classical) vector bundle to every $\epsilon$-approximate cohomology class.
In this sense, an $\epsilon$-approximate cocycle determines a true vector bundle, as long as $\epsilon$ is sufficiently small.

The following result, which appears as \cref{main-thm-class-map} and \cref{relation-to-true-vb}, says in particular that the map $\vv$ is stable.

\begin{thmx}
    \label{thmA}
For any $\epsilon \geq 0$, the map $\vv : \aH{\epsilon} \to \hcont{B}{\GR{d}^{\sqrt{2}\epsilon}}$ is independent of arbitrary choices, such as a choice of a partition of unity subordinate to $\U$ or a choice of enumeration of the opens of $\U$, and is such that, if $\disH{\Omega}{\Lambda} < \delta$ in $\aH{\epsilon}$, then $\vv(\Omega)$ and $\vv(\Lambda)$ become equal in $\hcont{B}{\GR{d}^{\sqrt{2}(\epsilon + \delta)}}$.
    Moreover, if $\Omega \in \aH{\epsilon}$ and $\epsilon \leq \sqrt{2}/4$, then there exists $\Lambda \in \aH{}$ such that $\pi_*(\vv(\Omega)) = \vv(\Lambda) \in \hcont{B}{\GR{d}}$ and $\disH{\Omega}{\Lambda} \leq 9\epsilon$.
\end{thmx}

\cref{main-thm-class-map} also addresses refinements of covers and their action on approximate cohomology classes; see also \cref{approximate-cohomology} for a notion of approximate \v{C}ech cohomology independent of a cover.

\paragraph{Discrete approximate notions of vector bundle.}
For a simplicial complex $K$, we introduce discrete analogues of approximate cocycles, approximate cohomology, and approximate local trivializations.
Most importantly, we introduce the discrete approximate cohomology set $\dapprH{1}{\epsilon}{K}{O(d)}$
and the set of discrete approximate local trivializations $\daA{\epsilon}$.
These are useful in practice, as specifying elements of these sets requires a finite amount of data when $K$ is finite.
To highlight their simplicity, we give here the notion of discrete approximate cocycle over a simplicial complex $K$:
\[
    \dapprZ{1}{\epsilon}{K}{O(d)} = \left\{ \{\Omega_{ij} \in O(d)\}_{(ij) \in K_1} : \Omega_{ij} = \Omega_{ji}^t \text{ and }
                                      \| \Omega_{ij} \Omega_{jk} - \Omega_{ik} \| < \epsilon \text{ for all } (ijk) \in K_2 \right\}
\]
so that a discrete approximate cocycle consists of a set of orthogonal matrices indexed by the ordered $1$-simplices of $K$, which satisfies an approximate cocycle condition.
The definition of discrete approximate local trivialization is equally simple.

We motivate the introduction of these constructions by showing that, when $\epsilon \leq 1/2$, any discrete $\epsilon$-approximate cocycle and any discrete $\epsilon$-approximate local trivialization represent a true vector bundle on the geometric realization of $K$.
Moreover, we prove a completeness result (\cref{simplicial-completeness-thm}), which says that any vector bundle over a compact triangulable space can be represented by a discrete approximate cocycle on a sufficiently fine triangulation of the space.

We remark here that the map $\bw$ restricts to an algorithmic map from discrete approximate local trivializations to discrete approximate cocycles (\cref{relationship-discrete-appr-vb}).

\paragraph{Reconstruction from finite samples.}
Building on a result of Niyogi, Smale, and Weinberger (\cite{NSW}), we prove a reconstruction result for vector bundles on compact, embedded manifolds.
For readability, we give here a version of the result using big-$\bO$ notation and an informal notion of $(\epsilon,\delta)$-closeness; a formal statement with precise bounds is given in \cref{main-reconstruction-thm}.
In the statement, $\GR{d,n}$ is the Grassmannian of $d$-planes in $\R^n$, which we metrize using the Frobenius distance (\cref{grassmannians-background}).

\begin{thmx}
    \label{thmB}
    Let $\M \subseteq \R^N$ be a smoothly embedded compact manifold with reach $\tau > 0$ and let $f : \M \to \GR{d,n}$ be $\ell$-Lipschitz.
    Let $P \subseteq \R^N$ be a finite set and let $g : P \to \GR{d,n}$ be a function such that $(P,g)$ is $(\epsilon,\delta)$-close to $(\M,f)$.
    If $\alpha \in \left(\bO(\epsilon), \tau - \bO(\epsilon)\right) \cap \left(0,\sqrt{2}/4 \ell^{-1} - 2\delta \ell^{-1} - \epsilon\right)$, then there exists a homotopy commutative diagram as follows, in which the vertical maps are homotopy equivalences:
    \[
    \begin{tikzpicture}
      \matrix (m) [matrix of math nodes,row sep=1.5em,column sep=4em,minimum width=2em,nodes={text height=1.75ex,text depth=0.25ex}]
        { \M & \GR{n,d} \\
          \cech(P)(\alpha) & \GR{n,d}^{2(\ell \alpha + \ell \epsilon + \delta)}.\\};
        \path[line width=0.75pt, -{>[width=8pt]}]
        (m-1-1) edge node [above] {$f$} (m-1-2)
                edge node [left] {} (m-2-1)
        (m-2-1) edge node [above] {$\cechm(g)$} (m-2-2)
        (m-1-2) edge [right hook->] node [right] {} (m-2-2)
        ;
    \end{tikzpicture}
    \qedhere
    \]
    Moreover, the map $\cechm(g)$ can be represented by a discrete local trivialization $\Phi \in \dapprA{2\ell \epsilon + \delta}{\cech(P)(\epsilon)}{d}$, in the sense that $\cechm(g) = \av(\Phi)$.
\end{thmx}

Applying $\bw$ to the discrete local trivialization of \cref{thmB}, we get an approximate cocycle that can be used to compute low dimensional characteristic classes of the true vector bundle using the algorithms of \cref{section-characteristic-classes}, which we describe next.
The extent to which the approximate cocycle $\bw(\Phi)$ recovers the true vector bundle $f$ is made precise in \cref{local-trivialization-for-reconstruction}; see also \cref{reconstruction-as-cocycle}.

\paragraph{Effective computation of characteristic classes.}
Our last main contribution is the definition of algorithms for the stable and consistent computation of the $1$- and $2$-dimensional characteristic classes of a vector bundle given by an approximate $O(d)$-cocycle.

The algorithms are based on well known results which say that the characteristic classes we consider are obstructions to lifting the structure group of the cocycle from an orthogonal group to certain other Lie groups.
In this sense, the algorithms are classical and most of our work goes into adapting them to the approximate setting and into giving precise bounds for their stability and consistency.
The following theorem is a summary of the results in \cref{section-characteristic-classes}.
In the theorem, $\check{\mathsf{H}}$ denotes \v{C}ech cohomology.

\begin{thmx}
    \label{thmC}
    Let $\U$ be a cover of a topological space $B$ with the property that non-empty binary intersections are locally path connected and simply connected. There are maps
    \begin{align*}
        \sw_1 &: \aH{\epsilon} \to \apprH{1}{}{\U}{\Z/2} \text{ for $\epsilon \leq 2$};\\
        \sw_2 &: \aH{\epsilon} \to \apprH{2}{}{\U}{\Z/2} \text{ for $\epsilon \leq 1$};\\
        \eu &: \apprH{1}{\epsilon}{\U}{SO(2)} \to \apprH{2}{}{\U}{\Z} \text{ for $\epsilon \leq 1$}.
    \end{align*}

    The map $\sw_1$ is $2$-stable, in the sense that, if $\Omega, \Omega' \in \aH{\epsilon}$ with $\epsilon \leq 2$, and $\disH{\Omega}{\Omega'} < 2$, then $\sw_1(\Omega) = \sw_1(\Omega') \in \apprH{1}{}{\U}{\Z/2}$.
    In this same sense, the maps $\sw_2$ and $\eu$ are $1$-stable.

    Assume that $\U$ is countable and that $B$ is paracompact and locally contractible.
    The map $\sw_1$ is $2/9$-consistent, in the sense that, if $\Omega \in \aH{\epsilon}$ with $\epsilon < 2/9$, then $\sw_1(\Omega)$ is the first Stiefel--Whitney class of the vector bundle classified by $\pi_*(\vv(\Omega)) : B \to \GR{d}$.
    In this same sense, the map $\sw_2$ is $1/9$-consistent and computes the second Stiefel--Whitney class, and the map $\eu$ is $1/9$-consistent and computes the Euler class of an oriented vector bundle of rank $2$.
\end{thmx}

We show that the maps of \cref{thmC} are algorithmic when the input approximate cocycles are discrete approximate cocycles on a finite simplicial complex (see \cref{pseudocode-sw1}, \cref{pseudocode-eu}, and \cref{pseudocode-sw2}).
The time and space complexity of the algorithms is polynomial in the number of vertices of the simplicial complex.
For $\sw_1$, the complexity is also polynomial in $d$, while, for $\sw_2$, the complexity is exponential in $d$ and depends on calculating geodesic distances on the Spin group.
In \cref{technicalities-about-algo} we explain how to perform these calculations for small values of $d$.

\paragraph{Computational examples.}
We demonstrate our algorithms on data and show how they can be combined with persistent cohomology computations to obtain cohomological as well as bundle-theoretical information of the data; this is done in \cref{examples-section}.
A proof-of-concept implementation of our algorithms, together with code to replicate the examples, can be found at~\cite{examples-code}.

\subsection{Related work}


\paragraph{Cohomology of synchronization problems.}
Cohomological aspects of synchronization problems are discussed in \cite{YL,GBM,HG}.

In \cite{GBM}, Gao, Brodzki, and Mukherjee describe a general framework for studying cohomological obstructions to synchronization problems using principal $G$-bundles.
Although useful in many practical applications, the approach is limited by the fact that it only encompasses so-called \textit{flat} principal $G$-bundles, namely, bundles over a space $B$ classified by group homomorphisms $\pi_1(B) \to G$.
To see that this is indeed a limitation, note that no non-trivial vector bundle over $S^2$ can be represented by a group homomorphism $\pi_1(S^2) \to O(d)$, since $\pi_1(S^2) = 0$.
As mentioned previously, vector bundles over $S^2$ are central in the cryo-EM synchronization problem.

The fact that the cryo-EM bundles cannot be represented by discrete and exact $SO(2)$-cocycles on a triangulation of $S^2$ is observed by Ye and Lim in \cite{YL}, where a solution to this problem is proposed in the special case of $SO(2)$-cocycles over a $2$-dimensional simplicial complex.

\paragraph{Vector bundles over simplicial complexes.}
In \cite{KP}, Kn\"{o}ppel and Pinkall give a method for describing arbitrary complex line bundles over finite simplicial complexes using a finite amount of data.
They also describe applications to Physics and Computer Graphics.
Their theory relies on the fact that, up to isomorphism, a complex line bundle over a simplicial complex can be described exactly using a finite amount of data consisting of an angle $\eta_{ij} \in S^1$ (encoded as a complex number of absolute value $1$) for each edge $(ij)$ of the simplicial complex as well as a real number $\Omega_{ijk} \in \R$ for each $2$-simplex $(ijk)$ of the simplicial complex, which satisfy $\eta_{ki} \eta_{jk} \eta_{ij} = e^{\iota \Omega_{ijk}}$ for every $2$-simplex $(ijk)$, where $\iota$ denotes the imaginary unit. 
This fact does not seem to generalize in a direct way to vector bundles of higher rank.

In \cite{HG}, Hansen and Ghrist discuss synchronization on networks and describe a framework based on cellular sheaves.
The framework specializes to encompass flat vector bundles over a simplicial complex.
This application of their theory has the same limitation as \cite{GBM}: by having the vector bundles be described by a family of matrices $\{\Omega_{ij}\}$ indexed by the edges of the simplicial complex, subject to the restriction that $\Omega_{jk} \Omega_{ij} = \Omega_{ik}$ for every 2-simplex $(ijk)$ of the simplicial complex, only flat vector bundles can be described.

These two approaches require a certain equality to hold for each $2$-simplex of the simplicial complex.
The main difference between these approaches and the approach described in this paper is that we do not require an equality to hold for each $2$-simplex; instead, we keep track of how much a certain equality does not hold.

\paragraph{Computation of characteristic classes.}
In \cite{S2}, Singer and Wu describe an algorithm for the consistent estimation of the orientability of an embedded manifold.
This can be interpreted as estimating whether $\sw_1(\Omega)$ is zero or not, where $\Omega$ is the approximate cocycle given by the transition functions of the approximate local trivialization defined using a local PCA computation, as sketched in \cref{motivation-local-pca}.
Their approach is robust to outliers, a property not enjoyed by the approach presented in this paper.
But the algorithm does not provide the user with an actual cocycle, and thus, in the case $\sw_1(\Omega)$ is deemed to be non-zero by the algorithm, it is not clear how one can write it in a basis of the cohomology of a simplicial complex built from the data.

In \cite{tinarrage}, Tinarrage presents a framework for the consistent estimation of characteristic classes of vector bundles over embedded compact manifolds.
The input data of the framework consists of the value of a sufficiently tame classifying map on a sufficiently dense sample of the manifold.
Theoretical algorithms are given for the stable and consistent computation of arbitrary characteristic classes.
The characteristic classes are computed in a geometric way, as the algorithm uses explicit triangulations of the Grassmannians and pulls back the universal characteristic classes to a simplicial complex built from the sample.
The practicality of the algorithms is limited by the fact that the number of simplices required to triangulate a Grassmannian $\GR{d,n}$ is exponential in both $d$ and $n$ (\cite{GMP}), and the fact that the algorithm often requires iterated subdivisions of the simplicial complex built from the data.

\paragraph{Vector bundles from finite samples.}
As mentioned above, in \cite{tinarrage} Tinarrage presents a framework for the consistent estimation of characteristic classes of a vector bundle, given a sample.

In \cite{R}, Rieffel addresses the problem of giving a precise correspondence between vector bundles on metric spaces $X$ and $Y$ that are at small Gromov--Hausdorff distance.
In particular, one can use Rieffel's framework to extend a vector bundle on a sample of a manifold to the entire manifold.

\subsection{Structure of the paper}
In \cref{background-section-section} we give basic background.
In \cref{notions-approx-section} we present our three notions of approximate vector bundle and in \cref{relationships-between-notions} we relate them to each other and to classical vector bundles.
In particular, we prove \cref{main-thm-class-map} and \cref{relation-to-true-vb}.
In \cref{discrete-vector-bundles-section} we introduce our notions of discrete approximate vector bundles and show that they can be used to represent vector bundles over triangulable spaces, and to reconstruct vector bundles from finite samples.
In \cref{section-characteristic-classes} we give algorithms for the stable and consistent computation of low dimensional characteristic classes starting from a discrete and approximate cocycle.
In \cref{examples-section} we run our algorithms on examples.
In \cref{conclusions-section} we discuss open questions.

\paragraph{Acknowledgements.}
We thank Dan Christensen, Peter Landweber, Fernando Martin, and Raph\"{a}el Tinarrage for helpful conversations and comments, and Ximena Fern\'{a}ndez for discussions and suggesting the example in \cref{example-1}.
We also thank Peter Landweber for suggesting improvements to the proofs in \cref{thickening-subsection}.
Finally, we thank the anonymous reviewers for helpful feedback, which has improved this paper.
This work was partially supported by the National Science Foundation through grants   CCF-2006661
and CAREER award  DMS-1943758.

\section{Background}
\label{background-section-section}

In this section, we introduce the basic background needed to state and prove the results in this paper.
Some more technical definitions and results are in \cref{technical-details}.
In \cref{orth-grass-stief-section}, we introduce orthogonal groups, Grassmannians, and Stiefel manifolds, and fix notation.
In \cref{principal-bundles-back}, we recall some of the basics of the theory of principal bundles and vector bundles.
We assume familiarity with the very basics of algebraic topology, including cohomology.

\subsection{Orthogonal groups, Grassmannians, and Stiefel manifolds}
\label{orth-grass-stief-section}

\subsubsection{Main definitions.}
\label{grassmannians-background}
We start by recalling the definition of the Frobenius norm.
Let $l,m \in \N_{\geq 1}$ and let $\R^{l \times m}$ denote the set of $l \times m$ matrices with real coefficients.
Let $A \in \R^{l \times m}$.
The \define{Frobenius norm} of $A$ is defined by
\[
    \| A\|  = \sqrt{\trace(A^t A)} = \sqrt{\sum_{ 1 \leq i \leq l, 1 \leq j \leq m} A_{ij}^2},
\]
where $A^t \in \R^{m \times l}$ denotes the transpose of $A$.
The Frobenius norm, as any norm, induces a distance on $\R^{l \times m}$ defined by $\dFr{A}{B} = \|A - B\|$.
We refer to this distance as the \define{Frobenius distance}.

We now introduce the spaces of matrices we are most interested in.
Let $n\geq d \geq 1 \in \N$.

The \define{Grassmannian} $\GR{d,n}$ has as elements the $n \times n$ real matrices $A$ that satisfy $A = A^t = A^2$ and have rank equal to $d$, and is thus a subset of $\R^{n \times n}$.
We metrize and topologize $\GR{d,n}$ using the Frobenius distance.
Note that the elements of $\GR{d,n}$ are canonically identified with the orthogonal projection operators $\R^n \to \R^n$ of rank $d$.
Since these orthogonal projections are completely determined by the subspace of $\R^n$ which they span, it follows that the elements of $\GR{d,n}$ correspond precisely to the $d$-dimensional subspaces of $\R^n$.

The (compact) \define{Stiefel manifold} $\ST{d,n}$ has as elements the set of $n \times d$ real matrices with orthonormal columns.
We metrize $\ST{d,n}$ with the Frobenius distance.
Note that the elements of $\ST{d,n}$ can be identified with the $d$-dimensional subspaces of $\R^n$ equipped with an ordered orthonormal basis.
The elements of a Stiefel manifold are sometimes referred to as \define{frames}.

When $d = n$, the Stiefel manifold $\ST{d,n}$ coincides with the \define{orthogonal group} $O(d)$, which consists of real $d \times d$ matrices $\Omega$ such that $\Omega \Omega^t = \id$.
We metrize $O(d)$ using the Frobenius distance.
With this definition, $O(d)$ is a topological group, as matrix multiplication and inversion are continuous.

We remark here that, although we shall encounter other metrics for $O(d)$ and $\GR{d,n}$, our main results are stated using the Frobenius distance.
The relevant results about other metrics on the orthogonal group can be found in \cref{metrics-on-orth-section}, and the ones about other metrics on the Grassmannian are in \cref{metrics-on-grass-section}.

\subsubsection{Infinite dimensional Grassmannians and Stiefel manifolds.}
\label{infinite-dim-grass-section}
Let $n\geq d \geq 1 \in \N$.
There is an inclusion $\GR{d,n} \subseteq \GR{d,n+1}$ given by adding a row and a column of zeros at the bottom and right, respectively.
This inclusion is norm-preserving and thus metric-preserving.
We define the infinite dimensional Grassmannian as $\GR{d} = \bigcup_{n \in \N} \GR{d,n}$.
We topologize $\GR{d}$ using the direct limit topology;
recall that the direct limit topology on a union $X = \bigcup_{n \in \N} X_n$ of topological spaces $\{X_n\}_{n \in \N}$, where $X_n$ is a subspace of $X_{n+1}$ for all $n \in \N$, is the topology where $A \subseteq X$ is open if and only if $A \cap X_n$ is open in $X_n$ for all $n \in \N$.
Note that this topology is \textit{finer} (i.e., has more opens) than the topology induced by the metric inherited by $\GR{d}$ by virtue of it being an increasing union of metric spaces.

Similarly, we have an inclusion $\ST{d,n} \subseteq \ST{d,n+1}$ given by taking the matrix representation of a $d$-frame in $\R^n$ and adding a row of zeros at the bottom of the matrix.
Again, these inclusions are metric-preserving, and we define $\ST{d} = \bigcup_{n \in \N} \ST{d,n}$, with the direct limit topology induced by the inclusions $\ST{d,n} \to \ST{d}$.

There is a principal $O(d)$-bundle $\P : \ST{d,n} \to \GR{d,n}$, defined by mapping a $d$-frame $M$ to the matrix $MM^t$ (see \cref{principal-bundle-section} for the notion of principal bundle).
Note that the maps $\P : \ST{d,n} \to \GR{d,n}$ for each $n \geq 1$ assemble into a map
\[
    \P : \ST{d} \to \GR{d}.
\]
It is clear that $\P$ is continuous, as it restricts to a continuous map $\P : \ST{d,n} \to \GR{d,n}$ for each $n \geq 1$.

\subsubsection{Thickenings of Grassmannians.}
\label{thickenings-grassmannians-background}
For the general notion of thickening and some basic properties, we refer the reader to \cref{thickening-subsection}; here we briefly introduce the thickenings of Grassmannians, as they play an important role in our results.

We will be interested in thickenings of Grassmannians, and for that we need to include Grassmannians into a larger metric space.
A natural candidate is to let $\GR{d,n} \subseteq \R^{n\times n}$, which is metric-preserving if we metrize $\R^{n\times n}$ using the Frobenius distance.
Similarly, we have $\ST{d,n} \subseteq \R^{n \times d}$.
Analogously to what we did for Grassmannians and Stiefel manifolds, we define $\R^{\infty \times \infty} := \bigcup_{n \in \N} \R^{n\times n}$ and $\R^{\infty \times d} := \bigcup_{n \in \N} \R^{n \times d}$.

The elements of $\R^{\infty \times \infty}$ thus consist of infinite matrices with rows and columns indexed by the positive natural numbers, which have finite support, meaning that they have only finitely many non-zero entries.
Similarly, the elements of $\R^{\infty \times d}$ are matrices with finite support, with $d$ columns and rows indexed by the positive natural numbers.
Again, although $\R^{\infty \times \infty}$ and $\R^{\infty \times d}$ inherit natural metrics, we use instead the direct limit topologies induced by inclusions $\R^{n \times n} \to \R^{\infty \times \infty}$ and $\R^{n \times d} \to \R^{\infty \times d}$, respectively.

Let $\epsilon > 0$.
The \define{$\epsilon$-thickening} of $\GR{d,n}$, denoted $\GR{d,n}^\epsilon \subseteq \R^{n \times n}$, consists of all matrices in $\R^{n \times n}$ at Frobenius distance strictly less than $\epsilon$ from a matrix in $\GR{d,n}$.
Similarly, we define $\GR{d}^\epsilon \subseteq \R^{\infty \times \infty}$.
Clearly, we have $\GR{d,n}^\epsilon \subseteq \GR{d,n+1}^\epsilon \subseteq \GR{d}^\epsilon$.
The $\epsilon$-thickenings $\GR{d,n}^\epsilon \subseteq \R^{n\times n}$ and $\GR{d}^\epsilon \subseteq \R^{\infty \times \infty}$ are open subsets.
If $\epsilon = 0$, it is convenient to define $\GR{d,n}^\epsilon = \GR{d,n}$ and $\GR{d}^\epsilon = \GR{d}$.

The main result about thickenings of Grassmannians we will need is \cref{pi-is-inverse}.
The result follows directly from a result of Tinarrage (\cref{tinarrages-lemma}) and says that, if $\epsilon \leq \sqrt{2}/2$, then $\GR{d,n}^\epsilon$ deformation-retracts onto $\GR{d,n}$.

\subsection{Principal bundles and vector bundles}
\label{principal-bundles-back}

We recall the main notions and results that are relevant to this article.
For a thorough exposition, we refer the reader to, e.g., \cite[Appendices~A~and~B]{spingeometry}, \cite{Steen}, \cite{Husemoller}, and \cite{HJJS}.
The classic book \cite{MS} contains all of the standard results on vector bundles we need; see also \cite{HVB} for a modern exposition.

\subsubsection{Covers and nerve.}
A \define{cover} $\U = \{U_i\}_{i \in I}$ of a topological space $B$ consists of an indexing set $I$ together with, for every $i \in I$, an open subset $U_i \subseteq B$, such that $B = \cup_{i \in I} U_i$.
The \define{nerve} of $\U$, denoted $N(\U)$, is the simplicial complex with underlying set $I$ and simplices consisting of finite, non-empty subsets $J \subseteq I$ such that $\cap_{j \in J} U_j \neq \emptyset$.
An \define{ordered simplex} of $N(\U)$ consists of a list $(i_0 i_1 \dots i_n)$ such that $U_{i_n} \cap U_{i_{n-1}} \cap \dots \cap U_{i_0} \neq \emptyset$.
When quantifying over ordered $1$-simplices of $N(\U)$ we will write $(ij) \in N(\U)$, and, similarly, when quantifying over ordered $2$-simplices of $N(\U)$ we will write $(ijk) \in N(\U)$.

\subsubsection{Principal bundles.}
\label{principal-bundle-section}
Let $B$ and $F$ be topological spaces.
A \define{fiber bundle} over $B$ with fiber $F$ consists of a continuous map $p : E \to B$ such that, for every $y \in B$, there exists an open neighborhood $U \subseteq B$ of $y$ and a homeomorphism $\rho_U : U \times F \to p^{-1}(U)$ such that $p \circ \rho_U = \pi_1 : U \times F \to U$, where $\pi_1$ denotes projection onto the first factor.
Let $\U = \{U_i\}_{i \in I}$ be a cover of $B$.
A collection of maps $\{\rho_i : U_i \times F \to p^{-1}(U_i)\}_{i \in I}$ with the property above is referred to as a \define{local trivialization} of $p$.
By definition, any fiber bundle admits some local trivialization.

Let $G$ be a topological group.
A \define{principal $G$-bundle} over $B$ consists of a fiber bundle $p : E \to B$ with fiber $G$ together with a continuous, fiberwise right action $- \cdot - : E \times G \to E$, such that there exists a cover $\U = \{U_i\}_{i \in I}$ and a local trivialization $\{\rho_i : U_i \times G \to p^{-1}(U_i)\}_{i \in I}$ that is equivariant, meaning that, for every $g,h \in G$ and $y \in U_i$, we have $\rho_i(y,g h) = \rho_i(y,g) \cdot h \in E$.

Two principal $G$-bundles $p : E \to B$ and $p' : E' \to B$ are \define{isomorphic} if there exists a $G$-equivariant map $m : E \to E'$ such that $p' \circ m = p$.
Denote by $\Prin_G(B)$ the set of isomorphisms classes of principal $G$-bundles over $B$ (this is a set and not a proper class).

We remark that the definitions we have given are sometimes referred to as \textit{locally trivial} fiber bundle and \textit{locally trivial} principal bundle.

\subsubsection{\v{C}ech cocycles.}
Let $B$ be a topological space, $G$ a topological group, and $\U$ a cover of $B$.
A \v{C}ech $1$-\define{cocycle} subordinate to $\U$ with coefficients in $G$ consists of a family of continuous maps $\{\rho_{ij} : U_j \cap U_i \to G\}_{(ij) \in N(\U)}$ indexed by the ordered $1$-simplices of $N(\U)$, which satisfies the \define{cocycle condition}, meaning that for every $(ijk) \in N(\U)$ and $y \in U_k \cap U_j \cap U_i$, we have
\[
    \rho_{ij}(y) \rho_{jk}(y) = \rho_{ik}(y) \in G.
\]
We remark that we are using a convention for defining $\rho_{ij}$ so that $\rho_{ij}$ and $\rho_{jk}$ can be composed from left to right.
The opposite convention is also common.
The set of $1$-cocycles subordinate to $\U$ with coefficients in $G$ is denoted by $Z^1(\U;G)$.

The following construction associates a cocycle to any principal $G$-bundle and motivates the notion of cocycle.
Given a principal $G$-bundle over $B$, there exists, by definition, a cover $\U = \{U_i\}_{i \in I}$ and a family of $G$-equivariant local trivializations $\{\rho_i : U_i \times G \to p^{-1}(U_i)\}_{i \in I}$.
Note that, whenever $i,j \in I$ are such that $U_j \cap U_i \neq \emptyset$, the map $\rho_j^{-1} \circ \rho_i : (U_j \cap U_i) \times G \to (U_j \cap U_i) \times G$ is $G$-equivariant.
It follows that $\rho_j^{-1} \circ \rho_i$ induces a continuous map $\rho_{ij} : U_j \cap U_i \to G$ satisfying $(\rho_j^{-1}\circ \rho_i)(y,g) = (y, g \cdot \rho_{ij}(y))$ for all $(y,g) \in (U_j \cap U_i) \times G$, and that the family $\{\rho_{ij} : U_j \cap U_i \to G\}_{(ij) \in N(\U)}$ satisfies the cocycle condition.

\subsubsection{\v{C}ech cohomology.}
\label{cech-cohomology-section}
Let $\U = \{U_i\}_{i \in I}$ be a cover of $B$.
A \v{C}ech \define{$0$-cochain} subordinate to $\U$ with values in $G$ consists of a family of continuous maps $\Theta = \{\Theta_i : U_i \to G\}_{i \in I}$.
Let $C^0(\U;G)$ denote the set of $0$-cochains subordinate to $\U$ with values in $G$.
There is an action $C^0(\U;G) \curvearrowright Z^1(\U;G)$ with $\Theta = \{\Theta_i : U_i \to G\}_{i \in I}$ acting on a cocycle $\rho = \{\rho_{ij} : U_j \cap U_i \to G\}_{(ij) \in N(\U)}$ by $\Theta \cdot \rho = \{\Theta_i \rho_{ij} \Theta_j^{-1} : U_j \cap U_i \to G\}_{(ij) \in N(\U)}$.
The quotient of $Z^1(\U;G)$ by the action of $C^0(\U;G)$ is denoted by $\apprH{1}{}{\U}{G}$.

Let $\U = \{U_i\}_{i \in I}$ and $\V = \{V_j\}_{j \in J}$ be two covers of a common topological space $B$.
A \define{refinement} $\nu : \U \to \V$ consists of a function $\nu : I \to J$ such that for every $i \in I$ we have $U_i \subseteq V_{\nu(i)}$.
Given a refinement $\nu : \U \to \V$ and $\rho \in Z^1(\V;G)$, define a cocycle $\nu(\rho) \in Z^1(\U;G)$ by $\nu(\rho)_{ij} = \rho_{\nu(i)\nu(j)}$.
It is not hard to check that any two refinements $\nu, \nu' : \U \to \V$ induce the same map $\apprH{1}{}{\V}{G} \to \apprH{1}{}{\U}{G}$.
One then defines the \define{\v{C}ech cohomology} of $B$ with coefficients in $G$ as
\[
    \apprH{1}{}{B}{G} = \colim_{\U} \apprH{1}{}{\U}{G}
\]
where the colimit is indexed by the poset whose objects are covers of $B$ and where $\V \preceq \U$ if there exists a refinement $\U \to \V$.

As we saw previously, any principal $G$-bundle over $B$ can be trivialized over some cover of $B$ and induces a cocycle over that cover.
It is well known, and easy to see, that this construction induces a bijection
\[
    \Prin_G(B) \to \apprH{1}{}{B}{G}.
\]
For a description of principal $G$-bundles from this point of view, see \cite[Appendix~A]{spingeometry}.

\subsubsection{Vector bundles.}
\label{section:vector-bundles-def}
Let $d \in \N_{\geq 1}$.
A \define{vector bundle} of rank $d$ over $B$ consists of a fiber bundle $p : E \to B$ with fiber $\R^d$, where each fiber comes with the structure of a real vector space of dimension $d$, and such that $p$ admits a local trivialization that is linear on each fiber.

It follows that $\rho_j^{-1} \circ \rho_i$ induces a well-defined continuous map $\rho_{ij} : U_j \cap U_i \to GL(d)$ satisfying $(\rho_j^{-1}\circ \rho_i)(y,g) = (y, g \cdot \rho_{ij}(y))$ for all $(y,g) \in (U_j \cap U_i) \times \R^d$.
So every vector bundle of rank $d$ over $B$ that trivializes over a cover $\U$ gives a cocycle in $Z^1(\U;GL(d))$.

An isomorphism between vector bundles $p : E \to B$ and $p' : E' \to B$ consists of a fiberwise map $m : E \to E'$ that is a linear isomorphism on each fiber.
The family of isomorphism classes of rank-$d$ vector bundles over $B$ is denoted by $\Vect_d(B)$.

A partition of unity argument shows that, if $B$ is paracompact (in the sense of \cite[Section~5.8]{MS}), the trivialization of a vector bundle on $B$ can be taken so that the associated cocycle takes values in the orthogonal group.
When $B$ is paracompact, this gives a bijection
\[
    \Vect_d(B) \to \apprH{1}{}{B}{O(d)}.
\]
In particular $\Vect_d(B) \cong \apprH{1}{}{B}{O(d)} \cong \Prin_{O(d)}(B)$.

\subsubsection{Classifying maps.}
\label{section:classifying-maps}
For details about the claims in this section, see \cite[Theorem~3.1]{milnor-universal-2} for the existence of classifying spaces of topological groups and \cite[Appendix~B]{spingeometry} for an account of classifying spaces of Lie groups.

For every topological group $G$, there exists a \define{classifying space} $\B G$, which consists of a topological space with the property that, for every paracompact topological space $B$, there is a natural bijection
\[
    \Prin_G(B)\cong \hcont{B}{\B G}.
\]
The classifying space of the orthogonal group $O(d)$ is the Grassmannian $\GR{d}$, and thus
\[
    \Vect_d(B) \cong \apprH{1}{}{B}{O(d)} \cong \Prin_{O(d)}(B) \cong \hcont{B}{\GR{d}}.
\]
The map $\hcont{B}{\GR{d}} \to \Prin_{O(d)}(B)$ is constructed as follows.
Given $[f] \in \hcont{B}{\GR{d}}$, let $f : B \to \GR{d}$ be a representative.
Then, the pullback of $\P : \ST{d} \to \GR{d}$ along $f$ is a principal $O(d)$-bundle over $B$, whose isomorphism type is independent of the choice of representative for $[f]$.

\subsubsection{Characteristic classes.}
Let $G$ be a topological group, $B$ a paracompact topological space, and $p$ be a principal $G$-bundle over $B$.
Fix an Abelian group $A$, and let $n \in \N$.
The bundle $p$ can be represented by a continuous map $f : B \to \B G$, which can then be used to pull back any cohomology class $x \in H^n(\B G; A)$ to a class $f^*(x) \in H^n(B; A)$.
The cohomology classes $f^*(x)$ obtained in this way are the \define{characteristic classes} of the principal $G$-bundle $p$, and they are invariant under isomorphism of principal $G$-bundles.
For a presentation of characteristic classes from this point of view, see \cite[Appendix~B]{spingeometry}.

The cohomology ring $H^\bullet(\GR{d}; \Z/2)$ is isomorphic to a polynomial ring $(\Z/2)[\sigma_1, \dots, \sigma_d]$ with $\sigma_i \in H^i(\GR{d}; \Z/2)$ (\cite[Section~7]{MS}).
For $p : E \to B$ a rank-$d$ vector bundle over a paracompact topological space $B$, one defines the $i$th \define{Stiefel--Whitney class} of $p$ as the characteristic class corresponding to $\sigma_i \in H^i(\GR{d}; \Z/2)$.

For any even $d \geq 1$, there is a distinguished element $e_d \in H^d(\B SO(d); \Z)$, the \textit{universal Euler class} (\cite[Section~9]{MS},\cite[Theorem~1]{Feshbach}).
For $p : E \to B$ an oriented, rank-$d$ vector bundle over a paracompact topological space $B$, which corresponds, up to isomorphism, to an principal $SO(d)$-bundle, one defines the \define{Euler class} of $p$ as the characteristic class corresponding to $e_d \in H^d(\B SO(d); \Z)$.
We remark that the Euler class can be defined for odd $d$ too, but in this case it is often less useful as, for odd $d$, we have $2 e_d = 0$ (\cite[Property~9.4]{MS}).

\section{Three notions of approximate vector bundle}
\label{notions-approx-section}

In this section, we introduce relaxations of three standard definitions of vector bundle.
The base space of our bundles will be denoted by $B$, and a typical element will usually be denoted by $y \in B$.

The classical notions of vector bundle that we consider are those of a vector bundle given by an $O(d)$-valued \v{C}ech cocycle; a vector bundle given by a family of compatible maps from opens of a cover of $B$ to the Stiefel manifold $\ST{d}$, which we interpret as a local trivialization; and a vector bundle given by a continuous map from $B$ to the Grassmannian $\GR{d}$.

The reader may be more familiar with the notion of vector bundle given by a $GL(d)$-valued cocycle.
Being able to lift the structure group from $GL(d)$ to $O(d)$ corresponds to endowing the vector bundle with a compatible, fiberwise inner product.
Vector bundles endowed with this extra structure are sometimes referred to as \textit{Euclidean vector bundles}.
Any vector bundle over a paracompact space can be endowed with an inner product in an essentially unique way (\cite[Problems~2-C~and~2-E]{MS}), and thus is isomorphic to the underlying vector bundle of a Euclidean vector bundle.
The main reason for working with Euclidean vector bundles is that, for many choices of distance on the orthogonal group $O(d)$, the group acts on itself by isometries.
This is not the case for $GL(d)$.

\subsection{Approximate cocycles}
\label{section:approximate-cocycles}

The notion of approximate cocycle makes sense for any metric group, which we define next.
This extra generality makes some arguments clearer and will be of use in \cref{change-coeff-section}.

\begin{defn}
    A \define{metric group} consists of a group $G$ endowed with the structure of a metric space, such that left and right multiplication by any fixed element $g \in G$, and taking inverse are all isometries $G \to G$.
\end{defn}

The main example of metric group to keep in mind is that of the orthogonal group $O(d)$ endowed with the Frobenius distance.
Other relevant examples include all connected Lie groups endowed with the geodesic distance induced by a bi-invariant Riemannian metric; recall that all compact Lie groups, such as the orthogonal groups, the unitary groups, and the compact symplectic groups, admit a bi-invariant Riemannian metric (see, e.g., \cite[Corollary~1.4]{milnor-2}).
It is also relevant to note that the Frobenius distance on $O(d)$ does not arise as the geodesic distance induced by a Riemannian metric; \cref{metrics-on-orth-section} deals with some of the relationships between the Frobenius distance and the (usual) geodesic distance on $O(d)$.

Let $G$ be a metric group and let us denote its distance by $d_G : G \times G \to \R$.
Let $B$ be a topological space and let $\U = \{U_i\}_{i \in I}$ be a cover of $B$.

\begin{defn}
    A \define{$1$-cochain} on $B$ subordinate to $\U$ with values in $G$ consists of a family of continuous maps $\Omega = \{\Omega_{ij} : U_j \cap U_i \to G\}_{(ij) \in N(\U)}$ indexed by the ordered $1$-simplices of $N(\U)$ that is \define{symmetric}, i.e., such that for all $(ij) \in N(\U)$ and $y \in U_j \cap U_i$ we have $\Omega_{ij}(y) = \Omega_{ji}(y)^{-1}$.
\end{defn}

We denote the set of all $1$-cochains on $B$ subordinate to $\U$ with values in $G$ by $C^1(\U;G)$.
If there is no risk of confusion, we may refer to an element of $C^1(\U;G)$ simply as a cochain subordinate to $\U$.

\begin{defn}
    \label{definition:approximate-cocycle}
    Let $\epsilon \in (0,\infty]$.
    A cochain $\Omega$ subordinate to $\U$ is an \define{$\epsilon$-approximate cocycle}
    if for every $(ijk) \in N(\U)$ and every $y \in U_k \cap U_j \cap U_i$ we have $d_G\left(\Omega_{ij}(y) \Omega_{jk}(y), \Omega_{ik}(y)\right)  < \epsilon$,
    and an \define{exact cocycle} if we have $\Omega_{ij}(y) \Omega_{jk}(y) = \Omega_{ik}(y)$.
\end{defn}

We denote the set of $\epsilon$-approximate cocycles by $\apprZ{1}{\epsilon}{\U}{G} \subseteq C^1(\U;G)$, and the set of exact cocycles by either $\apprZ{1}{0}{\U}{G}$ or simply $\apprZ{1}{}{\U}{G}$.
Of course, when $\epsilon = \infty$, $\epsilon$-approximate cocycles are merely cochains, so to keep notation uniform, we denote $C^1(\U;G)$ by $\apprZ{1}{\infty}{\U}{G}$.
We endow the set $\apprZ{1}{\infty}{\U}{G}$ with the metric given by
\[
    \dZ{\Omega}{\Lambda} := \sup_{(ij) \in N(\U)} \sup_{y \in U_j \cap U_i} d_G\left(\Omega_{ij}(y), \Lambda_{ij}(y) \right),
\]
for $\Omega,\Lambda \in \apprZ{1}{\infty}{\U}{G}$.
This induces a metric on all spaces of approximate and exact cocycles.
In particular, for $\epsilon \leq \epsilon'$, we have metric embeddings
\[
    \apprZ{1}{}{\U}{G} \subseteq \apprZ{1}{\epsilon}{\U}{G} \subseteq \apprZ{1}{\epsilon'}{\U}{G} \subseteq \apprZ{1}{\infty}{\U}{G}.
\]

\begin{defn}
    A \define{$0$-cochain} subordinate to $\U$ with values in $G$ consists of a family of continuous maps $\Theta = \{\Theta_i : U_i \to G\}_{i \in I}$.
\end{defn}

We denote the set of all $0$-cochains by $C^0(\U;G)$.
The set $C^0(\U;G)$ forms a group, by pointwise multiplication.
There is an action $C^0(\U;G) \curvearrowright \apprZ{1}{\infty}{\U}{G}$ with $\Theta$ acting on $\Omega$ by
\[(\Theta \cdot \Omega)_{ij}(y) = \Theta_i(y)\, \Omega_{ij}(y)\, \Theta_j^{-1}(y)\] for every $y \in U_j \cap U_i$.
Since $G$ acts on itself by isometries, the above action restricts to an action on $\apprZ{1}{\epsilon}{\U}{G}$ for every $\epsilon \geq 0$.

\begin{defn}
    \label{definition:approximate-cohomology-set}
    Let $\epsilon \in [0,\infty]$.
    Define the \define{$\epsilon$-approximate cohomology set} $\apprH{1}{\epsilon}{\U}{G}$ as the quotient of $\apprZ{1}{\epsilon}{\U}{G}$ by the action of $C^0(\U;G)$.
\end{defn}

\begin{nota}
    We denote a typical element of $\apprH{1}{\epsilon}{\U}{G}$ by $\Omega$, or by $[\Omega]$ if we want to refer to the equivalence class of an approximate cocycle $\Omega \in \apprZ{1}{\epsilon}{\U}{G}$.
    In the latter case, we say that the approximate cocycle $\Omega \in \apprZ{1}{\epsilon}{\U}{G}$ is a \define{representative} of the approximate cohomology class $[\Omega] \in \apprH{1}{\epsilon}{\U}{G}$.
\end{nota}

Since the action $C^0(\U;G) \curvearrowright \apprZ{1}{\infty}{\U}{G}$ is by isometries, the set $\apprH{1}{\epsilon}{\U}{G}$ inherits a metric $\disHsa$ from $\apprZ{1}{\epsilon}{\U}{G}$, given as follows.
For $\Omega,\Lambda \in \apprH{1}{\epsilon}{\U}{G}$, we have
\[
    \disH{\Omega}{\Lambda} := \inf_{\overline{\Omega},\overline{\Lambda} \in \aZ{\epsilon}} \dZ{\overline{\Omega}}{\overline{\Lambda}},
\]
where $\overline{\Omega}$ and $\overline{\Lambda}$ range over all representatives of $\Omega$ and $\Lambda$ respectively.

\begin{rmk}
    Although we will not use this in what follows, we remark that, for $\U$ a cover of a topological space and $G$ a metric group, we have constructed a filtered (or persistent) metric space $\apprH{1}{\epsilon}{\U}{G}$, parametrized by $\epsilon \in [0,\infty]$, which one can interpret as a functor $\apprH{1}{-}{\U}{G} : ([0,\infty], \leq) \to \Met$.
\end{rmk}

\subsection{Approximate local trivializations}
We start with some considerations about the notion of local trivialization of a vector bundle (\cref{section:vector-bundles-def}), which motivate our notion of approximate local trivialization.
Any rank-$d$ Euclidean vector bundle $p : E \to B$ admits an isometric local trivialization, that is, a local trivialization over an open cover $\U = \{U_i\}_{i \in I}$ of $B$ given by a family of homeomorphisms $\{U_i \times \R^d \to p^{-1}(U_i) \}_{i \in I}$ with the property that, given $y \in U_i$, we obtain an orthonormal basis of the vector space $p^{-1}(y)$ by evaluating the map $U_i \times \R^d \to p^{-1}(U_i)$ on $(y,e_j)$ for $1 \leq j \leq d$, where $e_j$ is the $j$th canonical basis vector of $\R^d$.
Recall that any vector bundle over a paracompact base $B$ is classified by some continuous map $B \to \GR{d}$ (\cref{section:classifying-maps}).
This implies, in particular, that, up to isomorphism of vector bundles, any vector bundle $p : E \to B$ over a paracompact base is a Euclidean vector bundle whose fibers $p^{-1}(y)$ that are not just abstract vector spaces, but $d$-dimensional subspaces of $\R^\infty$.

With the above in mind, any isometric local trivialization of a Euclidean vector bundle over a paracompact base $B$ gives us maps $\{\Phi_i : U_i \to \ST{d}\}$, where, as in \cref{grassmannians-background}, the space $\ST{d}$ stands for the Stiefel manifold of $d$-frames in $\R^\infty$.
One can check that a family of maps $\{\Phi_i : U_i \to \ST{d}\}_{i \in I}$ comes from a vector bundle over $B$ precisely when, for every intersection $U_j \cap U_i \neq \emptyset$ and every $y \in U_j \cap U_i$, we have that $\Phi_i(y)$ and $\Phi_j(y)$ span the same subspace of $\R^\infty$, and thus, equivalently, when there exist continuous maps $\{\Omega_{ij} : U_j \cap U_i \to O(d)\}_{(ij) \in N(\U)}$ such that $\Phi_i(y) \Omega_{ij}(y) = \Phi_j(y)$ for all $y \in U_j \cap U_i$.
Our notion of approximate local trivialization is based on this last equivalent characterization of local trivializations, and, specifically, on relaxing this last equality.

%

Let $\U = \{U_i\}_{i \in I}$ be a cover of a topological space $B$.

\begin{defn}
    \label{approx-local-triv-def}
    Let $\epsilon \in (0,\infty]$.
    An \define{$\epsilon$-approximate local trivialization} subordinate to $\U$ consists of a family of continuous maps $\Phi = \{\Phi_i : U_i \to \ST{d}\}_{i \in I}$ such that, for every $(ij) \in N(\U)$, there exists a continuous map $\Omega_{ij} : U_j \cap U_i \to O(d)$ such that, for every $y \in U_j \cap U_i$, we have $\| \Phi_i(y) \Omega_{ij}(y) - \Phi_j(y)\|  < \epsilon$.
    An \define{exact local trivialization} is an approximate local trivialization for which $\Phi_i(y) \Omega_{ij}(y) = \Phi_j(y)$.
\end{defn}

We say that the family $\Omega = \{\Omega_{ij}\}_{(ij) \in N(\U)}$ in \cref{approx-local-triv-def} is a \define{witness} of the fact that $\Phi$ is an $\epsilon$-approximate (or exact) local trivialization.
We remark that this witness is not part of the data of an $\epsilon$-approximate local trivialization, and that we merely require that a witness exist.

An \define{approximate local trivialization} consists of an $\epsilon$-approximate local trivialization for some $\epsilon \in [0,\infty]$.
We denote the set of $\epsilon$-approximate local trivializations subordinate to $\U$ by $\aA{\epsilon}$.
We define a metric on $\aA{\epsilon}$ by
\[
    \dA{\Phi}{\Psi} := \sup_{i \in I} \sup_{y \in U_i} \|\Phi_i(y) - \Psi_i(y)\|.
\]

An approximate local trivialization $\{\Phi_i\}_{i \in I}$ is \define{non-degenerate} if, for every $(ij) \in N(\U)$ and every $y \in U_j \cap U_i$, the $d\times d$ matrix $\Phi_i(y)^t \Phi_j(y)$ has full rank.

\subsection{Approximate classifying maps}

Recall from \cref{thickenings-grassmannians-background} the definition of the thickened Grassmannians, and recall, in particular, that the topology of $\GR{d}^\epsilon$ is the direct limit topology and not the one induced by the Frobenius metric.
Let $B$ be a topological space.

\begin{defn}
    \label{definition:approximate-classifying-map}
    An \define{$\epsilon$-approximate classifying map} consists of a continuous map $B \to \GR{d}^\epsilon$.
\end{defn}

The set of $\epsilon$-approximate classifying maps is denoted by $\cont{B}{\GR{d}^\epsilon}$.
We define a metric on this set by
\[
    \dC{f}{g} = \sup_{y \in B} \| f(y)-g(y)\| .
\]

We are also interested in the set of classifying maps up to homotopy, which we denote by $\hcont{B}{\GR{d}^\epsilon}$.
Although we will not make use of this fact, we mention that one can interpret this as a persistent set $\hcont{B}{\GR{d}^{-}} : ([0,\infty], \leq) \to \Set$, in the sense of \cite[Definition~2.2]{curry}.
The following result is easily proven using a linear homotopy.

\begin{lem}
    \label{close-become-homotopic}
    Let $f,g \in \cont{B}{\GR{d}^\epsilon}$ and let $[f]$ and $[g]$ denote their images in $\hcont{B}{\GR{d}^\epsilon}$.
    Let $\delta > 0$.
    If $\|f(y) - g(y)\| < \delta$ for all $y \in B$, then $[f]$ and $[g]$ become equal in $\hcont{B}{\GR{d}^{\epsilon + \delta}}$.
    In particular, if $\dC{f}{g} < \delta$, then $[f]$ and $[g]$ become equal in $\hcont{B}{\GR{d}^{\epsilon + \delta}}$.\qed
\end{lem}

\section{Relationships between the notions}
\label{relationships-between-notions}

We now consider the problem of going back and forth between the different notions of approximate vector bundle.
As a consequence of this study, we relate approximate vector bundles to true (exact) vector bundles.
In particular, this lets us extract a true vector bundle from an $\epsilon$-approximate vector bundle when $\epsilon$ is sufficiently small.

There are two main results in this section.
\cref{main-thm-class-map} associates an approximate classifying map to any approximate cocycle, and lets us, in particular, assign a true vector bundle to any $\epsilon$-approximate cocycle as long as $\epsilon \leq 1/2$.
This is done in a way that is stable and independent of arbitrary choices.
\cref{relation-to-true-vb} gives an upper bound for the distance from an $\epsilon$-approximate cocycle to an exact cocycle representing the same true vector bundle, when $\epsilon \leq \sqrt{2}/4$.
This is used in \cref{section-characteristic-classes} to prove the consistency of algorithms to compute characteristic classes.

Many proofs in this section rely on various results stated and proven in \cref{technical-details}.

\subsection{Cocycles and local trivializations}
\label{section-cocycles-and-local-trivializations}

In this section, we relate approximate cocycles and approximate local trivializations.
We give constructions (\cref{cocycle-to-atlas-construction} and \cref{atlas-to-cocycle-construction}) to go back and forth between the notions, and we show that, in a sense, these constructions are approximate inverses of each other (\cref{approximate-equivalence}).
The construction to go from approximate local trivializations to approximate cocycles is in general not canonical; we conclude the section by showing that, when $\epsilon \leq 1$, the construction can be made canonical.

To motivate the assumptions made in the following construction, recall that a set $I$ is \define{countable} if there exists an injection $\iota : I \to \N_{\geq 1}$.
Recall also that any vector bundle on a paracompact topological space can be trivialized on a countable open cover (\cite[Lemma~5.9]{MS}),
and that every open cover of a paracompact topological space admits a subordinate partition of unity.

We start with a simplification.
Given $\V = \{V_i\}_{i \in I}$ a countable open cover of a topological space $B$ and an injection $\iota : I \to \N_{\geq 1}$, consider a new open cover $\U = \iota_*(\V)$ of $B$ indexed by $\N_{\geq 1}$ with $U_n = V_i$ if $\iota(i) = n$ or $U_n = \emptyset$ if there is no $i \in I$ such that $\iota(i) = n$.

\begin{rmk}
    \label{simplification-countable-covers}
Note that, using $\iota$, one can construct a canonical bijection between the set of partitions of unity subordinate to $\V$ and the set of partitions of unity subordinate to $\U$.
The same is true for the sets of approximate cocycles subordinate to $\V$ and $\U$, and for the sets of approximate local trivializations subordinate to $\V$ and $\U$.
\end{rmk}

We give the main constructions of this section for covers indexed by $\N_{\geq 1}$ and we will later generalize them to arbitrary countable covers, as this simplifies exposition.

\begin{constr}
    \label{cocycle-to-atlas-construction}
Let $B$ be a paracompact topological space and let $\V = \{V_i\}_{i \in \N_{\geq 1}}$ be a cover of $B$.
Let $\phi$ be a partition of unity subordinate to $\V$.
Given a cochain $\Omega$ subordinate to $\V$ define, for each $i \in \N_{\geq 1}$, a map $\Phi_i : V_i \to \ST{d}$ where
the rows of $\Phi_i(y)$ from $d\times j$ to $d \times (j+1)-1$ are given by
\[
    \sqrt{\phi_j(y)} \Omega_{ij}(y)^t.\qedhere
\]
\end{constr}

Note that the maps $\Phi_i$ of \cref{cocycle-to-atlas-construction} are continuous, and are well-defined since, if $y \not\in V_j$, then $\phi_j(y) = 0$.

\begin{lem}
    \label{approx-cocycle-gives-approx-atlas}
    Let $\epsilon \in [0,\infty]$ and let $\Omega$ be an $\epsilon$-approximate cocycle subordinate to an open cover $\V = \{V_i\}_{i \in \N_{\geq 1}}$ of a paracompact topological space.
    The maps $\Phi_i$ of \cref{cocycle-to-atlas-construction} form an $\epsilon$-approximate local trivialization.
\end{lem}
\begin{proof}
    We give the proof for $\epsilon \in (0,\infty]$, the case $\epsilon = 0$ being similar.
    Let $(ij) \in N(\V)$.
    We claim that the original $\epsilon$-approximate cocycle $\Omega$ is a witness that the family $\Phi$ is an $\epsilon$-approximate local trivialization.
    To prove this, we must show that, for all $y \in V_j \cap V_i$, the Frobenius distance between $\Phi_i(y) \Omega_{ij}(y)$ and $\Phi_j(y)$ is less than $\epsilon$.
    Carrying out the product $\Phi_i(y) \Omega_{ij}(y)$, we get an element of $\ST{d}$ with rows from $d\times k$ to $d \times (k+1)-1$ given by
    \[\sqrt{\phi_k(y)} \Omega_{ik}(y)^t \Omega_{ij}(y).\]
    So $\|\Phi_i(y) \Omega_{ij}(y) - \Phi_j(y)\| = \left(\sum_{k \geq 1} \| \sqrt{\phi_k(y)} \left(\Omega_{ik}(y)^t \Omega_{ij}(y) - \Omega_{kj}(y)\right)\| ^2\right)^{1/2} < \epsilon$, as required.
\end{proof}

We have thus constructed a map
\[
    \atl^\phi : \apprZ{1}{\epsilon}{\V}{O(d)} \to \apprA{\epsilon}{\V}{d},
\]
for any cover $\V$ of a paracompact topological space $B$ that is indexed by $\N_{\geq 1}$.
It is important to note that this map depends on the choice of partition of unity $\phi$.
Nevertheless, using two different partitions of unity gives homotopic local trivializations, in the following sense.

\begin{lem}
    \label{cocycle-to-atlas-homotopy}
    Let $\Omega$ be an $\epsilon$-approximate cocycle subordinate to $\V = \{V_i\}_{i \in \N_{\geq 1}}$.
    If $\phi$ and $\phi'$ are two partitions of unity subordinate to $\V$, then $\atl^\phi(\Omega)$ and $\atl^{\phi'}(\Omega)$ are homotopic through a family of $\epsilon$-approximate local trivializations that admit $\Omega$ as a witness.
\end{lem}
\begin{proof}
    For any $\alpha \in [0,1]$, the formula $\phi_i^\alpha = \alpha \phi_i + (1-\alpha) \phi'_i$ gives a partition of unity.
    Now observe that the family of $\epsilon$-approximate local trivializations $\atl^{\phi^\alpha}(\Omega)$ admit $\Omega$ as a witness.
\end{proof}

Next, we show that the construction assigning an approximate local trivialization to an approximate cocycle is stable.

\begin{lem}
    \label{cocycle-to-altas-stable}
    Let $\V = \{V_i\}_{i \in \N_{\geq 1}}$ be a cover of a paracompact topological space $B$ and let $\phi$ be a partition of unity subordinate to $\V$.
    For $\Omega$ and $\Lambda$ $\epsilon$-approximate cocycles subordinate to $\V$,
    \[
        \dA{\atl^\phi(\Omega)}{\atl^\phi(\Lambda)} \leq \dZ{\Omega}{\Lambda}.
    \]
\end{lem}
\begin{proof}
    Let $\dZ{\Omega}{\Lambda} = \delta$, $\Phi := \atl^\phi(\Omega)$, and $\Psi := \atl^\phi(\Lambda)$.
    For $y \in V_i$, we have
    \[
        \| \Phi_i(y) - \Psi_i(y)\| ^2 = \sum_{k \in I} \phi_k(y) \| \Omega_{ik}(y) - \Lambda_{ik}(y)\| ^2 \leq \delta^2 \sum_{k \in I} \phi_k(y) = \delta^2,
    \]
    which proves the claim.
\end{proof}

Let $\U = \{U_i\}_{i \in I}$ be a countable open cover of a paracompact topological space $B$ and let $\iota : I \to \N_{\geq 1}$ be an injection.
Using \cref{simplification-countable-covers}, we can generalize \cref{cocycle-to-atlas-construction}, \cref{approx-cocycle-gives-approx-atlas}, \cref{cocycle-to-atlas-homotopy}, and \cref{cocycle-to-altas-stable} to $\U$.

In particular, we have a map $\atl^{\phi,\iota} : \aZ{\epsilon} \to \aA{\epsilon}$ that now also depends on the choice of injection $\iota$.
We now prove that, up to homotopy, $\atl$ is independent of the choice of $\iota$.
In order to show this, we prove a more general lemma that will be of use later.

\begin{lem}
    \label{homotopic-map-stiefel}
    Given an injection $\iota : \N_{\geq 1} \to \N_{\geq 1}$ define $\chi^\iota : \ST{d} \to \ST{d}$ by mapping a frame $\Phi \in \ST{d}$ to the frame whose $\iota(k)$th row is the $k$th row of $\Phi$, and whose other rows are identically $0$.
    If $\iota, \iota' : \N_{\geq 1} \to \N_{\geq 1}$ are injections, then $\chi^\iota$ and $\chi^{\iota'} : \ST{d} \to \ST{d}$ are homotopic.
\end{lem}

    The proof of the lemma is standard, but we give it here for completeness.

\begin{proof}
    Assume that $\iota$ and $\iota'$ have disjoint images.
    Then, there is a homotopy between $\chi^\iota$ and $\chi^{\iota'}$ given by $\sqrt{\alpha} \chi^{\iota} + \sqrt{1-\alpha} \chi^{\iota'}$ for $\alpha \in [0,1]$.
    Let $2\iota : \N_{\geq 1} \to \N_{\geq 1}$ be given by $2\iota(k) = 2 \times \iota(k)$, and define $2\iota' + 1$ in an analogous way.
    Since $2\iota$ and $2\iota'+1$ have disjoint images, our previous reasoning reduces the problem to showing that $\chi^{\iota}$ and $\chi^{2\iota}$ are homotopic, and that $\chi^{\iota'}$ and $\chi^{2\iota' - 1}$ are homotopic.
    Since the two proofs are entirely analogous, we give the details only for the case of $\iota$.
    Moreover, since $2\iota : \N_{\geq 1} \to \N_{\geq 1}$ is the composite of $\iota$ with multiplication by $2$, it is enough to give the proof for $\iota = \id$, which we now do.

    We must show that the identity $\ST{d} \to \ST{d}$ is homotopic to $\chi^2 : \ST{d} \to \ST{d}$.
    Informally, the proof works by moving each row of $\Phi$ at a time.
    To simplify exposition, in the rest of this proof, the notation $\Phi_m$ will be used to refer to the $m$th row of a frame $\Phi \in \ST{d}$.
    Consider the family of functions $f^\alpha : \ST{d} \to \ST{d}$ indexed by $\alpha \in [0,1]$ defined as follows.
    For $\alpha = 0$, let $f^\alpha = \id$.
    For $n \in \N_{\geq 1}$, $\alpha \in [1/(n+1), 1/n]$, and $m \in \N_{\geq 1}$, define
    \[
        \left(f^{\alpha}(\Phi)\right)_m =
        \begin{cases}
          \Phi_m, & \text{$m < n$}\\
         \sqrt{\beta}\,\, \Phi_m,  & \text{$m = n$}\\
          0,  & \text{$n < m < 2n$}\\
         \sqrt{1-\beta}\,\, \Phi_m,    & \text{$m = 2n$}\\
          \Phi_{m/2},  & \text{$m > 2n$ and $m$ even}\\
          0,  & \text{$m > 2n$ and $m$ odd}
        \end{cases}
    \]
    where $\beta = (1/n - \alpha) \times (1/n - 1/(n+1))$.
    By inspection, we see that $f^\alpha$ gives a homotopy between the identity and $\chi^2$, using that $\ST{d}$ has the direct limit topology.
\end{proof}

\begin{lem}
    \label{cocycle-to-atlas-homotopy-injection}
    Let $\Omega$ be an $\epsilon$-approximate cocycle subordinate to a countable open cover $\U = \{U_i\}_{i \in I}$ and let $\iota,\iota' : I \to \N_{\geq 1}$ be injections.
    Let $\phi$ be a partition of unity subordinate to $\U$.
    Then $\atl^{\phi,\iota}(\Omega)$ and $\atl^{\phi,\iota'}(\Omega)$ are homotopic through a family of $\epsilon$-approximate local trivializations that admit $\Omega$ as a witness.
\end{lem}
\begin{proof}
    The result follows at once from \cref{homotopic-map-stiefel} by noticing that, for any two injections $\iota,\iota' : I \to \N_{\geq 1}$ there exists a bijection $b$ of $\N_{\geq 1}$ such that $\iota' = b \circ \iota$.
\end{proof}

We now consider the problem of assigning an approximate cocycle to an approximate local trivialization.

\begin{constr}
    \label{atlas-to-cocycle-construction}
Let $\Phi = \{\Phi_i\}_{i \in I}$ be an $\epsilon$-approximate local trivialization subordinate to $\U = \{U_i\}_{i \in I}$.
By definition, there exists a witness that $\Phi$ is an $\epsilon$-approximate local trivialization.
Choose, arbitrarily, such a witness $\Omega$.
Without loss of generality, we may assume that $\Omega$ is symmetric, and thus that it is a cochain.
\end{constr}

\begin{lem}
    \label{altas-to-cocycle-lem}
    Let $\epsilon \in [0,\infty]$.
    Let $\Phi$ be an $\epsilon$-approximate local trivialization.
    Then, the cochain described in \cref{atlas-to-cocycle-construction} is a $3 \epsilon$-approximate cocycle.
    Thus, \cref{atlas-to-cocycle-construction} gives a map $\bw : \aA{\epsilon} \to \aZ{3\epsilon}$.
\end{lem}
\begin{proof}
    We address the case $\epsilon \in (0,\infty]$, the case $\epsilon = 0$ being similar.
    Let $(ijk) \in N(\U)$ and let $y \in U_k \cap U_j \cap U_i$.
    Since $\| \Phi_i(y) \Omega_{ik}(y) - \Phi_k(y)\|  < \epsilon$, we have that $\| \Phi_k(y)^t \Phi_i(y) \Omega_{ik}(y) - \id\|  < \epsilon$, by \cref{product-by-frame-isometry}.
    This implies that $\| \Phi_i(y)^t \Phi_k(y) - \Omega_{ik}(y)\|  < \epsilon$.

    A similar computation shows that $\| \Phi_i(y)^t \Phi_j(y) \Omega_{jk}(y) - \Phi_i(y)^t \Phi_k(y)\|  < \epsilon$.
    Using the triangle inequality and the first bound in the proof, we get that $\| \Phi_i(y)^t \Phi_j(y) \Omega_{jk}(y) - \Omega_{ik}(y) \|  < 2\epsilon$.
    By \cref{product-by-frame-isometry} and the first bound in this proof but for $i$ and $j$, we have that $\| \Phi_i(y)^t \Phi_j(y) \Omega_{jk}(y) - \Omega_{ij}(y) \Omega_{jk}(y)\|  < \epsilon$.
    The triangle inequality then finishes the proof.
\end{proof}

Although we can associate a $3\epsilon$-approximate cocycle to every $\epsilon$-approximate local trivialization, this choice is not canonical, as an approximate local trivialization can have many distinct witnesses.
Nonetheless, the following result says that any two witnesses cannot be too far apart.

\begin{lem}
    \label{stability-of-bw}
    Let $\Omega$ and $\Lambda$ be witnesses that $\Phi$ and $\Psi$ are, respectively, $\epsilon$- and $\delta$-approximate local trivializations.
    Then $\dZ{\Omega}{\Lambda} \leq \epsilon + \delta + \sqrt{2} \dA{\Phi}{\Psi}$.
\end{lem}
\begin{proof}
    We address the case in which $\epsilon,\delta>0$, the case in which any of them is $0$ being similar.
    Let $(ij) \in N(\U)$ and $y \in U_j \cap U_i$.
    We have $\| \Omega_{ij}(y) - \Phi_i(y)^t \Phi_j(y)\|  < \epsilon$ and $\| \Lambda_{ij}(y) - \Psi_i(y)^t \Psi_j(y)\|  < \delta$, so it suffices to show that $\| \Phi_i(y)^t \Phi_j(y) - \Psi_i(y)^t \Psi_j(y)\|  \leq \sqrt{2} \| \Phi_i(y) - \Psi_i(y)\| $, which follows from \cref{projection-is-lipschitz}.
\end{proof}

\begin{rmk}
    \label{approximate-equivalence}
From \cref{stability-of-bw}, it follows that, if $\Omega$ and $\Lambda$ are witnesses that $\Phi$ is an $\epsilon$-approximate local trivialization, then $\dZ{\Omega}{\Lambda} \leq 2\epsilon$, and thus $\bw$ is, approximately, a left inverse of $\atl^\phi$, in the sense that $\dZ{\Omega}{\bw(\atl^\phi(\Omega))} \leq 2\epsilon$ for every $\epsilon$-approximate cocycle $\Omega$.
The following result can be interpreted as saying that $\bw$ is also a right inverse of $\atl^\phi$, since it implies, in particular, that $\atl^\phi(\bw(\Phi))$ is homotopic to $\Phi$ through $3\epsilon$-approximate local trivializations, whenever $\Phi$ is an $\epsilon$-approximate local trivialization.
\end{rmk}

\begin{lem}
    \label{w-then-triv}
    Let $\epsilon \geq 0$.
    Let $\Omega$ be a witness that $\Phi$ and $\Psi$ are $\epsilon$-approximate cocycles.
    Then $\Phi$ and $\Psi$ are homotopic through $\epsilon$-approximate local trivializations subordinate to $\U = \{U_i\}_{i \in I}$ that admit $\Omega$ as a witness.
\end{lem}
\begin{proof}
    We use the language of \cref{homotopic-map-stiefel}.
    Consider the maps $\iota, \iota' : \N_{\geq 1} \to \N_{\geq 1}$ given by $\iota(k) = 2k$ and $\iota'(k) = 2k - 1$.
    It is clear that $\Omega$ is a witness that $\{\chi^\iota \circ \Phi_i\}_{i \in I}$ is an $\epsilon$-local trivialization, and \cref{homotopic-map-stiefel} implies that $\Phi$ is homotopic to $\{\chi^\iota \circ \Phi_i\}_{i \in I}$ through $\epsilon$-approximate local trivializations that admit $\Omega$ as a witness.
    Similarly, we deduce that $\Psi$ is homotopic to $\{\chi^{\iota'} \circ \Psi_i\}_{i \in I}$ through $\epsilon$-approximate local trivializations that admit $\Omega$ as a witness.

    Consider, for $\alpha \in [0,1]$, the family $\{\sqrt{\alpha}(\chi^\iota \circ \Phi_i) + \sqrt{1-\alpha} (\chi^{\iota'} \circ \Psi_i)\}_{i \in I}$.
    Since the images of $\iota$ and $\iota'$ are disjoint, this constitutes an $\epsilon$-approximate local trivialization that admits $\Omega$ as a witness that varies continuously with $\alpha$.
    The result follows.
\end{proof}

The following result gives conditions under which there is a canonical approximate cocycle associated to an approximate local trivialization.

\begin{lem}
    \label{best-witness-lem}
    Let $\Phi$ be a non-degenerate $\epsilon$-approximate local trivialization.
    For $(ij) \in N(\U)$ and $y \in U_j \cap U_i$, let $\Omega_{ij}(y) \in O(d)$ minimize $\| \Phi_i(y) \Omega - \Phi_j(y)\| $, where $\Omega$ ranges over $O(d)$.
    Then, the matrices $\Omega_{ij}(y)$ assemble into a $3\epsilon$-approximate cocycle.
\end{lem}
\begin{proof}
    To see that the mappings $\Omega_{ij} : U_j \cap U_i \to O(d)$ are continuous, use \cref{Q-is-continuous}.
    To see that $\Omega_{ij} = \Omega_{ji}^t$ use that the minimizers are unique.
    Then, \cref{altas-to-cocycle-lem} directly implies that $\Omega$ satisfies the $3\epsilon$-approximate cocycle condition, as required.
\end{proof}

We now give sufficient conditions for an approximate local trivialization to be non-degenerate.

\begin{lem}
    \label{well-defined-cocycle-from-atlas}
    If $\epsilon \leq 1$, then any $\epsilon$-approximate local trivialization is non-degenerate.
\end{lem}
\begin{proof}
    This is a consequence \cref{key-lemma}, which says that a matrix that is at Frobenius distance less than $1$ from an orthogonal matrix is invertible.
\end{proof}

We conclude by remarking that \cref{well-defined-cocycle-from-atlas} implies that, if $\epsilon \leq 1$, we have a canonical map (that is independent of any arbitrary choices)
\[
    \bw : \aA{\epsilon} \to \aZ{3\epsilon}
\]
given by taking the best witness, as in \cref{best-witness-lem}.

\subsection{Local trivializations and classifying maps}

In this section, we give a construction that, given an approximate local trivialization, returns an approximate classifying map.
We also observe that this construction behaves well with respect to homotopies between approximate local trivializations.

Recall from \cref{infinite-dim-grass-section} the definition of the map $\P : \ST{d} \to \GR{d}$.

\begin{constr}
    \label{atlas-to-class-construction}
    Let $\U$ be a countable open cover of a paracompact topological space $B$, and let $\phi$ be a partition of unity subordinate to $\U$.
    Let $\Phi$ be an approximate local trivialization subordinate to $\U$.
    Define a map $\av^\phi(\Phi) : B \to \R^{\infty \times \infty}$ by
    \[
        \av^\phi(\Phi)(y) = \sum_{i \in I} \phi_i(y) \P(\Phi_i(y)).\qedhere
    \]
\end{constr}

Note that the map $\av^\phi(\Phi)$ is continuous.

\begin{lem}
    \label{atlas-to-class-map-lem}
    Let $\epsilon \in [0,\infty]$.
    Under the hypotheses of \cref{atlas-to-class-construction}, if $\Phi$ is an $\epsilon$-approximate local trivialization, then $\av^\phi(\Phi)$ is a $\sqrt{2}\epsilon$-approximate classifying map.
\end{lem}
\begin{proof}
    We address the case $\epsilon \in (0,\infty]$, the case $\epsilon = 0$ being similar.
    By definition, there exist, for each $(ij) \in N(\U)$, a continuous map $\Omega_{ij} : U_j \cap U_i \to O(d)$ such that $\| \Phi_i(y) \Omega_{ij}(y) - \Phi_j(y)\|  < \epsilon$ for every $y \in U_j \cap U_i$.
    Since $\P : \ST{d} \to \GR{d}$ is $O(d)$-invariant and $\sqrt{2}$-Lipschitz (\cref{projection-is-lipschitz}), it follows that $\| \P(\Phi_i(y)) - \P(\Phi_j(y))\|  < \sqrt{2}\epsilon$.
    The result then follows from \cref{distance-projection-vertex}.
\end{proof}

The following is clear.

\begin{lem}
    \label{atlas-to-class-map-stable-homotopy}
    Let $\Phi$ and $\Psi$ be approximate local trivializations subordinate to $\U$, a countable open cover of a paracompact topological space $B$.
    Let $\phi$ be a partition of unity subordinate to $\U$.
    \begin{enumerate}
        \item We have that $\dC{\av^\phi(\Phi)}{\av^\phi(\Psi)} \leq \sqrt{2} \dA{\Phi}{\Psi}$.
        \item If $\Phi$ and $\Psi$ are homotopic through $\epsilon$-approximate local trivializations subordinate to $\U$, then $\av^\phi(\Phi)$ and $\av^\phi(\Psi)$ are homotopic as maps $B \to \GR{d}^{\sqrt{2}\epsilon}$. \qed
    \end{enumerate}
\end{lem}

\subsection{Cocycles and classifying maps}
\label{cocycles-to-classifying-maps-section}

In this section we relate approximate cocycles to approximate classifying maps in a way that is independent of any partition of unity and of any enumeration of the sets in the open cover the cocycle is subordinate to (\cref{main-thm-class-map}).
We also study the action of refinements on approximate cohomology.

Let $B$ be a paracompact topological space and let $\U = \{U_i\}_{i \in I}$ be a countable open cover.
Let $\phi$ be a partition of unity subordinate to $\U$ and let $\iota : I \to \N_{\geq 1}$ be an injection.
Using \cref{approx-cocycle-gives-approx-atlas} and \cref{atlas-to-class-map-lem}, we get a map $\av^\phi \circ \atl^{\phi,\iota} : \aZ{\epsilon} \to \cont{B}{\GR{d}^{\sqrt{2}\epsilon}}$ which we compose with the quotient map $\cont{B}{\GR{d}^{\sqrt{2}\epsilon}} \to \hcont{B}{\GR{d}^{\sqrt{2}\epsilon}}$ to obtain a map
\[
    \vv' : \aZ{\epsilon} \to \hcont{B}{\GR{d}^{\sqrt{2}\epsilon}}.
\]
Together, \cref{cocycle-to-atlas-homotopy}, \cref{atlas-to-class-map-stable-homotopy}, and \cref{cocycle-to-atlas-homotopy-injection} imply that this map is independent of the choice of partition of unity $\phi$ and of injection $\iota$.
The following result says that two approximate cocycles that differ in a $0$-cochain are sent to the same approximate classifying map by $\vv'$.

\begin{lem}
    \label{cobordant-are-homotopic}
    Let $B$ be a paracompact topological space and let $\U$ be a countable open cover.
    The map $\vv' : \aZ{\epsilon} \to \hcont{B}{\GR{d}^{\sqrt{2}\epsilon}}$ factors through $\aH{\epsilon}$.
\end{lem}
\begin{proof}
    By construction, we may assume that $\U$ is indexed by $I = \N$.
    Consider the following open cover $\V = \{V_j\}_{j \in J}$ indexed by $J = \N$.
    Let $V_j = U_i$ whenever $j = 2i + z$ with $z = 0$ or $z = 1$.
    So the open cover $\V$ consists of two copies of each open set of $\U$, where each open $U_i$ appears with an even index as $V_{2i}$ and with an odd index as $V_{2i+1}$.

    Suppose that $\Omega$ and $\Omega'$ are equal in $\aH{\epsilon}$, so that there is a $0$-cochain $\Theta$ such that $\Theta \cdot \Omega = \Omega'$.
    Let $\phi$ be a partition of unity subordinate to $\U$.
    This induces two partitions of unity $\phi^0$ and $\phi^1$ subordinate to $\V$, where $\phi^0_j$ is equal to $\phi_{j/2}$ if $j$ is even and is identically $0$ if $j$ is odd.
    Similarly, $\phi^1_j$ is equal to $\phi_{(j-1)/2}$ if $j$ is odd, and identically $0$ if $j$ is even.

    Consider the following cochain $\Lambda$ subordinate to $\V$.
    For $(jk) \in N(\V)$, define $\Lambda_{jk} = \Omega_{j/2\,k/2}$ if $j$ and $k$ are even, $\Lambda_{jk} = \Theta_{(j-1)/2}^t \Omega_{(j-1)/2\, (k-1)/2} \Theta_{(k-1)/2}$ if $j$ and $k$ are odd, $\Lambda_{jk} = \Theta_{(j-1)/2}^t \Omega_{(j-1)/2\,k/2}$ if $j$ is odd and $k$ is even, and $\Lambda_{jk} = \Omega_{j/2\,(k-1)/2} \Theta_{(k-1)/2}$ if $j$ is even and $k$ is odd.
    It is clear that $\Lambda$ is an $\epsilon$-approximate cocycle subordinate to $\V$.

    Finally, using \cref{cocycle-to-atlas-homotopy-injection}, if we use the partition of unity $\phi^0$, we see that $\vv'(\Lambda) = \vv'(\Omega)$, and if we use the partition of unity $\phi^1$, we see that $\vv'(\Lambda) = \vv'(\Omega')$.
    The result follows.
\end{proof}

Recall from \cref{cech-cohomology-section} the notion of refinement of a cover.

\begin{constr}
    \label{refinement-construction}
    Let $\nu : \U \to \V$ be a refinement of covers of a topological space $B$ and let $\epsilon \in [0,\infty]$.
    Let $\Omega \in \apprZ{1}{\epsilon}{\V}{O(d)}$.
    Define $\nu(\Omega) \in \apprZ{1}{\epsilon}{\U}{O(d)}$ by letting $\nu(\Omega)_{jk} = \Omega_{\nu(j)\nu(k)}$ for all $(jk) \in N(\U)$.
\end{constr}

\cref{refinement-construction} gives a map $\nu : \apprZ{1}{\epsilon}{\V}{O(d)} \to \apprZ{1}{\epsilon}{\U}{O(d)}$, that descends to a map
\[
    \nu : \apprH{1}{\epsilon}{\V}{O(d)} \to \apprH{1}{\epsilon}{\U}{O(d)}.
\]
It is clear that both these maps are $1$-Lipschitz with respect to $\dZsa$ and with respect to $\disHsa$.

\begin{lem}
    \label{refinement-well-behaved}
    Let $\epsilon \in [0,\infty]$, let $\mu,\nu : \U \to \V$, and let $\Omega \in \apprZ{1}{\epsilon}{\V}{O(d)}$.
    Then, for all $(jk) \in N(\U)$ and $y \in U_k \cap U_j$, we have $\|\mu(\Omega)_{jk}(y) - \nu(\Omega)_{jk}(y)\|< 2\epsilon$, and thus $\disH{\mu(\Omega)}{\nu(\Omega)} \leq 2\epsilon$.
\end{lem}
\begin{proof}
    We address the case $\epsilon \in (0,\infty]$, the case $\epsilon = 0$ being similar.
    We start by defining a $0$-cochain $\Theta$ subordinate to $\U$.
    Given $j \in N(\U)$, let $\Theta_j = \Omega_{\mu(j)\nu(j)}$ if $\mu(j) \neq \nu(j)$, and the identity if $\mu(j) = \nu(j)$.
    Let $(jk) \in N(\U)$ and let $y \in U_k \cap U_j$.
    To simplify notation in the rest of this proof, let us denote $\Omega_{ab}(y)$ by $\Omega_{ab}$.
    We have
    \begin{align*}
        \|\mu(\Omega)_{jk}(y) - (\Theta \cdot \nu(\Omega))_{jk}(y)\| &= \|\Omega_{\mu(j)\mu(k)} - \Omega_{\mu(j)\nu(j)}\Omega_{\nu(j)\nu(k)}\Omega_{\nu(k)\mu(k)}\| \\
            &= \|\Omega_{\nu(j)\mu(j)}\Omega_{\mu(j)\mu(k)} - \Omega_{\nu(j)\nu(k)}\Omega_{\nu(k)\mu(k)}\|\\
            &\leq \|\Omega_{\nu(j)\mu(j)}\Omega_{\mu(j)\mu(k)} - \Omega_{\nu(j)\mu(k)}\| + \|\Omega_{\nu(j)\mu(k)} - \Omega_{\nu(j)\nu(k)}\Omega_{\nu(k)\mu(k)}\| < 2 \epsilon,
    \end{align*}
    where for the second equality we used the fact that $\Omega_{\nu(j)\mu(j)} = \Omega_{\mu(j)\nu(j)}^{-1}$ combined with the fact that the Frobenius norm is invariant under multiplication by an orthogonal matrix, and for the inequalities we used the triangle inequality and the approximate cocycle condition.
\end{proof}

We are now ready to state and prove the main result of this section.

\begin{thm}
    \label{main-thm-class-map}
    Let $B$ be a paracompact topological space and let $\U$ be a countable cover of $B$.
    Let $\epsilon \in [0,\infty]$.
    The map $\vv'$ induces a map
    \[\vv : \aH{\epsilon} \to \hcont{B}{\GR{d}^{\sqrt{2}\epsilon}}\]
    such that, if $\disH{\Omega}{\Lambda} < \delta$ in $\aH{\epsilon}$, then $\vv(\Omega)$ and $\vv(\Lambda)$ become equal in $\hcont{B}{\GR{d}^{\sqrt{2}(\epsilon + \delta)}}$.
    Moreover, if $\mu,\nu : \V \to \U$ are refinements and $\V$ is a countable cover of $B$, then $\vv(\mu(\Omega))$ and $\vv(\nu(\Omega))$ become equal in $\hcont{B}{\GR{d}^{2\sqrt{2}\epsilon}}$.
\end{thm}
\begin{proof}
    The map $\vv$ is well-defined thanks to \cref{cobordant-are-homotopic}.
    For the stability of $\vv$, note that, using \cref{atlas-to-class-map-stable-homotopy}, we see that if $\Omega$ and $\Lambda$ are $\epsilon$-approximate cocycles subordinate to a countable cover $\U = \{U_i\}_{i \in I}$, $\phi$ is a partition of unity subordinate to $\U$, and $\iota : I \to \N_{\geq 1}$ is an injection, then
    \[
        \dC{\av^\phi \circ \atl^{\phi,\iota}\, (\Omega)\,}{\av^\phi \circ \atl^{\phi,\iota}\, (\Lambda)} \leq \sqrt{2} \dZ{\Omega}{\Lambda}.
    \]
    The stability then follows from \cref{close-become-homotopic}.
    Finally, the claim about refinements follows directly from \cref{refinement-well-behaved}.
\end{proof}

Note that the map $\vv$ is independent of any choice of partition of unity or enumeration of the cover $\U$.
We conclude with an interesting remark that is not used in the rest of the paper.

\begin{rmk}
    \label{approximate-cohomology}
    Let $\cov(B)$ be the category whose objects are the countable covers of a paracompact topological space $B$ and whose morphisms are the refinements.
    Let $\epsilon \in [0,\infty]$.
    \cref{refinement-construction} gives a functor $\apprH{1}{\epsilon}{-}{O(d)} : \cov(B)^{\OP} \to \Met$.
    We can then define
    \[
        \apprH{1}{\epsilon}{B}{O(d)} = \colim_{\U \in \cov(B)} \aH{\epsilon},
    \]
    with the caveat that $\apprH{1}{\epsilon}{B}{O(d)}$ may be a \textit{pseudo} metric space.
    \cref{main-thm-class-map} implies that there is a well-defined map $\vv : \apprH{1}{\epsilon}{B}{O(d)} \to \hcont{B}{\GR{d}^{2\sqrt{2}\epsilon}}$, natural in $\epsilon \in [0,\infty]$.
\end{rmk}

\subsection{Relationship to classical vector bundles}

We now relate approximate vector bundles to exact vector bundles, following the intuition that $\epsilon$-approximate vector bundles should correspond to true vector bundles as long as $\epsilon$ is sufficiently small.
For this, we use \cref{main-thm-class-map}.
In the case where an approximate cocycle represents a true vector bundle, we study the problem of constructing an exact cocycle that represents the same vector bundle.
We also give upper and lower bounds for the distance from an approximate cocycle to an exact cocycle representing the same vector bundle (\cref{relation-to-true-vb} and \cref{prop-lower-bound}).

We start by recalling that small thickenings of the Grassmannian embedded in $\R^{\infty \times \infty}$ retract to the Grassmannian.
More precisely, if $\epsilon \leq \sqrt{2}/2$, there is a map $\pi : \GR{d}^\epsilon \to \GR{d}$ which is a homotopy inverse of the inclusion $\GR{d} \subseteq \GR{d}^\epsilon$, by \cref{pi-is-inverse}.
Let $B$ be a topological space.
By postcomposing with $\pi$ we get an inverse for the natural map $\hcont{B}{\GR{d}} \to \hcont{B}{\GR{d}^\epsilon}$ which we denote by
\[
    \pi_* : \hcont{B}{\GR{d}^\epsilon} \to \hcont{B}{\GR{d}}
\]
By an abuse of notation, we also let $\pi_* : \cont{B}{\GR{d}^\epsilon} \to \cont{B}{\GR{d}}$.

Recall that we constructed a map $\vv : \aH{\epsilon} \to \hcont{B}{\GR{d}^{\sqrt{2}\epsilon}}$, so, if $\epsilon \leq 1/2$, then any $\epsilon$-approximate cocycle $\Omega$ represents a true vector bundle, namely $\pi_*(\vv(\Omega)) \in \hcont{B}{\GR{d}}$.
To summarize, if $\epsilon \leq 1/2$, we have defined a map
\[
    \pi_* \circ \vv : \aH{\epsilon} \to \hcont{B}{\GR{d}}
\]

\paragraph{Upper bound.}
For the rest of this section, we let $\Phi$ be an $\epsilon$-approximate local trivialization subordinate to a countable cover $\U = \{U_i\}_{i \in I}$ of a paracompact topological space $B$, with $\epsilon \in [0,\infty]$; we let $\phi$ be a partition of unity subordinate to $\U$ and let $\iota : I \to \N_{\geq 1}$ be an injection.
Recall from \cref{atlas-to-class-map-lem} that
\[\av^\phi(\Phi)(y) = \sum_{i \in I} \phi_i(y) \, \P(\Phi_i(y))\]
defines a $\sqrt{2}\epsilon$-approximate classifying map $B \to \GR{d}^{\sqrt{2}\epsilon}$.
We will make use of results in \cref{polar-decomposition-and-procrustes-section}.

\begin{lem}
    \label{perturbed-frame-lem}
    If $\epsilon \in (0,\infty]$, then for, $i \in I$ and $y \in U_i$, we have 
    \[
        \left\| \Phi_i(y)\Phi_i(y)^t - \pi_*\left(\av^\phi(\Phi)(y)\right) \right\|  < 2\sqrt{2}\epsilon \;\;\text{ and }\;\;
        \left\| \Phi_i(y) - \pi_*\left(\av^\phi(\Phi)(y)\right)\Phi_i(y)\right\|  < 2\sqrt{2}\epsilon.
    \]
    If $\epsilon = 0$, then $\Phi_i(y)\Phi_i(y)^t = \pi_*\left(\av^\phi(\Phi)(y)\right)$ and $\Phi_i(y) = \pi_*\left(\av^\phi(\Phi)(y)\right)\Phi_i(y)$.
\end{lem}
\begin{proof}
    We address the case $\epsilon \in (0,\infty]$, the case $\epsilon = 0$ being similar.
    The second inequality is a consequence of the first one and \cref{product-by-frame-isometry}.
    For the first inequality, use \cref{distance-projection-vertex} together with \cref{projection-is-lipschitz}.
\end{proof}

\begin{lem}
    \label{perturbed-frame-full-rank}
    Assume that $\epsilon \leq \sqrt{2}/4$.
    For every $i \in I$ and $y \in U_i$, the matrix $\pi_*\left(\av^\phi(\Phi)(y)\right)\Phi_i(y)$ has rank $d$.
\end{lem}
\begin{proof}
    To simplify notation, let us omit from the formulas the evaluations on $y \in U_i$.
    It is enough to show that $\left(\pi_*(\av^\phi(\Phi))\Phi_i\right)^t\pi_*(\av^\phi(\Phi))\Phi_i$ has full rank.
    Note that
    \begin{align*}
        \left(\pi_*(\av^\phi(\Phi))\,\Phi_i\right)^t \,\, \pi_*(\av^\phi(\Phi))\,\,\Phi_i &= \Phi_i^t \,\,\pi_*(\av^\phi(\Phi))\,\,\pi_*(\av^\phi(\Phi))\,\,\Phi_i 
        = \Phi_i^t\,\, \pi_*(\av^\phi(\Phi))\,\,\Phi_i.
    \end{align*}
    Using \cref{perturbed-frame-lem} and \cref{product-by-frame-isometry}, we conclude that
    \[
        \left\| \Phi_i^t \,\, \pi_*\left(\av^\phi(\Phi)\right) \,\, \Phi_i - \id\right\|  = \|  \Phi_i^t \,\, \pi_*(\av^\phi(\Phi)) \,\, \Phi_i - \Phi_i^t \,\, \Phi_i \,\, \Phi_i^t \,\, \Phi_i\|  < 2\sqrt{2}\epsilon.
    \]
    The result then follows from \cref{key-lemma}, as $2\sqrt{2}\epsilon \leq 1$, by assumption.
\end{proof}

Given an $\epsilon$-approximate local trivialization $\Phi$ with $\epsilon \leq \sqrt{2}/4$, we now define an exact local trivialization $\Psi$ that represents the same vector bundle.
Given $i \in I$ and $y \in U_i$, we let
\[
    \Psi_i(y) := Q\left(\pi_*\left(\av^\phi(\Phi)(y)\right)\,\,\Phi_i(y)\right),
\]
where the map $Q$ is the one of \cref{Q-is-continuous}.
By \cref{perturbed-frame-full-rank} and \cref{Q-is-continuous} the maps $\Psi_i$ are well-defined and continuous.
To see that it is an exact local trivialization it suffices to check that, if $y \in U_j \cap U_i$, then the columns of $\Psi_i(y)$ and $\Psi_j(y)$ span the same subspace of $\R^\infty$.
This is a consequence of the fact that the columns of $\Psi_i(y)$ span the image of $\pi_*(\av^\phi(\Phi)(y)$.
We also deduce the following.

\begin{lem}
    Let $\phi$ be a partition of unity subordinate to $\U$.
    Then $\av^\phi(\Psi) = \pi_*(\av^\phi(\Phi))$.\qed
\end{lem}

We now bound the distance between $\Psi$ and $\Phi$.

\begin{lem}
    \label{exact-local-trivialization-close-lem}
    Assume that $\epsilon \leq \sqrt{2}/4$, then $\dA{\Phi}{\Psi} \leq 4 \sqrt{2}\epsilon$.
\end{lem}
\begin{proof}
    To simplify notation, let us omit from the formulas the evaluations on $y \in U_i$.
    For every $i \in I$ and $y \in U_i$, we have
    \begin{align*}
    \left\| \Phi_i - Q\left(\pi_*\left(\av^\phi(\Phi)\right)\Phi_i\right)\right\|  \leq &\,\,\, \| \Phi_i - \pi_*(\av^\phi(\Phi))\Phi_i\|  
        + \left\| \pi_*(\av^\phi(\Phi))\Phi_i - Q\left(\pi_*\left(\av^\phi(\Phi)\right)\Phi_i\right)\right\|  \\
        < &\,\,\, 2\sqrt{2}\epsilon + 2\sqrt{2}\epsilon = 4\sqrt{2} \epsilon,
    \end{align*}
    where we bounded the first summand using \cref{perturbed-frame-lem}, and the second summand using \cref{approximate-stiefel-lem}.
    In order to satisfy the hypotheses of \cref{approximate-stiefel-lem}, we use the same argument as in \cref{perturbed-frame-full-rank}.
\end{proof}

\begin{thm}
    \label{relation-to-true-vb}
    Let $\epsilon \leq \sqrt{2}/4$ and let $\Omega \in \aZ{\epsilon}$.
    There exists $\Lambda \in \aZ{}$ such that $\vv(\Lambda) = \pi_*(\vv(\Omega))$ and such that $\dZ{\Omega}{\Lambda} \leq 9\epsilon$.
\end{thm}
\begin{proof}
    Let $\phi$ be a partition of unity subordinate to $\U$.
    Since $\epsilon \leq \sqrt{2}/4 \leq 1/2$, it follows that $\pi_*(\vv(\Omega))$ is well-defined.
    By construction, $\pi_*(\vv(\Omega))$ is the homotopy class of $\pi_* \circ \av^\phi \circ \atl^{\phi,\iota}(\Omega)))$.
    By \cref{exact-local-trivialization-close-lem}, there is an exact local trivialization $\Psi$ such that $\dA{\atl^{\phi,\iota}(\Omega)}{\Psi} \leq 4\sqrt{2} \epsilon$ and such that $\av^\phi(\Psi) = \pi_* \circ \av^\phi \circ \atl^{\phi,\iota}(\Omega) $.
    Let $\Lambda = \bw(\Psi)$.
    Then $\Lambda$ is a witness that $\Psi$ is an exact cocycle, so, by \cref{stability-of-bw}, we have that $\dZ{\Omega}{\Lambda} \leq \epsilon + \sqrt{2} \times 4\sqrt{2} \epsilon = 9 \epsilon$.
    To conclude, note that $\vv(\Lambda) = [\av^\phi \circ \atl^{\phi,\iota} \circ \bw(\Psi)] = [\av^\phi(\Psi)] = [\pi_* \circ \av^\phi \circ \atl^{\phi,\iota} (\Omega)] = \pi_*(\vv(\Omega))$, where in the second equality we used \cref{w-then-triv} to conclude that $\atl^{\phi,\iota}(\bw(\Psi))$ and $\Psi$ are homotopic through $0$-approximate local trivializations.
\end{proof}

\paragraph{Lower bound.}

The following definition and result are inspired by Robinson's notion of consistency radius \cite{R1,R2}.
The idea of this short section is to give a lower bound for the distance from an approximate cocycle to an exact cocycle.

\begin{defn}
    Let $\Omega$ be a cochain subordinate to $\U$. The \define{consistency radius} of $\Omega$, denoted by $r(\Omega)$, is the infimum over all $\epsilon$ such that $\Omega$ belongs to $\aZ{\epsilon}$.
\end{defn}

A similar argument to the one in \cref{altas-to-cocycle-lem} proves the following.

\begin{lem}
    \label{close-to-cocycle}
    Let $\Lambda \in \aZ{\epsilon}$ and $\Omega \in C^1(\U; O(d))$.
    Let $\delta > 0$.
    If $\dZ{\Lambda}{\Omega} < \delta$, then $\Omega \in \aZ{\epsilon + 3\delta}$.
    \qed
\end{lem}

\begin{lem}
    \label{prop-lower-bound}
    Let $\Omega$ be a cochain and let $\epsilon > r(\Omega)$.
    Then, the distance from $\Omega$ to $\aH{\epsilon}$ is bounded below by $\frac{r(\Omega) - \epsilon}{3}$.
    In particular, if $\Lambda$ is an exact cocycle, then $\disH{\Omega}{\Lambda} \geq r(\Omega)/3$.
\end{lem}
\begin{proof}
    Let $\Lambda$ be an exact cocycle and let $\delta > \dZ{\Omega}{\Lambda}$.
    It is enough to show that $\delta \geq r(\Omega)/3$.
    Equivalently, it is enough to show that, for every $(ijk) \in N(\U)$ and $y \in U_k \cap U_j \cap U_i$, we have $\| \Omega_{ij}(y) \Omega_{jk}(y) - \Omega_{ik}(y)\| < 3\delta$.
    This follows from \cref{close-to-cocycle}.
\end{proof}

\section{Discrete approximate vector bundles}
\label{discrete-vector-bundles-section}

In this section, we specialize the notions of approximate cocycle and approximate local trivialization to a certain open cover associated to any simplicial complex.
This gives us the notions of discrete approximate cocycle and of discrete approximate local trivializations.

We show in \cref{simplicial-completeness-thm} that any vector bundle over a compact triangulable space can be represented by a discrete approximate cocycle over a sufficiently fine triangulation of the space.
In \cref{reconstruction-section} we study the problem of reconstructing a vector bundle from finite samples as a discrete approximate local trivialization and as a discrete approximate cocycle.
We prove in \cref{main-reconstruction-thm} that this is possible provided the classifying map of the vector bundle we wish to reconstruct is sufficiently regular, and that we are given a sufficiently dense sample.

\subsection{Discrete approximate vector bundles over simplicial complexes}

We introduce two notions of discrete approximate vector bundle over a simplicial complex.
These notions of discrete approximate vector bundle induce approximate vector bundles over the geometric realization of the simplicial complex.

Since this will be relevant in \cref{section-characteristic-classes}, we define discrete approximate cocycles with values in an arbitrary metric group $G$.
Fix a simplicial complex $K$.

\begin{defn}
    \label{definition:discrete-approximate-vector-bundles}
    A \define{discrete $\epsilon$-approximate cocycle} on $K$ with values in $G$ consists of, for every ordered $1$-simplex $(ij) \in K$, an element $\Omega_{ij} \in G$ such that, for every ordered $2$-simplex $(ijk) \in K$, we have $d_G\left(\Omega_{ij} \Omega_{jk}, \Omega_{ik}\right)  < \epsilon$, and such that $\Omega = \{\Omega_{ij}\}_{(ij) \in K}$ is \define{symmetric}, i.e., we have $\Omega_{ij} = \Omega_{ji}^t$.
    A \define{discrete exact cocycle} consists of the same data, but subject to $\Omega_{ij} \Omega_{jk} = \Omega_{ik}$.
\end{defn}

We denote the set of discrete $\epsilon$-approximate cocycles on a simplicial complex $K$ with values in $G$ by $\dapprZ{1}{\epsilon}{K}{G}$.

\begin{defn}
    Let $i \in K$ be a vertex.
    Let $\st(i)$ be the geometric realization of the open star of $i$ seen as a vertex of the geometric realization $|K|$.
    The \define{star cover} of $|K|$ consists of the family of open sets $\{\st(i)\}_{(i) \in K}$.
    Denote the star cover of $|K|$ by $\st_K$.
\end{defn}

Note that there is a canonical isomorphism of simplicial complexes $N(\st_K) \cong K$ that maps a vertex $i \in K$ to itself.

\begin{constr}
    Let $\Omega$ be a discrete $\epsilon$-approximate cocycle on a simplicial complex $K$ with values in $G$.
    Define, for each $(ij) \in K$, a continuous map $\st(j) \cap \st(i) \to G$ that is constantly $\Omega_{ij}$.
    This defines a natural map $\dapprZ{1}{\epsilon}{K}{G} \to \apprZ{1}{\epsilon}{\st_K}{G}$.
\end{constr}

Note that the map $\dapprZ{1}{\epsilon}{K}{G} \to \apprZ{1}{\epsilon}{\st_K}{G}$ is injective.
With this in mind, we endow $\dapprZ{1}{\epsilon}{K}{G}$ with the metric $\dZsa$, and interpret the map $\dapprZ{1}{\epsilon}{K}{G} \to \apprZ{1}{\epsilon}{\st_K}{G}$ as an embedding of metric spaces.

\begin{defn}
    A \define{discrete $\epsilon$-approximate local trivialization} on a simplicial complex $K$ consists of a frame $\Phi_i \in \ST{d}$ for every $(i) \in K$ such that, for every $(ij) \in K$, there exists $\Omega_{ij} \in O(d)$ such that $\| \Phi_i \Omega_{ij} - \Phi_j\|  < \epsilon$.
    A \define{discrete exact local trivialization} consists of the same data, but subject to $\Phi_i \Omega_{ij} = \Phi_j$.
\end{defn}

Denote the set of discrete $\epsilon$-approximate local trivializations on a simplicial complex $K$ by $\daA{\epsilon}$.
In this discrete case too, the witness $\Omega$ that $\Phi$ is a discrete $\epsilon$-approximate local trivialization is not part of the data of the approximate local trivialization.

\begin{constr}
    Let $\Phi$ be a discrete $\epsilon$-approximate local trivialization on $K$.
    Define, for each $(i) \in K$, a map $\st(i) \to \ST{d}$ that is constantly $\Phi_i$.
    This defines a natural map $\daA{\epsilon} \to \apprA{\epsilon}{\st_K}{d}$.
\end{constr}

\begin{rmk}
    \label{relationship-discrete-appr-vb}
Using \cref{atlas-to-cocycle-construction} we obtain a map $\daA{\epsilon} \to \daZ{3\epsilon}$.
If $K$ is finite, this map is algorithmic since the minimization problem
\[
    \min_{\Omega \in O(d)} \| \Phi_i \Omega - \Phi_j\|
\]
can be solved by using the polar decomposition (\cref{polar-decomposition-and-procrustes-section}).
\end{rmk}

The next result guarantees that any vector bundle on a compact triangulable space can be encoded as a discrete approximate cocycle on a sufficiently fine triangulation of the space.

\begin{prop}
    \label{simplicial-completeness-thm}
    Let $E \to B$ be a vector bundle over a compact triangulable space $B$ and let $\epsilon \leq 3/8$.
    There exists a triangulation $K$ of $B$ and a discrete $\epsilon$-approximate cocycle $\Omega \in \daZ{\epsilon}$ such that $\pi_*(\vv(\Omega))$ represents the vector bundle $E \to B$.
\end{prop}
\begin{proof}
    Let $S$ be a finite simplicial complex such that $|S| \cong B$.
    Without loss of generality, we assume $|S| = B$.
    Since the star cover of $S$ consists of contractible sets, the vector bundle $E \to |S|$ trivializes over $\st_{S}$ and thus is represented by an exact cocycle $\Lambda \in \apprZ{1}{}{\st_{S}}{O(d)}$.
    Moreover, since the closed stars $\overline{\st(i)}$ are also contractible, for each $(ij) \in S$, the continuous map $\Lambda_{ij} : \st(j) \cap \st(i) \to O(d)$ can be taken such that it extends to $\overline{\st(j)} \cap \overline{\st(i)}$, which is a closed set.
    Pick a metric that metrizes $|S|$, which must exists since $S$ is a finite simplicial complex.
    It follows that the maps $\Lambda_{ij}$ are uniformly continuous, and thus there exists $\delta > 0$ such that for every $(ij) \in S$ and every $T \subseteq \st(j) \cap \st(i)$ of diameter less than $\delta$, the diameter of $\Lambda_{ij}(T) \subseteq O(d)$ is less than $\epsilon/3$.

    For $n \in \N$, let $S^{(n)}$ denote the $n$th barycentric subdivision of $S$.
    Note that, for every $n \in \N$, the star cover $\st_{S^{(n+1)}}$ refines the star cover $\st_{S^{(n)}}$.
    Choose a refinement map $r^n : \st_{S^{(n+1)}} \to \st_{S^{(n)}}$ for each $n$ and let $\Lambda^n \in \apprZ{1}{\epsilon}{\st_{S^{(n)}}}{O(d)}$ denote the restriction of $\Lambda$ along these refinements, obtained by using \cref{refinement-construction}.
    Note that, by \cref{refinement-well-behaved}, $\Lambda$ still represents the original vector bundle $E \to |S|$.

    As $n$ goes to $\infty$, the maximum of the diameters $\max_{i \in S^{(n)}}\diam(\st_i)$ goes to $0$, since the diameter of simplices goes to $0$ uniformly, as $S$ has finitely many simplices.
    In particular, there exists $n_0$ such that the diameter of the image of $\Lambda^{n_0}_{ij}$ is less than $\epsilon/3$ for every $(ij) \in S^{(n_0)}$, since the original $\Lambda$ consists of uniformly continuous maps.
    Let $K = S^{n_0}$.

    For every $(ij) \in K$, pick a matrix $\Omega_{ij}$ in the image of $\Lambda^{n_0}_{ij}$ in such a way that $\Omega_{ij} = \Omega_{ji}^t$.
    The matrices $\Omega_{ij}$ assemble into a cochain $\Omega \in C^1(\st_{K}; O(d))$ such that $\dZ{\Omega}{\Lambda^{n_0}} < \epsilon/3$.
    It follows from \cref{close-to-cocycle} that $\Omega$ is a $\epsilon$-approximate cocycle, and since it is constant on each intersection, we have $\Omega \in \dapprZ{1}{\epsilon}{K}{O(d)}$.

    To conclude, note that, by \cref{main-thm-class-map}, we have that $\vv(\Omega)$ and $\vv(\Lambda^{n_0}) = \vv(\Lambda)$ become equal in $\hcont{B}{\GR{d}^{\sqrt{2}(\epsilon + \epsilon/3)}}$.
    Since $\sqrt{2}(\epsilon + \epsilon/3) \leq \sqrt{2}/2$, by assumption, we have that $\pi_*(\vv(\Omega)) = \vv(\Lambda)$, as required.
\end{proof}

\subsection{Reconstruction of vector bundles from finite samples}
\label{reconstruction-section}

Let $\delta \geq 0$ and $\ell > 0$.
A function $f : X \to Y$ between metric spaces is a \define{$\delta$-approximate $\ell$-Lipschitz map} if, for every $x,x' \in X$, we have $\ell d_X(x,x') + \delta \geq d_Y(f(x), f(x'))$.

\begin{defn}
    Let $X \subseteq \R^N$.
    The \define{\v{C}ech complex} of $X$ at distance scale $\epsilon > 0$, denoted $\cech(X)(\epsilon)$, consists of the simplicial complex given by the nerve of the cover $\{ B(x, \epsilon) \}_{x \in X}$ of $X^\epsilon$.
\end{defn}

Note that the cover $\{ B(x, \epsilon)\}_{x \in X}$ is indexed by the elements of $X$, so the $0$-simplices of $\cech(X)(\epsilon)$ consist of the elements of $X$.
As a set, $|\cech(X)(\epsilon)|$ consists of formal linear combinations
\[
    p = \sum_{x \in X} c_x [x]
\]
such that $S_p = \{x \in X : c_x > 0\}$ is a finite set with the property that the intersection $\cap_{x \in S_p} B(x,\epsilon) \subseteq \R^N$ is non-empty.
We will often write $\cech(X)(\epsilon)$ for the geometric realization $|\cech(X)(\epsilon)|$.

\begin{constr}
    \label{construction-of-cechm}
Let $X \subseteq \R^N$ and let $\epsilon > 0$, $\delta \geq 0$, and $\ell > 0$.
Assume that $X$ is finite.
Let $Y \subseteq \R^M$ and let $f : X \to Y$ be a $\delta$-approximate $\ell$-Lipschitz map with respect to the distances induced by the Euclidean norm $\|-\|_2$.
Define a continuous map
\begin{align*}
    \cechm(f)(\epsilon) : \cech(X)(\epsilon) &\to Y^{2\ell \epsilon + \delta} \subseteq \R^M\\
                 \sum_{x \in X} c_x [x] &\mapsto \sum_{x \in X}c_x f(x).\qedhere
\end{align*}
\end{constr}

The map $\cech(f)$ is well-defined since, if $p = \sum_{x \in X} c_x [x]$ is such that $c_{x_0} > 0$, then $\|f(z) - f(x_0)\|_2 < \ell 2\epsilon + \delta$ for every $z \in X$ such that $c_z > 0$, since, in that case, $\|x_0 - z\| < 2\epsilon$, as the balls $B(x_0, \epsilon)$ and $B(z,\epsilon)$ must intersect.

\begin{lem}
    \label{local-trivialization-for-reconstruction}
    Let $X \subseteq \R^N$ and let $\epsilon > 0$, $\delta \geq 0$, and $\ell > 0$.
    Assume that $X$ is finite.
    Let $f : X \to \GR{n,d} \subseteq \R^{n \times n}$ be a $\delta$-approximate $\ell$-Lipschitz map with respect to the Euclidean norm and the Frobenius norm.
    The $(2\ell \epsilon + \delta)$-approximate classifying map $\cechm(f)(\epsilon) : \cech(X)(\epsilon) \to \GR{d,n}^{2\ell \epsilon + \delta}$ of \cref{construction-of-cechm} can be represented by a discrete approximate local trivialization, in the sense that there exists a partition of unity $\phi$ of the star cover of $\cech(X)(\epsilon)$ and $\Phi \in \dapprA{2\ell \epsilon + \delta}{\cech(X)(\epsilon)}{d}$ such that $\av^\phi(\Phi) = \cech(f)(\epsilon)$.

    Moreover, there is a discrete $(3(2\ell\epsilon + \delta))$-approximate cocycle $\bw(\Phi)$
    such that $\vv(\Omega)$ is equal to $\cechm(f)(\epsilon)$ in $\hcont{\cech(X)(\epsilon)}{\GR{d,n}^{3\sqrt{2}(2\ell \epsilon +\delta)}}$.
\end{lem}
\begin{proof}
    The second claim is a consequence of the first one and \cref{w-then-triv}, where $\bw$ is the map defined in \cref{altas-to-cocycle-lem}.

    For the first claim, for each $x \in X$, let $\Phi_x \in \ST{d,n}$ be an orthonormal basis of the subspace of $\R^n$ spanned by $f(x) \in \GR{d,n}$.
    Since $f$ is a $\delta$-approximate $\ell$-Lipschitz map, the family $\{\Phi_x\}_{x \in X}$ constitutes a discrete $(2\ell \epsilon +\delta)$-approximate local trivialization on the simplicial complex $\cech(X)(\epsilon)$, by \cref{lem-equivalence-metrics-grassmannians}.
    By taking the partition of unity $\phi$ subordinate to the star cover of $\cech(X)(\epsilon)$ given by $\phi_x(\sum_{x \in X} c_x [x]) = c_x$, we see that $\av^\phi(\Phi) = \cechm(f)(\epsilon)$.
\end{proof}

The following result is well known, see, e.g., \cite[Corollary~4G.3]{hatcher}.

\begin{lem}[Nerve lemma]
    Let $\U = \{U_i\}_{i \in I}$ be an open cover of a paracompact topological space $B$ and let $\phi$ be a partition of unity subordinate to $B$.
    If $\U$ has the property that any finite intersection of its elements is either contractible or empty, then the map $B \to |N(\U)|$ that sends $y$ to $\sum_{i \in I} \phi_i(y) [i]$ is well-defined, continuous, and a homotopy equivalence.\qed
\end{lem}

\begin{cor}
    \label{nerve-lem-cor}
    Let $X \subseteq \R^N$ be a finite subset, let $\epsilon > 0$, and let $\phi = \{\phi_x\}_{x\in X}$ be a partition of unity subordinate to $\{ B(x,\epsilon) \}_{x \in X}$.
    Then, the following map is a homotopy equivalence:
\begin{align*}
    \pushQED{\qed}
    R^\phi : X^\epsilon &\to \cech(X)(\epsilon) \\
               z &\mapsto \sum_{x \in X} \phi_x(z) [x].\qedhere
    \popQED
\end{align*}
\end{cor}

Our reconstruction theorem for vector bundles builds on the following result by Niyogi, Smale, and Weinberger, which allows one to recover the homotopy type of a compact manifold smoothly embedded into $\R^N$ from a sufficiently close and dense sample.
In the result, $d_H$ denotes the Hausdorff distance.

\begin{prop}[{\cite[Proposition~7.1]{NSW}}]
    \label{NSWprop}
    Let $\M \subseteq \R^N$ be a smoothly embedded compact manifold with $\tau = \reach(M) > 0$.
    Let $P\subseteq \R^N$ such that $d_H(P,\M) < \epsilon < (3 - \sqrt{8}) \tau$ and let
    \[
        \alpha \in \left( \frac{( \epsilon + \tau) - \sqrt{ \epsilon^2 + \tau^2 - 6\tau \epsilon}}{2}, \frac{( \epsilon + \tau) + \sqrt{ \epsilon^2 + \tau^2 - 6\tau \epsilon}}{2} \right),
    \]
    which is a non-empty open interval.
    Then $\M \subseteq P^\alpha$ and the inclusion $\M \to P^\alpha$ is a homotopy equivalence.\qed
\end{prop}

We are now ready to prove the reconstruction theorem for vector bundles.
Before doing so, we give a short remark about representing vector bundles by Lipschitz maps.

\begin{rmk}
    \label{rmk-lipschitz-vb}
    In \cite[Proposition~3.1]{R} it is shown that any rank-$d$ vector bundle on a compact metric space $X$ can be represented by a Lipschitz map $X \to \GR{d,n}$ for some $n$, where $\GR{d,n}$ is seen as a subspace of the space of square matrices $\R^{n \times n}$ with the operator norm.
    It follows that a vector bundle on $X$ can also be represented by a Lipschitz map $X \to \GR{d,n}$ where now we use the Frobenius distance on $\GR{d,n}$, as we do in this paper.
    This motivates the assumptions made in the following result.
\end{rmk}

\begin{thm}
    \label{main-reconstruction-thm}
    Let $\M \subseteq \R^N$ be a smoothly embedded compact manifold and let $f : \M \to \GR{d,n}$ be an $\ell$-Lipschitz map with respect to the Euclidean distance on $\M$ and the Frobenius distance on $\GR{d,n}$. Assume that $\reach(\M) = \tau > 0$.
    Let $P \subseteq \R^N$ be a finite set and let $g : P \to \GR{d,n}$ be a function.
    Let $\epsilon, \delta > 0$ be such that
    \begin{itemize}
        \item[$\rhd$] for every $x \in \M$ there exists $p \in P$ such that $\|p - x\|_2 < \epsilon$;
        \item[$\rhd$] for every $p \in P$ there exists $x \in \M$ such that $\|p - x\|_2 < \epsilon$ and $\|g(p) - f(x)\| < \delta$,
    \end{itemize}
    so that $g$ is a $2(\delta + \ell \epsilon)$-approximate $\ell$-Lipschitz map.
    If $\epsilon < (3 - \sqrt{8}) \tau$, then, for every
    \[
        \alpha \in \left( \frac{( \epsilon + \tau) - \sqrt{ \epsilon^2 + \tau^2 - 6\tau \epsilon}}{2}, \frac{( \epsilon + \tau) + \sqrt{ \epsilon^2 + \tau^2 - 6\tau \epsilon}}{2} \right) \cap \left(0, \frac{\sqrt{2}/2-2\delta - 2\ell \epsilon}{2\ell}\right),
    \]
    there is a homotopy commutative diagram as follows, in which the vertical maps are homotopy equivalences:
\[
    \begin{tikzpicture}
      \matrix (m) [matrix of math nodes,row sep=1.5em,column sep=4em,minimum width=2em,nodes={text height=1.75ex,text depth=0.25ex}]
        { \M & \GR{n,d} \\
          \cech(P)(\alpha) & \GR{n,d}^{2(\ell \alpha + \ell \epsilon + \delta)}.\\};
        \path[line width=0.75pt, -{>[width=8pt]}]
        (m-1-1) edge node [above] {$f$} (m-1-2)
                edge node [left] {} (m-2-1)
        (m-2-1) edge node [above] {$\cech(g)$} (m-2-2)
        (m-1-2) edge [right hook->] node [right] {} (m-2-2)
        ;
    \end{tikzpicture}
    \]
\end{thm}

Note that the interval to which $\alpha$ must belong to is non-empty as long as $\epsilon$ and $\delta$ are sufficiently small, a condition that depends only on $\tau$ and $\ell$.

\begin{proof}
    We start by showing that $g$ is indeed a $2(\delta + \ell \epsilon)$-approximate $\ell$-Lipschitz map, so that the bottom map of the diagram in the statement is well-defined.
    To do this, note that if $p,q \in P$, then there exist $x,y \in \M$ with $\|p - x\|_2, \|q - y\|_2 < \epsilon$ and $\|g(p) - f(x)\|, \|g(q) - f(y)\| < \delta$.
    Since $f$ is $\ell$-Lipschitz, we have $\|g(p) - g(q)\| \leq 2\delta + \ell \|x - y\| \leq 2\delta + \ell 2\epsilon + \ell \|p - q\|$, as required.

    The homotopy equivalence $\M \to \cech(P)(\alpha)$ is given by composing the inclusion $\M \subseteq P^\alpha$ with the map $R^\phi: P^\alpha \to \cech(P)(\alpha)$ for a choice of partition of unity $\phi$ subordinate to $\{B(p,\alpha)\}_{p\in P}$. The map $R^\phi$ is a homotopy equivalence by \cref{nerve-lem-cor}.
    The fact that the inclusion $\M \to P^\alpha$ is well-defined and a homotopy equivalence is the content of \cref{NSWprop}, whose hypotheses are satisfied since our conditions imply that $d_H(P,\M) < \epsilon$.

    The map $\GR{n,d} \to \GR{n,d}^{2(\ell \alpha + \ell \epsilon + \delta)}$ is simply the inclusion, and it is a homotopy equivalence by \cref{pi-is-inverse}, since $\alpha < (\sqrt{2}/2-2\delta - 2\ell\epsilon)/(2\ell)$.

    To conclude the proof, we must show that the diagram in the statement commutes up to homotopy.
    For this, let $z \in \M$.
    We have $\cechm(g)(R^\phi(z)) = \cechm(g)\left( \sum_{p \in P} \phi_p(z) [p] \right) = \sum_{p \in P} \phi_p(z) g(x)$.
    Consider the linear path $\beta f(z) + (1-\beta) \sum_{p \in P} \phi_p(z) g(p)$ for $\beta\in [0,1]$.
    It suffices to show that it is included in $\GR{n,d}^{2\ell \alpha + 2\ell \epsilon + 2\delta}$ for all $\beta\in [0,1]$, and we will show that it is at distance less than $\ell \alpha + \ell \epsilon + \delta$ from $f(z)$.
    We compute
    \begin{align*}
        \left\| f(z) - \left(\beta f(z) + (1-\beta) \sum_{p \in P} \phi_p(z) g(x)\right)\right\|_2 &= \left\|\beta f(z) + (1-\beta) f(x) - \left(\beta f(z) + (1-\beta) \sum_{p \in P} \phi_p(z) g(p)\right)\right\|_2\\
          &\leq  (1-\beta) \sum_{p\in P} \phi_p(z) \,\, \| f(z) - g(p) \|_2\\
          &< (1-\beta) \sum_{x\in X} \phi_x(z) \,\,(\ell \alpha + \ell \epsilon + \delta) \leq \ell \alpha + \ell \epsilon + \delta,
    \end{align*}
    where the strict inequality comes from the fact that, if $\phi_p(z)$ is non-zero, then $z \in B(p,\alpha)$, and thus there exists $x \in \M$ such that $\|f(z) - g(p)\|_2 \leq \|f(z) - f(x) \| +\|f(x) - g(p)\| < \ell \|z-x\|_2 + \delta \leq\ell (\|z - p\|_2 +  \|p - x\|_2) + \delta < \ell \alpha + \ell \epsilon + \delta$.
\end{proof}

We conclude this section with a few remarks.

\begin{rmk}
    For simplicity, we have proven \cref{main-reconstruction-thm} using the \v{C}ech complex.
    A similar result can be obtained for the Vietoris--Rips complex, using, for instance, the results of \cite{ALS}.
\end{rmk}

\begin{rmk}
    \label{reconstruction-as-cocycle}
As proven in \cref{local-trivialization-for-reconstruction}, the map reconstructed in \cref{main-reconstruction-thm} has a combinatorial description using a discrete approximate local trivialization that in turn induces a discrete approximate cocycle.
This discrete approximate cocycle can be used to compute characteristic classes combinatorially, which is the subject of the next section.
\end{rmk}

\section{Effective computation of characteristic classes}
\label{section-characteristic-classes}

In this section, we present three algorithms to compute, respectively, the first two Stiefel--Whitney classes of an approximate vector bundle given by an approximate $O(d)$-cocycle, and the Euler class of an oriented approximate vector bundle of rank $2$ given by an approximate $SO(2)$-cocycle.
The algorithms are based on well known results which say that the characteristic classes we consider are obstructions to lifting the structure group of the cocycle to certain other Lie groups.
The difficulty is in showing that these algorithms can be extended in a stable and consistent way to $\epsilon$-approximate cocycles, provided $\epsilon$ is sufficiently small.
Throughout the section, we will make use of basic Riemannian geometry; a reference for this topic is \cite{jL}.

In \cref{change-coeff-section} we recall two standard constructions used to change the coefficient group of a \v{C}ech cocycle, and we extend them to approximate cocycles.
In \cref{sw1-section} we give the algorithm for the first Stiefel--Whitney class, in \cref{eu-section} we give the algorithm for the Euler class, and in \cref{sw2-section} we give the algorithm for the second Stiefel--Whitney class.

\subsection{Change of coefficients}
\label{change-coeff-section}

In this section, we will make use of basic \v{C}ech cohomology with coefficients in a sheaf of Abelian groups.
We recall the essential components now; for an introduction to the subject, see for example \cite[Chapter~5]{warner}.

Let $\U$ be a cover of a topological space $B$.
For $A$ an Abelian group and $n \in \N$, we let $\apprH{n}{}{\U}{A}$ denote the $n$th \v{C}ech cohomology group of $(B,\U)$ with coefficients in the sheaf of locally constant functions with values in $A$ (\cite[p.~201]{warner}).
This cohomology group is a quotient of the subgroup of cocycles of $C^n(\U;A)$, which is the Abelian group of locally constant functions defined on all $(n+1)$-fold intersections $U_{i_n} \cap \cdots \cap U_{i_0} \to A$.
As usual, the \v{C}ech cohomology of $B$ with values in $A$ is defined as $\apprH{n}{}{B}{A} = \colim_{\U \text{ cover}} \apprH{n}{}{\U}{A}$.
It is well known that, when $B$ is paracompact and locally contractible, there is a natural isomorphism $\apprH{n}{}{B}{A} \cong H^n(B;A)$, where the right hand side denotes singular cohomology with coefficients in $A$.

Since we use these constructions only for $n = 1,2$, we elaborate on these two cases.
For $n=1$, the \v{C}ech cohomology is precisely the one we introduced in \cref{cech-cohomology-section}, where the group $A$ is endowed with the discrete topology.
For $n=2$, the $2$-cocycles $\apprZ{2}{}{\U}{A}$ consist of families of locally constant functions $\{\Gamma_{ijk} : U_k \cap U_j \cap U_i \to A\}_{(ijk) \in N(\U)}$ such that, for every $(ijkl) \in N(\U)$ and every $y \in U_l \cap U_k \cap U_j \cap U_i$ we have
\[
    \Gamma_{ijk}(y) \Gamma_{ijl}(y)^{-1} \Gamma_{ikl}(y) \Gamma_{jkl}(y)^{-1} = 1_A,
\]
where we are writing the operations of the Abelian group $A$ multiplicatively.
The operation on $\apprZ{2}{}{\U}{A}$ is pointwise multiplication.
The cohomology group $\apprH{2}{}{\U}{A}$ is the quotient of $\apprZ{2}{}{\U}{A}$ by the subgroup of $2$-cocycles of the form $y \mapsto \Omega_{ij}(y) \Omega_{jk}(y) \Omega_{ki}(y)$, for $\Omega \in C^1(\U;A)$.

Let $G$ and $H$ be topological groups and let $\zeta : G \to H$ be a continuous group morphism.
The map $\zeta$ induces a map $\apprZ{1}{}{\U}{G} \to \apprZ{1}{}{\U}{H}$, simply by applying $\zeta$ pointwise to a cocycle.
This map induces a well-defined map $\apprH{1}{}{\U}{G} \to \apprH{1}{}{\U}{H}$.
A bit more interestingly, given a central extension of groups $1 \to F \to G \to H \to 1$, there is a well-defined so-called \textit{connecting morphism} $\apprH{1}{}{\U}{H} \to \apprH{2}{}{\U}{F}$.
For a short introduction to these concepts, see \cite[Appendix~A]{spingeometry}.

In this section we generalize these two constructions to the case of approximate cocycles, when $F$, $G$, and $H$ are well-behaved metric groups
We start by generalizing the first construction.

\begin{constr}
    \label{change-of-coeff-construction}
    Let $\zeta : G \to H$ be a continuous group morphism between topological groups.
    Given $\Omega \in C^1(\U ; G)$, define a cochain $\zeta(\Omega)$ with values in $H$ by $\zeta(\Omega)_{(ij)}(y) = \zeta(\Omega_{(ij)}(y))$ for every $(ij) \in N(\U)$ and $y \in U_j \cap U_i$.
\end{constr}

\begin{lem}
    \label{simple-change-of-coefficients}
    Let $G$ and $H$ be metric groups and let $\zeta : G \to H$ be an $\ell$-Lipschitz group morphism.
    Let $\epsilon \in [0,\infty]$.
    \cref{change-of-coeff-construction} induces maps $\zeta : \apprZ{1}{\epsilon}{\U}{G} \to \apprZ{1}{\ell\epsilon}{\U}{H}$ and $\zeta : \apprH{1}{\epsilon}{\U}{G} \to \apprH{1}{\ell\epsilon}{\U}{H}$, and these maps are $\ell$-Lipschitz with respect to the distances $\dZsa$ and $\disHsa$ respectively.

    In particular, if the infimum over all distances between distinct elements of $H$ is bounded below by $\delta$ and $\ell\epsilon \leq \delta$, then $\zeta$ induces a map $\zeta : \apprH{1}{\epsilon}{\U}{G} \to \apprH{1}{}{\U}{H}$, such that, if $\Omega,\Omega' \in \apprH{1}{\epsilon}{\U}{G}$ satisfy $\dZ{\Omega}{\Omega'} < \delta/\ell$, then $\zeta(\Omega) = \zeta(\Omega') \in \apprH{1}{}{\U}{H}$.
\end{lem}
\begin{proof}
    We start by checking that, if $\Omega \in \apprZ{1}{\epsilon}{\U}{G}$, then $\zeta(\Omega)$ is an $\ell\epsilon$-approximate cocycle.
    For this, let $(ijk) \in N(\U)$ and let $y \in U_k \cap U_j \cap U_i$.
    We have $d_H\left(\zeta(\Omega)_{ij}(y) \zeta(\Omega)_{jk}(y), \zeta(\Omega)_{ik}(y)\right) = d_H\left(\zeta( \Omega_{ij}(y) \Omega_{jk}(y) ), \zeta(\Omega_{ik}(y) )\right) \leq \ell d_G\left(\Omega_{ij}(y) \Omega_{jk}(y), \Omega_{ik}(y)\right) < \ell \epsilon$, using the fact that $\zeta$ is an $\ell$-Lipschitz group morphism.
    The fact that the maps are $\ell$-Lipschitz is clear.

    To see that $\zeta$ descends to approximate cohomology, note that, if $\Theta \in C^0(\U; G)$, then we can define $\zeta(\Theta) \in C^0(\U; H)$ by $\zeta(\Theta)_i(y) = \zeta(\Theta_i(y))$ for all $i \in N(\U)$ and $y \in U_i$, and that, with this definition, we have $\zeta(\Theta) \cdot \zeta(\Omega) = \zeta(\Theta \cdot \Omega)$, for every $\Omega \in \apprZ{1}{\epsilon}{\U}{G}$.

    The second claim is a consequence of the first one.
\end{proof}

We now generalize the connecting morphism construction.

\begin{constr}
    \label{connecting-morphism-construction}
    Let $\zeta : G \to H$ be a continuous and surjective group morphism between metric groups and let $\epsilon \geq 0$.
    Let $F$ be the kernel of $\zeta$ and assume that $F$ is locally compact and discrete in $G$, so that, in particular $G \to H$ is a covering map.
    Suppose that $\U$ is a cover of a topological space with the property that non-empty binary intersections of elements of $\U$ are locally path connected and simply connected.
    Let $\Omega \in C^1(\U ; H)$.
    Given $(ij) \in N(\U)$, choose a continuous lift $\Lambda_{ij} : U_j \cap U_i \to G$ of $\Omega_{ij} : U_j \cap U_i \to H$, such that $\Lambda_{ij} = (\Lambda_{ji})^{-1}$.
    Finally, for $(ijk) \in N(\U)$ and $y \in U_k \cap U_j \cap U_i$, let $\Gamma_{ijk}(y) \in F$ be a closest point of $F$ to $\Lambda_{ij}(y)\Lambda_{jk}(y)\Lambda_{ki}(y)$.
\end{constr}

A priori, the maps $\Gamma_{ijk}$ of \cref{connecting-morphism-construction} do not necessarily assemble into a $2$-cocycle with values in $F$, since the maps may not even be continuous.
We now give conditions under which the maps $\Gamma_{ijk}$ constitute a $2$-cocycle.
In order to do this, we need to introduce some definitions.

\begin{defn}
    \label{systole-def}
    Let $M$ be a metric space.
    The \define{systole} of $M$, denoted $\sysp(M)$, is the infimum of the lengths of all non-nullhomotopic loops of $M$.
\end{defn}

Recall from, e.g., \cite[Chapter~2]{jL}, that any connected Riemannian manifold $M$ can be endowed with the \define{geodesic distance} where the distance between two points is taken to be the infimum of the lengths of all piecewise regular (i.e., with non-zero velocity) paths between the two points.
If the manifold is not connected, then the same construction gives an \define{extended distance}, meaning a distance that can also take the value $\infty$.
Whenever we endow a Riemannian manifold with the geodesic distance, we are referring to this extended distance.
Finally, recall that, if the manifold is complete (as are all the examples we consider here), then the geodesic distance between two points can be calculated as the infimum of the length of all geodesics between the two points (\cite[Corollary 6.21]{jL}).

Before being able to give conditions under which the connecting morphism is well behaved, we need to prove the following technical lemma.

\begin{lem}
    \label{projections-and-multiplications}
    Let $F \subseteq G$ be an isometric inclusion of a discrete subgroup into a metric group.
    Let $2r$ be the infimum of the distances between distinct elements of $F$.
    Let $\langle - \rangle : F^r \to F$ denote the projection to the closest element of $F$, which is well-defined and continuous.
    \begin{enumerate}
        \item If $a \in F^s$, $b \in F^t$, $c \in F^u$ with $s + t + u \leq r$ then $\langle a \rangle \langle b \rangle \langle c \rangle = \langle a b c \rangle$.
        \item If $a \in F^r$ and $b \in F$, then $\langle a b\rangle = \langle a \rangle b$.
        \item If $a,c \in G$, $b \in F^{s}$, and $a \langle b \rangle c \in F^{t}$ with $s + t \leq r$, then $\langle a \langle b \rangle c\rangle = \langle a b c \rangle$.
    \end{enumerate}
\end{lem}
\begin{proof}
    As we will use this for each claim, start by noting that, if $g \in G$, $h \in F$ and $d_G(g,h) < r$, then $\langle g \rangle = h$.
    Since multiplication is an isometry, we have $d_G(abc,\langle a \rangle b c) < s$.
    Using this idea two more times, and the triangle inequality, we deduce $d_G(abc, \langle a \rangle \langle b \rangle \langle c \rangle) < s + t + u \leq r$, and the first claim follows.
    For the second claim, note that $d_G(a b, \langle a \rangle b) < r$.
    Since $\langle a \rangle b \in F$, the second claim follows.
    Finally, $d_G(\langle a \langle b \rangle c \rangle, a \langle b \rangle c) < s$ and $d_G(a \langle b \rangle c, a b c) < t$, so $d_G(\langle a \langle b \rangle c \rangle, abc) < s + t \leq r$, and the third claim follows.
\end{proof}

\begin{thm}
    \label{change-of-coefficients}
    Let $1 \to F \to G \to H \to 1$ be a central extension of Lie groups with $H$ compact and $F$ discrete.
    Fix a bi-invariant Riemannian metric on $H$ and use it to metrize $H$ and $G$ with the geodesic distance and $F$ with the distance inherited from $G$.
    Let $\U$ be a cover of a topological space $B$ such that each set and each non-empty binary intersection is locally path connected and simply connected.
    Let $\epsilon \leq \sysp(H)/8$.
    Then, \cref{connecting-morphism-construction} induces maps $\partial : \apprZ{1}{\epsilon}{\U}{H} \to \apprZ{2}{}{\U}{F}$ and $\partial : \apprH{1}{\epsilon}{\U}{H} \to \apprH{2}{}{\U}{F}$.

    The second map is independent of any choice of lift, and is stable in the sense that if
    $\Omega,\Omega' \in \apprH{1}{\epsilon}{\U}{H}$
    are such that $\disH{\Omega}{\Omega'} < \sysp(H)/8$, then $\partial(\Omega) = \partial(\Omega') \in \apprH{2}{}{\U}{F}$.
\end{thm}
\begin{proof}
    Let $2r$ denote the infimum over the distances between distinct elements of $F$.
    By \cref{reach-and-systole}, we have $\sysp(H)/2 \leq r$ and thus $\epsilon \leq r/4$.
    For $x \in F^r$, let $\langle x \rangle \in F$ denote its closest point in $F$, which is well-defined and continuous.

    We start by showing that, for $(ijk) \in N(\U)$, the map $\Gamma_{ijk} : U_k \cap U_j \cap U_i \to F$ is continuous.
    It suffices to show that $\Lambda_{ij}(y) \Lambda_{jk}(y) \Lambda_{ki}(y) \in F^r$, as $\Gamma_{ijk}(y) = \langle \Lambda_{ij}(y) \Lambda_{jk}(y) \Lambda_{ki}(y) \rangle$.
    By definition of the metrics on $G$ and $H$, the distance from $\Lambda_{ij}(y) \Lambda_{jk}(y) \Lambda_{ki}(y)$ to $F$ is equal to the distance from $\Omega_{ij}(y) \Omega_{jk}(y) \Omega_{ki}(y)$ to the identity of $H$, which is less than $\epsilon \leq r/4$, by assumption.

    We now show that $\Gamma$ is an $H$-valued $2$-cocycle.
    Fix $y \in U_l \cap U_k \cap U_j \cap U_i$ and let us write $(ij)$ for $\Lambda_{ij}(y)$, and likewise for the other elements of $G$.
    We must show that
    \[
        \langle (ij) (jk) (ki) \rangle\; \langle (ij) (jl) (li) \rangle^{-1} \langle (ik) (kl) (li) \rangle\; \langle (jk) (kl) (lj) \rangle^{-1} = 1_F.
    \]
    Note that, although $F$ is Abelian, we are using multiplicative notation to be consistent with the fact that $G$ need not be Abelian.
    We now compute:
    \begin{align*}
        &\langle (ij) (jk) (ki) \rangle\; \langle (ij) (jl) (li) \rangle^{-1} \langle (ik) (kl) (li) \rangle\; \langle (jk) (kl) (lj) \rangle^{-1} && \\
        &\;\;\;\;\;\;= \langle (ij) (jk) (ki) \rangle\; \langle (il) (lj) (ji) \rangle\; \langle (ik) (kl) (li) \rangle\; \langle (jl) (lk) (kj) \rangle && \text{(taking inverse is an isometry)}\\
        &\;\;\;\;\;\;= \langle (ij) (jk) (ki) \rangle\; \langle (ik) (kl) (li) \rangle\; \langle (il) (lj) (ji) \rangle\;  \langle (jl) (lk) (kj) \rangle && \text{($F$ is central)}\\
        &\;\;\;\;\;\;= \langle (ij) (jk) (ki) (ik) (kl) (li) (il) (lj) (ji) \rangle\;  \langle (jl) (lk) (kj) \rangle && \text{(\cref{projections-and-multiplications}(1), with $s,t,u = \epsilon \leq r/4$)} \\
        &\;\;\;\;\;\;= \langle (ij) (jk) (kl) (lj) (ji) \rangle\;  \langle (jl) (lk) (kj) \rangle && \text{(cancellations)}\\
        &\;\;\;\;\;\;= \left\langle (ij) (jk) (kl) (lj) (ji)\; \langle (jl) (lk) (kj) \rangle\; \right\rangle && \text{(\cref{projections-and-multiplications}(2), with $s = 3\epsilon \leq 3/4 r$)} \\
        &\;\;\;\;\;\;= \left\langle (ij) (jk) (kl) (lj) \; \langle (jl) (lk) (kj) \rangle \;(ji) \right\rangle && \text{($F$ is central)}\\
        &\;\;\;\;\;\;= \langle (ij) (jk) (kl) (lj) (jl) (lk) (kj) (ji) \rangle && \text{(\cref{projections-and-multiplications}(3), with $s = \epsilon$ and $t = 3\epsilon$)} \\
         &\;\;\;\;\;\;= 1_F &&  \text{(cancellations)}.
    \end{align*}

    We now prove that the composite $\apprZ{1}{\epsilon}{\U}{H} \to \apprZ{2}{}{\U}{F} \to \apprH{2}{}{\U}{F}$ is independent of the choice of lifts $\Lambda$.
    In order to see this note that, if $\Lambda_{ij}, \Lambda'_{ij} : U_j \cap U_i \to G$ are two lifts of $\Omega_{ij} : U_j \cap U_i \to H$, then they differ by a function $f_{ij} : U_j \cap U_i \to F$.
    Since $F$ is central, it follows that $\Gamma_{ijk}$ and $\Gamma'_{ijk}$ differ by $f_{ij} f_{jk} f_{ki}$, using \cref{projections-and-multiplications}(2).

    Now, let $\Theta \in C^0(\U; H)$ and let $\Pi \in C^0(\U; G)$ be a lift of it, which exists since the elements of $\U$ are simply connected.
    Let us write $(i)$ for $\Pi_i(y)$.
    A lift for $\Theta \cdot \Omega$ is given by $\Pi \cdot \Lambda$, and the $2$-cochain in this case is
    \begin{align*}
        \langle (i) (ij) (j)^{-1} (j) (jk) (k)^{-1} (k) (ki) (i)^{-1} \rangle
        &= \langle (i) (ij) (jk) (ki) (i)^{-1} \rangle
        = \left\langle (i)\; \langle (ij) (jk) (ki) \rangle \; (i)^{-1} \right\rangle\\
        &= \left\langle (i) (i)^{-1} \; \langle (ij) (jk) (ki) \rangle \; \right\rangle = \Gamma_{ijk}(y),
    \end{align*}
    where in the second equality we used \cref{projections-and-multiplications}(3) and in the third one we used that $F$ is central.
    This shows that the map $\partial : \apprH{1}{\epsilon}{\U}{H} \to \apprH{2}{}{\U}{F}$ is well-defined.

    We conclude the proof by proving the stability claim.
    Let $[\Omega],[\Omega'] \in \apprH{1}{\epsilon}{\U}{H}$ be such that $\disH{[\Omega]}{[\Omega']} < \sysp(H)/8$,
    and choose representatives $\Omega,\Omega' \in \apprZ{1}{\epsilon}{\U}{H}$ such that $\dZ{\Omega}{\Omega'} < \sysp(H)/8$.
    By \cref{metric-lift}(2), we can pick lifts $\Lambda, \Lambda' \in C^1(\U; G)$ of $\Omega$ and $\Omega'$ respectively such that $\dZ{\Lambda}{\Lambda'} < \sysp(H)/8$.
    Now, let $y \in U_k \cap U_j \cap U_i$, and let us write $(ij)'$ for $\Lambda_{ij}'(y)$.
    We have
    \begin{align*}
        d_G\left((ij)' (jk)' (ki)', \langle (ij) (jk) (ki)\rangle\right) 
        &\leq d_G\left((ij)' (jk)' (ki)', (ij) (jk) (ki)\right) +
    d_G\left((ij) (jk) (ki), \langle (ij) (jk) (ki)\rangle \right)\\
        &< 3 \sysp(H)/8 + \epsilon = \sysp(H)/2,
    \end{align*}
    so $\langle \Lambda'_{ij}(y) \Lambda'_{jk}(y) \Lambda'_{ki}(y) \rangle = \langle \Lambda_{ij}(y) \Lambda_{jk}(y) \Lambda_{ki}(y)\rangle$, and the result follows.
\end{proof}

\subsection{First Stiefel--Whitney class}
\label{sw1-section}


Consider the map $\det : O(d) \to \{\pm 1\} \cong \Z/2$, and metrize $O(d)$ using the Frobenius distance and $\{\pm 1\}$ by $d(1,-1) = 2$.
It follows from \cref{real-det-is-lip} that $\det$ is $1$-Lipschitz.


\paragraph{Algorithm.}
Let $\U$ be a cover of a topological space.
Let $\epsilon \leq 2$.
By applying $\det$ to a representative cocycle, we get a map $\sw_1 : \apprH{1}{\epsilon}{\U}{O(d)} \to \apprH{1}{}{\U}{\Z/2}$ such that, if $\Omega,\Omega' \in \apprH{1}{\epsilon}{\U}{O(d)}$ satisfy $\dZ{\Omega}{\Omega'} < 2$, then $\sw_1(\Omega) = \sw_1(\Omega') \in \apprH{1}{}{\U}{\Z/2}$.
This follows directly from \cref{simple-change-of-coefficients}.

In particular, if $K$ is a simplicial complex, we get an algorithm $\sw_1 : \dapprH{1}{\epsilon}{K}{O(d)} \to H^1(K;\Z/2)$, that is polynomial in the number of vertices of $K$.
From \cref{simple-change-of-coefficients} it follows that the algorithm is stable, in the sense that if $\Omega,\Omega' \in \dapprH{1}{\epsilon}{K}{O(d)}$ satisfy $\dZ{\Omega}{\Omega'} < 2$, then $\sw_1(\Omega) = \sw_1(\Omega') \in H^1(K;\Z/2)$.

\paragraph{Oriented approximate vector bundles.}
The following observation is relevant for the computation of Euler classes, in the next section.
Let $\U$ be a cover of a topological space, with the property that non-empty binary intersections are locally path connected and simply connected.
If $\epsilon \leq 2$ and $\Omega \in \apprH{1}{\epsilon}{\U}{O(d)}$ is such that $\sw_1(\Omega) = 0$, then $\Omega = [\Lambda]$ for some $\Lambda \in \apprZ{1}{\epsilon}{\U}{SO(d)}$.
To prove this, let $[\Gamma] = \sw_1(\Omega)$, let $\Theta \in C^0(\U;\Z/2)$ be such that $\Theta \cdot \Gamma = \id$, and lift $\Theta$ to $\Pi \in C^0(\U;O(d))$.
Then $\Pi \cdot \Omega \in \apprZ{1}{\epsilon}{\U}{SO(d)}$.

\begin{table}
    {\fbox{\parbox{\textwidth}{%
    \begin{enumerate}
        \item[]\textbf{Input:} A simplicial complex $K$ with vertices $\{1, \dots, k\}$ and a discrete approximate cocycle $\Omega \in \dapprZ{1}{\epsilon}{K}{O(d)}$ with $\epsilon \leq 2$.
        \item[]\textbf{Output:} A simplicial cocycle $\sw_1(\Omega) \in Z^1(K;\Z/2)$.
        \item[] For each $(ij)$ $1$-simplex of $K$ with $i < j$ let $\sw_1(\Omega)_{ij} := \det(\Omega_{ji})$.
    \end{enumerate}}
    }\caption{Pseudocode for the algorithm $\sw_1$.} \label[algorithm]{pseudocode-sw1}}
    \vspace{-6mm}
\end{table}

\paragraph{Consistency of the algorithm.}
The following well known result can be found in, e.g., \cite[Example~A.3]{spingeometry}.
\begin{lem}
    \label{consistency0-sw1}
    Let $\U$ be a cover of a locally contractible topological space.
    If $\epsilon = 0$, the construction $\sw_1 : \apprH{1}{\epsilon}{\U}{O(d)} \to \apprH{1}{}{\U}{\Z/2}$ computes the first Stiefel--Whitney class of the vector bundle represented by the cocycle.\qed
\end{lem}

\begin{prop}
    \label{consistency-sw1}
    Let $\U$ be a countable cover of a locally contractible, paracompact space $B$.
    Let $\epsilon < 2/9$ and let $\Omega \in \apprH{1}{\epsilon}{\U}{O(d)}$.
    Then $\sw_1(\Omega) \in \apprH{1}{}{\U}{\Z/2}$ coincides with the first Stiefel--Whitney class of the vector bundle classified by $\pi_*(\vv(\Omega)) : B \to \GR{d}$.
\end{prop}
\begin{proof}
    Since $2/9 \leq \sqrt{2}/4$, there exists, by \cref{relation-to-true-vb}, an exact cocycle $\Lambda \in \apprH{1}{}{\U}{O(d)}$ such that $\disH{\Omega}{\Lambda} < 2/9 \times 9 = 2$, and such that $\pi_*(\vv(\Omega)) = \vv(\Lambda)$.
    From the stability of $\sw_1$, it follows that $\sw_1(\Omega) = \sw_1(\Lambda)$.
    Finally, by \cref{consistency0-sw1}, $\sw_1(\Lambda)$ is the first Stiefel--Whitney class of the vector bundle classified by $\vv(\Lambda)$, so the result follows.
\end{proof}

\subsection{Euler class of oriented, rank-$2$ vector bundles}
\label{eu-section}

Consider the group $SO(2)$ of orthogonal $2$-by-$2$ matrices with positive determinant.
There is a short exact sequence of Lie groups $1 \to \Z \to \R \to SO(2) \to 1$ given by mapping a real number $r$ to the rotation matrix
\[
    \begin{pmatrix}
        \cos(2\pi r) & -\sin(2\pi r) \\
        \sin(2\pi r) & \cos(2\pi r)
    \end{pmatrix}.
\]

\paragraph{Algorithm.}
Let $\U$ be a cover of a topological space, with the property that non-empty binary intersections are locally path connected and simply connected.
As usual, we endow $SO(2)$ with the Frobenius distance, and we let $SO(2)_g$ denote the same group, but endowed with the geodesic distance, with Riemannian structure inherited from the inclusion $SO(2) \subseteq \R^{2 \times 2} = \R^4$.

By \cref{systole-of-O}, we have $\sysp(SO(2)) = 2\sqrt{2} \pi$, so, applying \cref{change-of-coefficients} to the extension $1 \to \Z \to \R \to SO(2)_g \to 1$, we see that \cref{connecting-morphism-construction} gives a map $\eu : \apprH{1}{\epsilon}{\U}{SO(2)_g} \to \apprH{2}{}{\U}{\Z}$, as long as $\epsilon \leq \sqrt{2}\pi/4$.
In order to give an algorithm using Frobenius distances, we note that, if $\epsilon \leq 1$, then, by \cref{comparison-geodesic-and-frobenius}, any $\epsilon$-approximate $SO(2)$-cocycle is an $(\sqrt{2}\pi/4)$-approximate $SO(2)_g$-cocycle.
So \cref{connecting-morphism-construction} gives a map $\eu : \apprH{1}{\epsilon}{\U}{SO(2)} \to \apprH{2}{}{\U}{\Z}$, as long as $\epsilon \leq 1$.
By the stability statement of \cref{change-of-coefficients}, we see that, if $\Omega,\Omega' \in \apprH{1}{\epsilon}{\U}{SO(2)}$ are such that $\disH{\Omega}{\Omega'} < 1$, then $\eu(\Omega) = \eu(\Omega')$.

In particular, if $K$ is a simplicial complex, we get a $1$-stable algorithm $\eu : \dapprH{1}{\epsilon}{K}{SO(2)} \to H^2(K;\Z)$.


\begin{table}[h]
    {\fbox{\parbox{\textwidth}{%
    \begin{enumerate}
        \item[]\textbf{Input:} A simplicial complex $K$ with vertices $\{1, \dots, k\}$ and a discrete approximate cocycle $\Omega \in \dapprZ{1}{\epsilon}{K}{SO(2)}$ with $\epsilon \leq 1$.
        \item[]\textbf{Output:} A simplicial cocycle $\eu(\Omega) \in Z^2(K;\Z)$.
        \item[1.] For each $(ij)$ $1$-simplex of $K$ with $i < j$ let $\Lambda_{ij} \in \R$ be a lift of $\Omega_{ij} \in SO(2)$.
        \item[2.] For each $(ijk)$ $2$-simplex of $K$ with $i < j < k$ let $\eu(\Omega)_{ijk} \in \Z \subseteq \R$ the closest element to $\Lambda_{ij} + \Lambda_{jk} - \Lambda_{ik} \in \R$, using the usual distance of $\R$.
    \end{enumerate}}
    }\caption{Pseudocode for the algorithm $\eu$.} \label[algorithm]{pseudocode-eu}}
    \vspace{-6mm}
\end{table}

\paragraph{Consistency of the algorithm.}
We need the following well known result.

\begin{lem}
    \label{consistency0-eu}
    Let $\U$ be a cover of locally contractible space, with the property that non-empty binary intersections are locally path connected and simply connected.
    If $\epsilon = 0$, the construction $\eu : \apprH{1}{\epsilon}{\U}{SO(2)} \to \apprH{2}{}{\U}{\Z}$ computes the Euler class of the vector bundle represented by the cocycle.
\end{lem}
\begin{proof}
    In this proof, we freely use the language of complex vector bundles.
    The Lie group $SO(2)$ is isomorphic to the group of unit complex numbers $U(1)$, and thus, up to isomorphism, real, oriented, rank-$2$ vector bundles are exactly the same as complex rank-$1$ vector bundles.
    By \cite[(20.10.6)]{BT}, the top Chern class of a complex vector bundle coincides with the Euler class of the bundle, seen as an oriented real vector bundle in view of the inclusions $U(n) \subseteq SO(2n)$.
    Finally, when $\epsilon = 0$, the construction $\eu$ computes the first Chern class (\cite[Example~A.5]{spingeometry}).
\end{proof}

\begin{prop}
    \label{consistency-eu}
    Let $\U$ be a countable cover of a locally contractible, paracompact space $B$, with the property that non-empty binary intersections are locally path connected and simply connected.
    Let $\epsilon < 1/9$ and let $\Omega \in \apprH{1}{\epsilon}{\U}{O(d)}$.
    Then $\eu(\Omega) \in \apprH{2}{}{\U}{\Z}$ coincides with the Euler class of the vector bundle classified by $\pi_*(\vv(\Omega)) : B \to \GR{d}$.
\end{prop}
\begin{proof}
    Since $1/9 \leq \sqrt{2}/4$, there exists, by \cref{relation-to-true-vb}, an exact cocycle $\Lambda \in \apprH{1}{}{\U}{O(d)}$ such that $\disH{\Omega}{\Lambda} < 1$, and such that $\pi_*(\vv(\Omega)) = \vv(\Lambda)$.
    From the stability of $\eu$, it follows that $\eu(\Omega) = \eu(\Lambda)$.
    Finally, by \cref{consistency0-eu}, $\eu(\Lambda)$ is the Euler class of the vector bundle classified by $\vv(\Lambda)$.
\end{proof}

\subsection{Second Stiefel--Whitney class}
\label{sw2-section}


In order to compute the second Stiefel--Whitney class we will use a certain short exact sequence $1 \to \Z/2 \to \Pin(d) \to O(d) \to 1$.
To describe this sequence, we introduce the Clifford algebra associated to the standard inner product of $\R^d$.
We state some well known results whose proofs can be found in \cite{KT}, \cite[Chapter~1]{spingeometry}, and \cite[Chapter~1, Section~6]{BD}.

\paragraph{The group $\Pin$.}
Fix $d \in \N_{\geq 1}$.
The \define{Clifford algebra} corresponding to the inner product space $\left(\R^d, \langle-, -\rangle\right)$, which we denote by $\cl(d)$, is the quotient of the tensor algebra $\Ten(\R^d)$ by the two-sided ideal generated by elements of the form $v \otimes w + w \otimes v - 2\langle v, w\rangle \, \mathbf{1}$, for $v,w \in \R^d$, where $\mathbf{1}$ is the unit of $\Ten(\R^d)$.
We denote the product of two elements $x,y \in \cl(d)$ by $x \cdot y \in \cl(d)$.

Here we are using the ``positive convention'' for Clifford algebras, as we want the $\Pin$ group to be the group $\Pin^+$ discussed in, e.g., \cite{KT}.
We refer the interested reader to \cite[Section~1]{KT} for further details about this choice.

\begin{rmk}
    \label{representation-Pin}
A more concrete description of $\cl(d)$ is given by the free non-commutative $\R$-algebra generated by elements $\{e_1, \dots, e_d\}$ subject to the relations $e_i \cdot e_j = - e_j \cdot e_i$ if $i > j$, and $e_i^2 = \mathbf{1}$.
This description makes it evident that $\dim(\cl(d)) = 2^d$, since the elements of the form $e_{i_1}\cdots e_{i_k}$ for $i_1 < \dots < i_k$ form a basis of $\cl(d)$ (\cite[Chapter~1, Corollary~6.7]{BD}).
\end{rmk}

If $v \in \R^d$ is of unit length, then it is invertible when seen as an element of $\cl(d)$, since $v \cdot v = \mathbf{1}$.
Let $\Pin(d) \subseteq \cl(d)^\times$ be the subgroup of the group of units of $\cl(d)$ generated by elements $v \in \R^d$ of unit length.
There is a group morphism $\ad : \Pin(d) \to O(d)$ that is defined on generators by mapping $v \in \R^d$ with $\|v \| = 1$ to the orthogonal transformation given by reflection about the hyperplane orthogonal to $v$.
This morphism is surjective, and its kernel consists of $\{\pm \mathbf{1}\}$ (\cite[Section~1]{KT}).
This gives a central sequence of Lie groups $\Z/2 \cong \{\pm \mathbf{1}\} \to \Pin(d) \to O(d)$.

\paragraph{Algorithm.}
Let $\U$ be a cover of a topological space, with the property that non-empty binary intersections are locally path connected and simply connected.
Let $\epsilon \leq 1$.
An analogous analysis to the one made for the Euler class shows that, applying \cref{change-of-coefficients} to the extension $1 \to \Z/2 \to \Pin(d) \to O(d) \to 1$, we get a map $\sw_2 : \apprH{1}{\epsilon}{\U}{O(d)} \to \apprH{2}{}{\U}{\Z/2}$.
And that $\sw_2$ is such that, if $\Omega,\Omega' \in \apprH{1}{\epsilon}{\U}{O(d)}$ satisfy $\dZ{\Omega}{\Omega'} < 1$, then $\sw_2(\Omega) = \sw_2(\Omega') \in \apprH{2}{}{\U}{\Z/2}$.

In particular, if $K$ is a simplicial complex, we get a $1$-stable algorithm $\sw_2 : \dapprH{1}{\epsilon}{K}{O(d)} \to H^2(K;\Z/2)$.


\begin{table}[h]
    {\fbox{\parbox{\textwidth}{%
    \begin{enumerate}
        \item[]\textbf{Input:} A simplicial complex $K$ with vertices $\{1, \dots, k\}$ and a discrete approximate cocycle $\Omega \in \dapprZ{1}{\epsilon}{K}{O(d)}$ with $\epsilon \leq 1$.
        \item[]\textbf{Output:} A simplicial cocycle $\sw_2(\Omega) \in Z^2(K;\Z/2)$.
        \item[1.] For each $(ij)$ $1$-simplex of $K$ with $i < j$ let $\Lambda_{ij} \in \Pin(d)$ be a lift of $\Omega_{ij} \in O(d)$.
        \item[2.] For each $(ijk)$ $2$-simplex of $K$ with $i < j < k$ let $\sw_2(\Omega)_{ijk} \in \{\pm \mathbf{1}\} \subseteq \Pin(d)$ the closest element to $\Lambda_{ij} \Lambda_{jk} \Lambda_{ik}^{-1} \in \Pin(d)$, using the geodesic distance of $\Pin(d)$.
    \end{enumerate}}
    }\caption{Pseudocode for the algorithm $\sw_2$.} \label[algorithm]{pseudocode-sw2}}
    \vspace{-6mm}
\end{table}

\paragraph{Technicalities about the algorithm.}
\label{technicalities-about-algo}
In order to see that $\sw_2$ is really an algorithm, we must explain how to lift elements from $O(d)$ to $\Pin(d)$, how to multiply elements in $\Pin(d)$, and how to decide, given an element $x \in \Pin(d)$, which one of $\mathbf{1}$ or $-\mathbf{1}$ is closest to it in the geodesic distance.

We see $\Pin(d)$ as a subset of the Clifford algebra, which we represent as in \cref{representation-Pin}.
By definition of the map $\ad : \Pin(d) \to O(d)$, to lift a matrix $M \in O(d)$ to an element of $\Pin(d)$, we must write $M$ as a product of reflection matrices, which is exactly what the QR factorization of an orthogonal matrix by means of Householder reflections does (\cite[Theorem~2.1.14]{HJ}).

In order to multiply elements of $\Pin(d)$, we again use the representation of \cref{representation-Pin}.
We can thus multiply elements of $\Pin(d)$ in $\bO(4^d)$ time.

The problem of deciding which of $\mathbf{1}$ or $-\mathbf{1}$ is closest to an arbitrary element of $\Pin(d)$ in the geodesic distance is more involved, as it depends on the geometry of $\Pin(d)$.
We explain how this can be done explicitly for the cases $d = 2,3,4$, using extraordinary isomorphisms.
The same idea works for $d = 5, 6$, but one must be able to compute geodesic distances in $Sp(4)$ and $SU(4)$.

We first remark that, since $O(d)$ has two connected components, so must $\Pin(d)$, and that $+\mathbf{1}$ and $-\mathbf{1}$ belong to the same connected component of $\Pin(d)$, as they are both preimages of the identity matrix under the covering map $\Pin(d) \to O(d)$.
Let the \define{spin group} $\Spin(d) \subseteq \Pin(d)$ be the connected component of $+\mathbf{1}$.
Checking if $x \in \Pin(d)$ belongs to $\Spin(d)$ is easy, since it amounts to checking if its image in $O(d)$ has positive determinant.

If $x \in \Pin(d)$ does not belong to $\Spin(d)$, then its distance to $+\mathbf{1}$ is the same as the distance to $-\mathbf{1}$, namely infinity.
With this in mind, we may assume that $x \in \Spin(d)$ and that we want to know if $x$ is closest to $+\mathbf{1}$ or to $-\mathbf{1}$ in the geodesic distance of $\Spin(d)$.

For explicit formulas for the extraordinary isomorphisms mentioned below, see the example after Lemma~1.8.3 in \cite{Jurgen} or \cite{Pertti}.
One should keep in mind that \cite{Jurgen} is working with the ``negative convention'' for the Clifford algebra, and thus the isomorphisms are given for $\Spin(d) \subseteq \cl^-(d)$; but it is not hard to modify their formulas to get the isomorphisms when $\Spin(d) \subseteq \cl(d)$.

\begin{itemize}
    \item[($d=2$)] In this case, $\Spin(d)$ is isomorphic to the circle as Lie groups,
    so the geodesic distance on $\Spin(2)$ can be computed, up to a multiplicative constant, using arclengths on the circle.
    Note that the multiplicative constant is inconsequential for the purposes of determining if an element is closer to $+\mathbf{1}$ or to $-\mathbf{1}$.

    \item[($d=3$)] In this case, $\Spin(d)$ is isomorphic to $SU(2)$, which in turn is diffeomorphic to the $3$-sphere $S^3$.
    The geodesic distance is again computed, up to a multiplicative constant, using arcs.
    \item[($d=4$)] In this case, $\Spin(d)$ is isomorphic to $SU(2) \times SU(2)$, which in turn is diffeomorphic to $S^3 \times S^3$.
    Moreover, the Riemannian metric on $\Spin(4)$ induced by the double cover $\Spin(4) \to SO(4)$ coincides, up to a multiplicative constant, with the product metric on $S^3 \times S^3$, where the two copies of $S^3$ are endowed with the Riemannian metric inherited from the inclusion $S^3 \subseteq \R^4$.
    To compute the geodesic distance on $S^3 \times S^3$, recall that a geodesic in a product Riemannian manifold corresponds to a pair of geodesics, one on each factor (\cite[Problem~5-7]{jL}).
    This implies that, in a product of complete Riemannian manifolds $X\times Y$, the geodesic distance satisfies the Pythagorean theorem, meaning that the geodesic distance from $(x,y)$ to $(x',y')$ is equal to the square root of the sum of the squares of the geodesic distances from $x$ to $x'$ and from $y$ to $y'$.
\end{itemize}


So, when $d = 2,3,4$, this results in a numerical algorithm $\sw_2$ that is polynomial in the number of vertices of $K$.

\paragraph{Consistency of the algorithm.}
We need a well known result; see, e.g., \cite[Lemma~1.3]{KT}.

\begin{lem}
    \label{consistency0-sw2}
    Let $\U$ be a cover of a locally contractible space, with the property that non-empty binary intersections are locally path connected and simply connected.
    If $\epsilon = 0$, the construction $\sw_2 : \apprH{1}{\epsilon}{\U}{O(d)} \to \apprH{2}{}{\U}{\Z/2}$ computes the second Stiefel--Whitney class of the vector bundle represented by the cocycle.\qed
\end{lem}

Note that, also by \cite[Lemma~1.3]{KT}, if we would have used the negative convention for the Clifford algebra, the algorithm would be computing $\sw_1^2 + \sw_2$ instead of $\sw_2$.

The proof of the following proposition is analogous to the proof of \cref{consistency-eu}, but uses \cref{consistency0-sw2} instead of \cref{consistency0-eu}.

\begin{prop}
    \label{consistency-sw2}
    Let $\U$ be a countable cover of a locally contractible, paracompact space $B$, with the property that non-empty binary intersections are locally path connected and simply connected.
    Let $\epsilon < 1/9$ and let $\Omega \in \apprH{1}{\epsilon}{\U}{O(d)}$.
    Then $\sw_2(\Omega) \in \apprH{2}{}{\U}{\Z/2}$ is the second Stiefel--Whitney class of $\pi_*(\vv(\Omega))$.\qed
\end{prop}

\section{Computational examples}
\label{examples-section}

In this section, we run the algorithms of \cref{section-characteristic-classes} on data.
We describe the basic pipeline used in the examples in \cref{examples-pipeline} .

In \cref{example-1}, we use $\sw_1$ to study the topology and, in particular, the orientability of attractors in the time-variant double-gyre dynamical system.
This illustrates how characteristic classes can be used to classify such attractors.
The trajectory of particles in the dynamical system was computed numerically using the code in \cite{fernandez}.

In \cref{example-2}, we use $\sw_1$ and $\sw_2$ to provide experimental evidence that confirms that a certain dataset of lines in $\R^2$ can be parametrized by a projective plane.
The theoretical explanation of this phenomenon appears in \cite[Section~2.4]{PC}.
The code used to generate the dataset is from \cite{dreimac}.

In \cref{example-3}, we show how $\eu$ can be used to detect non-synchronizability of data.
The example is a version of the cryo-EM problem mentioned in \cref{synchronization-subsection}.

In \cref{table:times} we summarize the runtime of our algorithms.
Code to replicate these examples can be found at \cite{examples-code}.

\subsection{Pipeline}
\label{examples-pipeline}
We make free use of basic tools from persistence theory such as Vietoris--Rips complexes, persistent cohomology, and persistence diagrams; see, e.g., \cite{TDA1,TDA2}.
The persistent cohomology computations are done using the Python interface (\cite{ripserpy}) for Ripser (\cite{ripser}).

Let $\{K_r\}_{r \in \R_{\geq 0}}$ be a filtration of a finite simplicial complex $K$ by subcomplexes, so that we have $K_s \subseteq K_r \subseteq K$ for $s \leq r \in \R$.
For $\Omega \in \apprH{1}{\infty}{K}{O(d)}$ and $\epsilon \geq 0$, define the \define{$\epsilon$-death} of $\Omega$ as
\[
    \delta = \sup \left(\{r \geq 0 : \Omega \text{ is an $\epsilon$-approximate cohomology class on $K_r$}\} \cup \{0\}\right).
\]
The \define{$\epsilon$-span} of $\Omega$ is the subset of the persistence diagram of $K$ of classes whose (homological) birth is at most $\delta$ and whose (homological) death is at least $\delta$.

Fix a basis for the persistent cohomology of $K$, that is, cocycles $B = \{\Lambda_1, \dots, \Lambda_n\}$ representing each point of the persistence diagram of the $\Z/2$-cohomology of $K$.
Let $\epsilon = 2$.
For any $r < \delta$ we can write $\sw_1(\Omega)$ in the basis $B$.
To represent $\sw_1(\Omega)$ in the persistence diagram of $K$, we decorate the diagram by circling the classes of $B$ that appear with a non-zero coefficient.
Note that any such class must live inside the $2$-span of $\Omega$.
An analogous analysis, replacing $2$ with $1$, applies for $\sw_2$.
For $\eu$, we use $\Z/3$-cohomology and the mod $3$ reduction of the Euler class.

\subsection{Orientability of attractors}
\label{example-1}

The \define{double-gyre} is the dynamical system characterized by the differential equations $\dot{x} = \partial \psi/\partial y$ and
$\dot{y} = - \partial \psi/\partial x$ where
\[
    \psi(x,y,t) = A \;\sin(\pi f(x,t)) \sin(\pi y),\;\;\;\; f(x,t) = a(t)\; x^2 + (1-2a(t))\; x,\;\;\;\; a(t) = \epsilon \sin(\omega t),
\]
and $A$, $\epsilon$, and $\omega$ are positive parameters.
The system is defined over the domain $(x,y) \in [0,2] \times [0,1]$ with $t \in \R$ representing time.
It was introduced in \cite{dg} as a simplified model of the double-gyre pattern observed in geophysical flows.
Geometrically, the double-gyre system consists of a pair of vortices oscillating back and forth, horizontally, in time; see, e.g., \cite[Figure~5]{dg}.

Since the vector field characterizing the system varies periodically with time, the flow line of a particle initially at $(x_0,y_0)$ depends on the time $t_0$ at which the particle is at that spot.
Because of this, the \textit{phase space} of the system (the space parametrizing the possible states of a particle) is $[0,2] \times [0,1] \times S^1$, where the last coordinate represents time $\R$ modulo the period $\pi/\omega$.

Dynamical systems can be analyzed by studying the topology of their attractors (\cite{takens,DS1,TTS}).
Informally, an \textit{attractor} $\M$ of a dynamical system consists of a subset of the phase space that is invariant under the action of the system, and such that any point sufficiently close to the attractor gets arbitrarily close to it as the system evolves.
We refer the interested reader to \cite{smale,williams,takens} for formal notions of attractor.

Usually, one only has access to partial information about the trajectory of a particle, which one wants to use to study the topology of the attractor $\M$ the particle is converging to.
For example, one may be given a real-valued time series $\{x_n\}_{n=1,\dots,N}$, which comes from applying a differentiable map $F$ defined on the phase space to a finite sample of the trajectory of the particle.
If the attractor $\M$ is a smooth manifold and certain other technical conditions are satisfied, a theorem of Takens (\cite{takens}) implies that the \textit{delay embedding} of the time series
\[
    X = \left\{(x_i,x_{i+\tau},x_{i+2\tau},\dots,x_{i+(d-1)\tau})\right\}_{i=1,\dots,N-(d-1)\tau} \subseteq \R^d
\]
is concentrated around a diffeomorphic copy of $\M$, and is sufficiently dense in $\M$ so that the Vietoris--Rips complex of $X$ can be used to estimate the homology of $\M$, and local PCA can be used to estimate the tangent bundle of $\M$.
Here $d$ is the \textit{target dimension} and $\tau$ is the \textit{delay}.
We refer the reader to \cite{takens,DS1,TTS,TT} for more information about recovering the geometry of attractors from a delay embedding.

\begin{figure}[h]
        \centering
    \begin{subfigure}{0.49\textwidth}
        \centering
        \includegraphics[width=0.5\textwidth]{"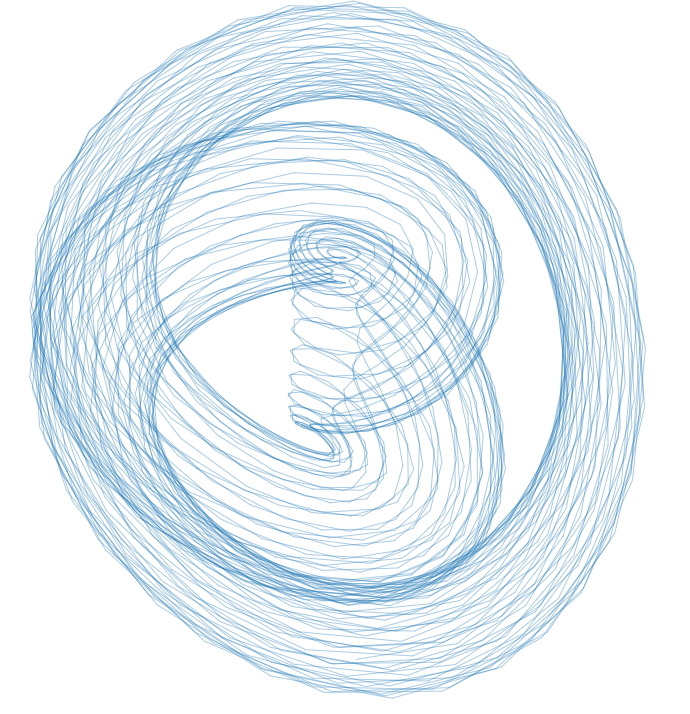"}
        \caption{2D projection of the \texttt{ATTR1} dataset.}
        \label{fig:attr}
    \end{subfigure}\hfill
    \begin{subfigure}{0.49\textwidth}
        \centering
        \includegraphics[width=0.45\textwidth]{"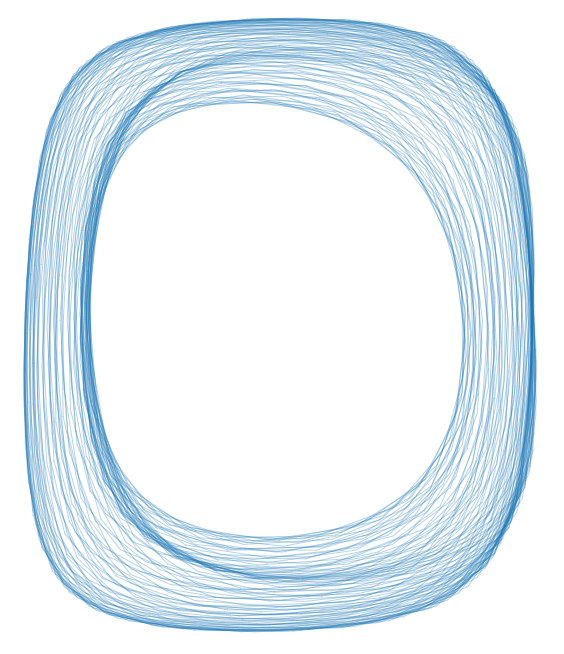"}
        \caption{2D projection of the \texttt{ATTR2} dataset.}
        \label{fig:attr2}
    \end{subfigure}
\end{figure}

In \cite[Section~4.1]{dgtop}, four attractors of the double-gyre dynamical system are studied.
The four attractors can be distinguished using their homology, except for two of them which, topologically, look like a cylinder and a M\"obius strip respectively, and thus have the same homology groups.
In order to deal with such examples, the authors of \cite{dgtop} develop an algorithm for detecting orientability of this kind of data.
Here, we show that our general purpose algorithms readily apply to this kind of data, and confirm, using $\sw_1$, that \cite[Section~4.1,~Example~4]{dgtop} is parametrized by a M\"obius band and \cite[Section~4.1,~Example~1]{dgtop} by a cylinder.

Fixing the initial condition $(x_0,y_0,t_0) = (0.55,0.5,0) \in [0,2] \times [0,1] \times S^1$, for a double-gyre with $A=0.1$, $\epsilon = 0.1$, and $\omega=\pi/5$, as in \cite{dgtop}, we sample the trajectory of the particle at $2000$ equally spaced times from $t=0$ to $t=1000$, obtaining a time series of positions $\{(x_n,y_n)\}_{n=1,\dots,2000} \subseteq [0,2] \times [0,1]$.
We take the function $F$ to be projection onto the $x$-coordinate, as in \cite{dgtop}, the delay $\tau =5$ and the target dimension $d=5$.
A 2D projection of $X$ is depicted in \cref{fig:attr}.
The dataset \texttt{ATTR1} consists of a subsample of $1000$ points of $X$.

We approximate the tangent bundle of $\M$ using local PCA as sketched in \cref{motivation-local-pca}.
This give us a discrete approximate local trivialization $\Phi$, which we use to compute an approximate cocycle $\Omega = \bw(\Phi)$ over the Vietoris--Rips complex of $X$.
We choose $k = 58$ for the local PCA computations as this value maximizes the $2$-death of $\Omega$.

\begin{figure}[h]
        \centering
    \begin{subfigure}{0.49\textwidth}
        \centering
        \hspace*{-0.9cm}
        \includegraphics[width=1.2\textwidth]{"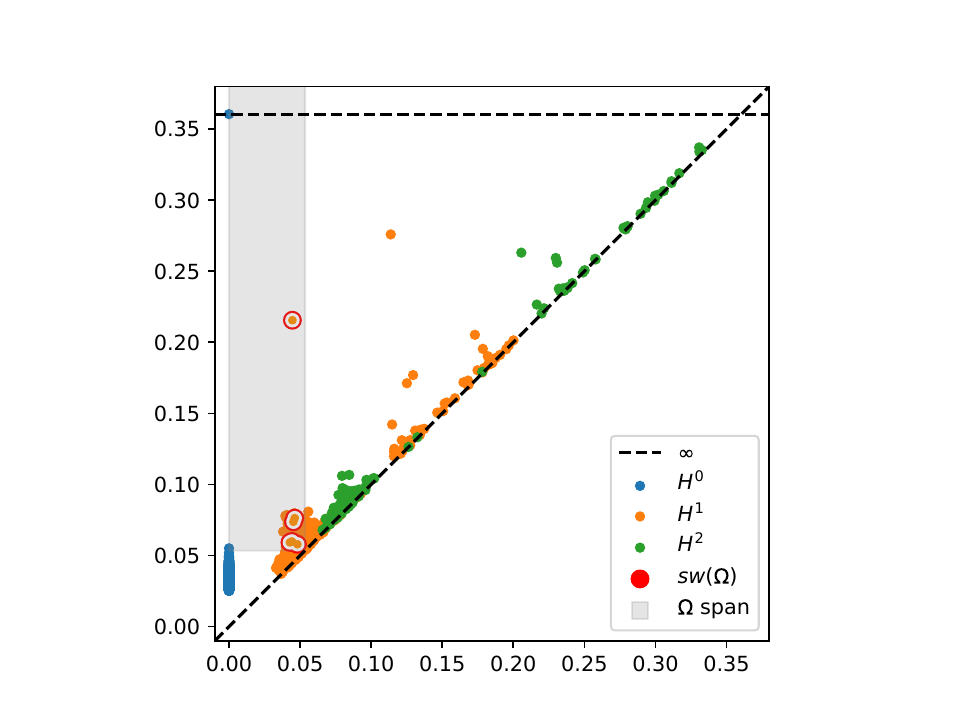"}
        \vspace*{-0.65cm}
        \caption{Persistence diagram of $\Z/2$-cohomology of \texttt{ATTR1} with Stiefel--Whitney class.}
        \label{fig:attr-pd}
    \end{subfigure}\hfill
    \begin{subfigure}{0.49\textwidth}
        \centering
        \hspace*{-0.9cm}
        \includegraphics[width=1.2\textwidth]{"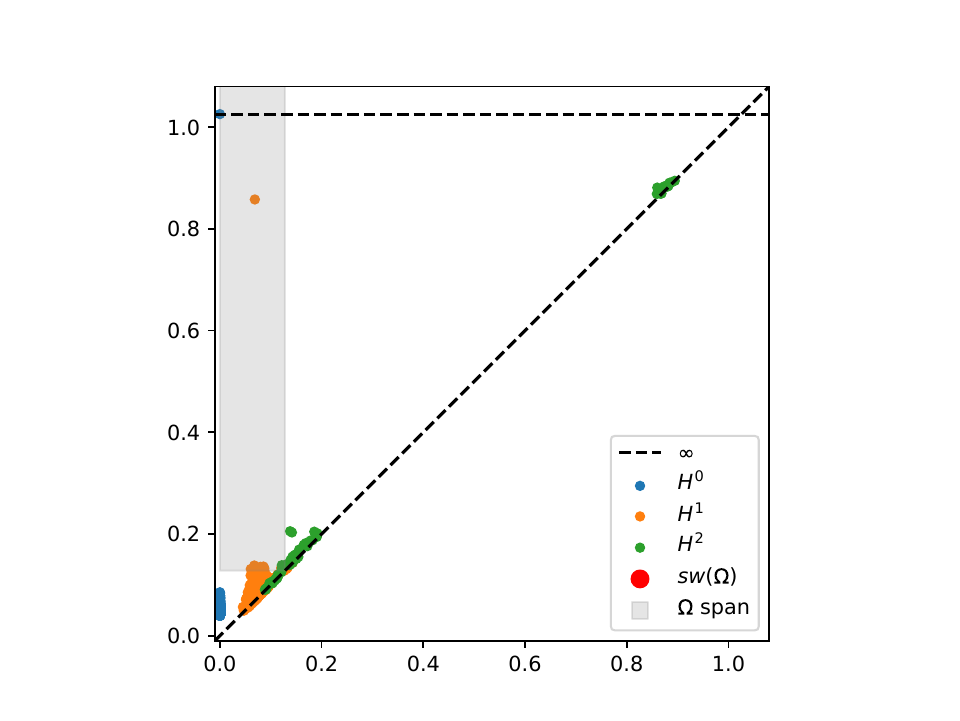"}
        \vspace*{-0.65cm}
        \caption{Persistence diagram of $\Z/2$-cohomology of \texttt{ATTR2} with Stiefel--Whitney class.}
        \label{fig:attr-pd2}
    \end{subfigure}
\end{figure}

Finally, we apply the algorithm $\sw_1$ to $\Omega$ and write the cohomology class thus obtained in the basis of the persistent cohomology provided by Ripser.
The classes circled in red in \cref{fig:attr-pd} are the classes that sum to $\sw_1(\Omega)$.
The fact that the most persistent class of the persistent diagram is colored in red tells us that, as one goes around the $1$-dimensional hole this class represents, the approximate vector bundle given by local PCA is changing orientation.
This, coupled with the fact that the dataset is locally $2$-dimensional, suggests that the dataset is parametrized by a M\"obius band.

We run the same pipeline but with initial condition $(x_0,y_0,t_0) = (0.25, 0.125, 0)$, as in \cite[Section~4.1,~Example~1]{dgtop}; we refer to this second dataset as \texttt{ATTR2}.
In this case, regardless of the parameters chosen for local PCA, the first Stiefel--Whitney class of the approximate cocycle it gives is zero, as we see in \cref{fig:attr-pd2}.

\subsection{Projective space of lines}
\label{example-2}

Datasets of simple geometrical shapes, such as lines in $\R^2$ with different orientations, appear, for instance, as datasets of weights learned by autoencoders and other neural networks (\cite[Figure~2.A]{hinton2}, \cite[Figures~2~and~3]{hintonetal}), and as datasets of local patches of natural images (\cite{CISZ}).
It has been shown that topology can be used to understand these datasets by giving insight into the functions learned by the hidden layers of neural networks (\cite{GC,NZL}), and by finding convenient parametrizations of spaces of image patches (\cite{PC}).

Here, we show how characteristic classes can be used to understand the topology of a set of lines in $\R^2$.
The dataset \texttt{LINES} consists of $160$ grayscale images represented by a $10 \times 10$ real matrix.
Each image represents a fuzzy line in $\R^2$ with a certain slope and offset.
A sample of $56$ elements of \texttt{LINES} is displayed in \cref{fig:projective-grid}.
We interpret this dataset as a point cloud $X \subseteq \R^{10 \times 10}$.

We proceed with the same pipeline as in \cref{example-1}.
We compute the persistence diagram of $\VR(X)$ with coefficients in $\Z/2$ and observe that there are two significant classes, one in $H^1$ and the other in $H^2$ (\cref{fig:projective-pd}).
By applying local PCA with a range of values $k$, we see that the dataset seems to have an intrinsic dimension of $2$.
This may lead one to suspect that the dataset is parametrized by the real projective plane.
One way to confirm this suspicion is by computing persistent cohomology with coefficients in $\Z/3$, and seeing that the most persistent classes in $H^1$ and $H^2$ disappear.

\begin{figure}
        \centering
    \begin{subfigure}{0.49\textwidth}
        \centering
        \includegraphics[width=0.65\textwidth]{"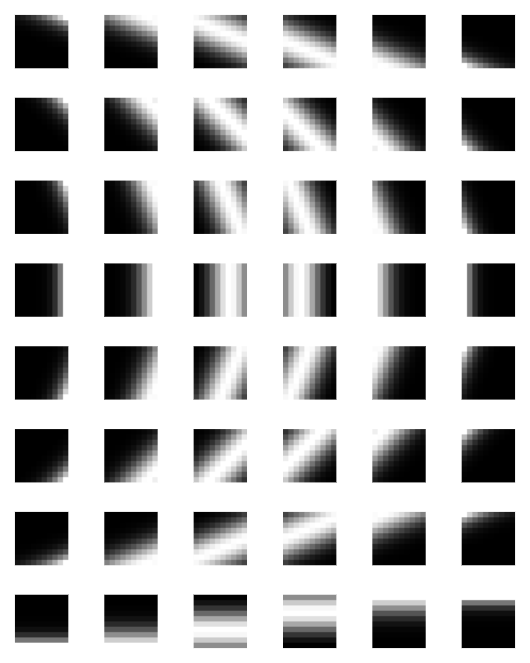"}
        \vspace*{0.73cm}
        \caption{Grid displaying $56$ images of the \texttt{LINES} dataset.}
        \label{fig:projective-grid}
    \end{subfigure}\hfill
    \begin{subfigure}{0.49\textwidth}
        \centering
        \hspace*{-0.9cm}
        \includegraphics[width=1.3\textwidth]{"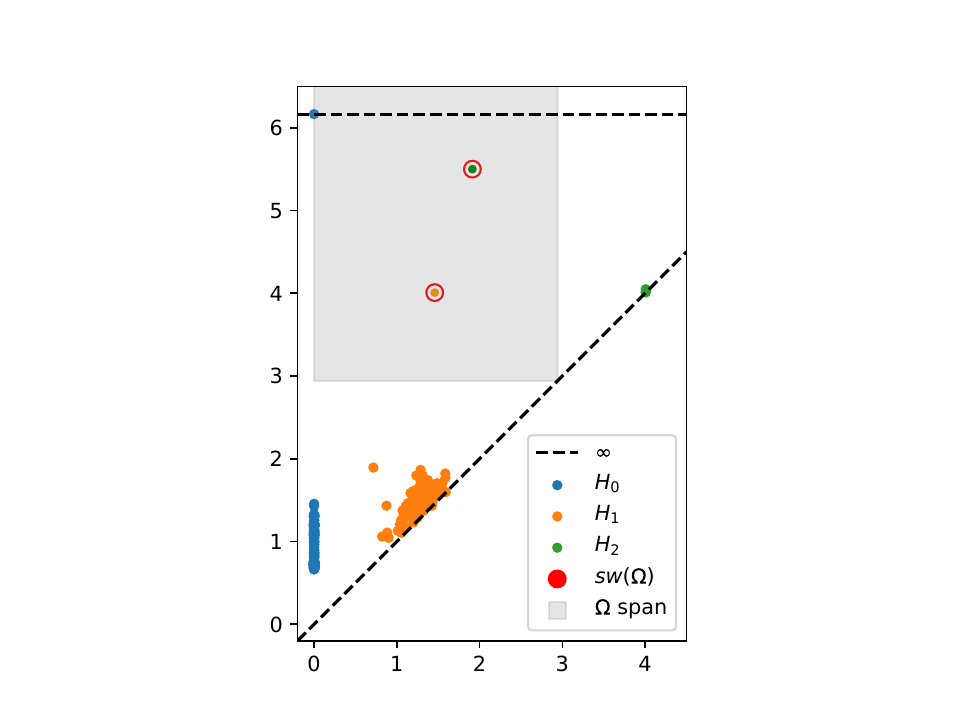"}
        \caption{Persistence diagram of $\Z/2$-cohomology of \texttt{LINES} with Stiefel--Whitney classes.}
        \label{fig:projective-pd}
    \end{subfigure}
\end{figure}

Another method to corroborate the hypothesis is the following.
Local PCA gives a discrete approximate local trivialization $\Phi$ over $\VR(X)$, and the $1$-death of $\Omega = \bw(\Phi)$ is maximized using $k = 18$.
We apply the algorithms $\sw_1$ and $\sw_2$ to $\Omega$ and see that the cohomology classes thus obtained coincide with the most persistent $H^1$ and $H^2$ classes of \cref{fig:projective-pd}, respectively.
This confirms the hypothesis that the dataset is parametrized by a projective plane, as we know that $H^1(\R P^2; \Z/2) \cong H^2(\R P^2; \Z/2) \cong \Z/2$ and that the Stiefel--Whitney classes of the tangent bundle of $\R P^2$ are non-zero.


\subsection{Lack of global synchronization in cryo-EM}
\label{example-3}

We give an example of how characteristic classes provide an obstruction to synchronization.
In order to do this, we simulate an instance of the main problem of cryo-EM (\cite{F,vHea}).
In this problem, one is given a set of 2D projections of an unknown 3D shape, and is asked to reconstruct the shape.
One possible formalism is to assume that the shape is given by a density $S : \R^3 \to \R$.
We think of the projection process as first acting on the molecule by an unknown element $v \in SO(3)$, which we call \define{projection angle}, giving $S \circ v : \R^3 \to \R$, and then integrating $S\circ v$ along the $z$-axis, yielding a map $S_v : \R^2 \to \R$.

One of the main difficulties is that one is only given a set of projected images, and is not given the projection angle $v$ corresponding to each image.
Much attention has thus been given to the problem of recovering the projection matrices corresponding to each image, up to a global rigid automorphism of $S^2$, i.e., an element of $SO(3)$.

Let $X$ be a set of 2D projections of an unknown 3D shape.
A successful approach (\cite{HS1,HS2,YL}) to recovering the projection angles starts by computing rotations $g_{ij} \in SO(2)$ that best align $x_i,x_j \in X$.
Formally, one chooses a distance $d$ between maps $\R^2 \to \R$ and finds $g_{ij} \in SO(2)$ that minimizes the distance between $x_i : \R^2 \to \R$ and $g_{ij} \circ x_j : \R^2 \to \R$.
One then fixes a threshold $\delta$ and constructs a $2$-dimensional simplicial complex $K$ with $0$-simplices given by the elements of $X$, $1$-simplices given by pairs $(ij)$ such that $d(x_i, g_{ij}x_j) < \delta$ and, and having a $2$-simplex $(ijk)$ anytime $(ij)$, $(jk)$, and $(ik)$ are $1$-simplices of $K$.
This construction can be interpreted as giving an approximate $SO(2)$-cocycle over $K$.
This cocycle approximates the tangent bundle of $S^2 \subseteq \R^3$, and can be used to recover the projection angles.

We are interested in simulating an instance of this problem, and in computing the Euler class of the approximate cocycle.
As we expect this class to be non-zero, this shows that characteristic classes of approximate vector bundles can be used for model validation and to detect non-synchronizability of data.
This also confirms the theoretical analysis of \cite{SZSH}, which shows that global synchronization is not possible.

In order to construct an instance of the problem, we define $S : \R^3 \to \R$ to be the characteristic function of a union of four balls of different radii, centered at different points of $\R^3$.
We then compute $400$ projections $\{S_{v_i} : \R^2 \to \R\}_{i = 1, \dots, 400}$ for $\{v_i \in SO(3)\}$ a well-distributed random sample of $SO(3)$.
We save these projections as $100 \times 100$ grayscale images, which constitute the dataset \texttt{PROJS}.
A sample of $9$ elements of \texttt{PROJS} is displayed in \cref{fig:proj-grid}.

\begin{figure}
        \centering
    \begin{subfigure}{0.49\textwidth}
        \centering
        \includegraphics[width=0.78\textwidth]{"./pictures/projs.png"}
        \vspace{0.5cm}
        \caption{Grid displaying $9$ images of the \texttt{PROJS} dataset.}
        \label{fig:proj-grid}
    \end{subfigure}\hfill
    \begin{subfigure}{0.49\textwidth}
        \centering
        \hspace*{-0.8cm}
        \includegraphics[width=1.2\textwidth]{"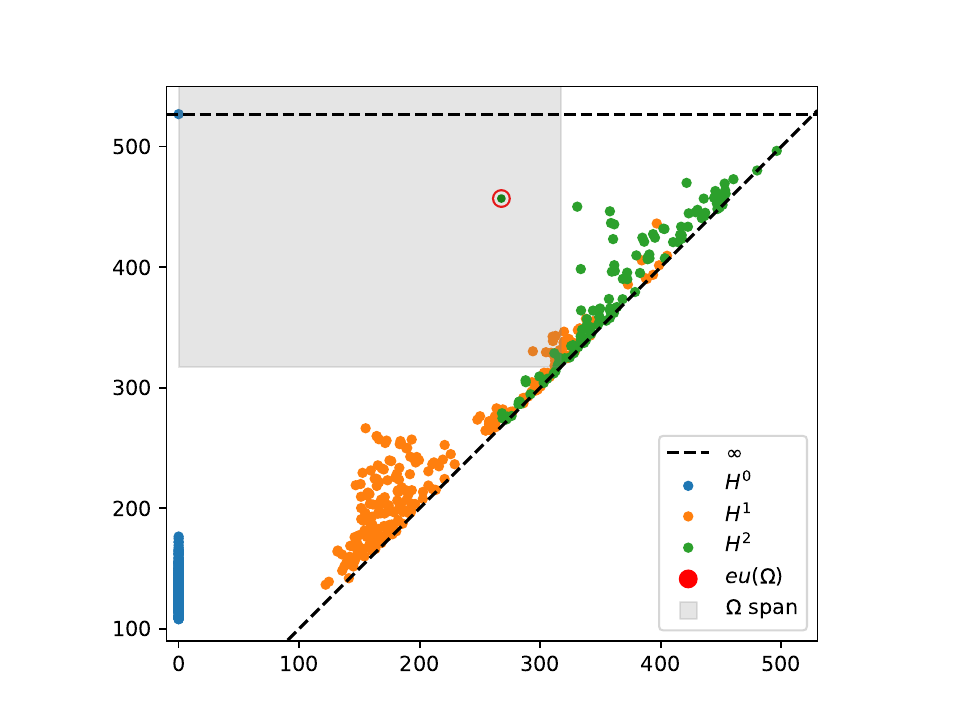"}
        \caption{Persistence diagram of $\Z/3$-cohomology of \texttt{PROJS}, with mod $3$ reduction of the Euler class.}
        \label{fig:proj-pd}
    \end{subfigure}
\end{figure}

Since aligning images optimally is a difficult problem in itself, we simplify our computations by using our knowledge of the projection angles $v_i$ and $v_j$ to align the images $x_i$ and $x_i$.
More specifically, given $v \in SO(3)$, let $v^3 \in S^2$ denote its third column.
We compose $v_i$ with the rotation $r_{ij} \in SO(3)$ along a minimizing geodesic between $v_i^3$ and $v_j^3$, and define $\Omega_{ij} \in SO(2)$ such that
\[
    r_{ij} v_i =
      \begin{pmatrix} \Omega_{ij} & 0 \\
                      0 & 1 \end{pmatrix} v_j.
\]

We let $X$ be a sample of $300$ elements of \texttt{PROJS} and compute the Vietoris--Rips complex of $X$, where the dissimilarity function $d(i,j)$ is given by the distance between the image $x_i$ and the rotated image $\Omega_{ij}x_j$.

Finally, we compute the persistence diagram of $\VR(X)$ with coefficients in $\Z/3$, compute the Euler class of $\Omega$ at its $1$-death, reduce it mod $3$, and write it in the basis provided by the persistent cohomology computations.
We see that the persistent class representing the reduction mod $3$ of the Euler class, which appears circled in red in \cref{fig:proj-pd}, is the most persistent class of $\VR(X)$.

\begin{table}[h]
    \centering
\begin{tabular}{|c|r|c|c|}
\hline
Dataset $X$                     &\multicolumn{1}{|c|}{$\VR(X)$ size} & Algorithms       & Time   \\
\hline
\texttt{ATTR1} & $2292$ \;$1$-simplices at $2$-death of $\Omega$            & $\bw$ and $\mathsf{sw}_1$ & $180$ms  \\
\hline
\texttt{ATTR2} & $7834$ \;$1$-simplices at $2$-death of $\Omega$             & $\bw$ and $\mathsf{sw}_1$ & $450$ms  \\
\hline
\texttt{LINES} & $3174$ \;$1$-simplices at $2$-death of $\Omega$             & $\bw$ and $\mathsf{sw}_1$ & $270$ms  \\
\hline
\texttt{LINES} & $26683$ \;$2$-simplices at $1$-death of $\Omega$            & $\bw$ and $\mathsf{sw}_2$ & $3950$ms \\
\hline
\texttt{PROJS} & $45573$ \;$2$-simplices at $1$-death of $\Omega$            & $\mathsf{eu}$   & $180$ms \\
\hline
\end{tabular}
        \caption{Runtime of the algorithms on laptop with a 2.20 GHz Intel Core i7 and 16GB of RAM.}
        \label{table:times}
\end{table}

\section{Future work}
\label{conclusions-section}


\paragraph{Robustness and applications.}
The notions of vector bundle we have presented can tolerate a certain amount of noise but are not robust to outliers.
What this means is that an $\epsilon$-approximate cocycle on a simplicial complex can satisfy the $\delta$-approximate cocycle condition for a very small $\delta$ on almost all $2$-simplices, except on very few, and this can still make $\epsilon$ very large.
This problem is similar to the robustness problem of persistent homology, were the addition of a few outliers to the dataset can completely alter the inferred persistent homology.

We are interested in studying extensions of our framework that lead to algorithms for the inference of topological information of approximate vector bundles even in the presence of outliers.
The work in \cite{S2} is an example of such an algorithm, as it can be interpreted as a robust algorithm for inferring whether the first Stiefel--Whitney class of the tangent space of an embedded manifold is $0$ or not.

\paragraph{Invariants of approximate cocycles.}
A related problem is that of defining invariants of $\epsilon$-approximate cocycles that do not require $\epsilon$ to be sufficiently small.
A possible avenue is to consider the persistent cohomology of the thickened Grassmannians, and use $\epsilon$-approximate classifying maps to pull back the cohomology classes of these Grassmannians.
The question of how long the universal Stiefel--Whintey classes persist in the thickened Grassmannians is related to the filling radius introduced by Gromov (\cite{Gromov}) and to the generalization considered by Lim, M\'{e}moli, and Okutan in \cite{LMO}.
A related question is whether better bounds for our results can be obtained by using the Vietoris--Rips complex of the Grassmannians, instead of the thickening.

\paragraph{Approximate sections.}
In applications, one is interested in finding sections of approximate vector bundles.
In order to do this in practice, and to prove consistency theorems for this kind of computation, a notion of approximate section needs to be introduced and studied.

\paragraph{Approximate vector bundles with connection.}
Discrete approximate cocycles over simplicial complexes, as in \cref{definition:discrete-approximate-vector-bundles}, appear in the physics literature (e.g., \cite{christiansen-halvorsen}), where they have a different interpretation: the value $\Omega_{ij} \in O(d)$ of the (discrete) cocycle $\Omega$ on the $1$-simplex $(ij)$ represents parallel transport from the fiber of the vertex $i$ to the fiber of the vertex $j$.
It would be interesting to relate this intepretation to the results in this paper, and to study to what extent a discrete approximate cocycle can be used to reconstruct a vector bundle together with a connection.

\paragraph{Algorithms for higher characteristic classes.}
Finding algorithms for the effective computation of higher characteristic classes, at least in the exact case, has been the subject of much research in the past.
For instance, \cite{BMc} and \cite{Mc} give cocycles representing Chern, Pontryajin, Stiefel--Whitney, and Euler classes of vector bundles over compact manifolds.
One should note that many of the formulas are not entirely algorithmic, as they require to determine, for example, whether certain singular cycles are $0$ in homology or not.

\paragraph{Other structure groups.}
Some synchronization problems are most naturally described using a structure group different from the orthogonal group $O(d)$.
We hope our theory can be extended to bundles having other structure groups, such as $U(d)$, $PO(d)$, and $(\R^n, +)$.


%

\appendix
\section{Technical details}
\label{technical-details}

\subsection{Reach and thickenings}
\label{thickening-subsection}

Let $\emptyset \neq X \subseteq \R^N$ be closed.
For $\epsilon \in [0,\infty]$, define the \define{$\epsilon$-thickening} of $X$ by
$X^\epsilon = X$ if $\epsilon = 0$, and otherwise by 
\[
    X^\epsilon = \left\{y \in \R^N : \exists x \in X, \| x - y\| _2 < \epsilon\right\}.
\]


A \define{closest point} of $y \in \R^N$ in $X$ is a point $x \in X$ that minimizes $\| x - y\| _2$.
The \define{medial axis} of $X$, denoted $\med(X)$, consists of all points $y \in \R^N$ that admit more than one closest point in $X$.
The \define{reach} of $X$ is the supremum of $\epsilon \geq 0$ such that $X^\epsilon$ does not intersect $\med(X)$.
Define a function $\pi : \R^N \setminus \med(X) \to X$ by mapping a point $y$ to its closest point of $X$.

\begin{lem}
    \label{projection-well-defined}
    If $\emptyset \neq X \subseteq \R^N$ is closed, then the map $\pi : \R^N \setminus \med(X) \to X$ is continuous.
    In particular, if $X$ has strictly positive reach and $\epsilon \leq \reach(X)$, then $\pi : X^\epsilon \to X$ is well-defined and continuous.
    It follows that the map $\pi : X^\epsilon \to X$ is a homotopy equivalence, with inverse the inclusion $X \to X^\epsilon$.
\end{lem}

\begin{proof}
    We start by proving the first claim.
    For this, we show that $\pi$ maps convergent sequences to convergent sequences.
    Let $y_n \to y$ be a convergent sequence in $\R^N \setminus \med(X)$.
    By the triangle inequality, we have
    \[
        \|\pi(y_n) - \pi(y)\| \leq \|\pi(y_n) - y_n\| + \|y_n - y\| + \|y - \pi(y)\| = d(y_n, X) + \|y_n - y\| + d(y,X) \to 2 d(y,X).
    \]
    Thus, for $n$ sufficiently large, we have that $y_n$ is inside the closed ball of radius $2d(y,X) + 1$ around $\pi(y)$, which is a compact set.
    Combining this with the fact that $X$ is closed, we can assume, without loss of generality, that $\pi(y_n)$ converges to some point $x \in X$.
    We must show that $x = \pi(y)$, and to prove this it is enough to show that $\| y - x\| _2 = d(y,X)$, by uniqueness of minimizers.
    By definition, $\| y_n - \pi(y_n)\| _2 = d(y_n, X)$, and $d(-,X)$ is a continuous function, so, taking a limit, we see that $\| y - x\|_2  = d(y, X)$.

    For the second claim it suffices to show that $X^{\reach(X)}$ does not intersect $\med(X)$, which is clear since any point of $X^{\reach(X)}$ belongs to $X^\delta$ for some $\delta < \reach(X)$.

    Finally, to see that $\pi$ is a homotopy equivalence, note that the inclusion followed by the projection is certainly the identity of $X$, and the projection followed by the inclusion is homotopic to the identity of $X^\epsilon$ by means of a linear homotopy.
\end{proof}

%

\begin{lem}
    \label{distance-projection-vertex}
    Let $X \subseteq \R^N$ be closed with strictly positive reach $r$.
    Let $z$ be a convex combination of $x_1,\dots,x_n \in X$ such that $\| x_i - x_j\|  < r$ for all $1 \leq i,j \leq n$.
    Then $\| z - x_j\|  < r$ and $\| \pi(z) - x_j\|  < 2 r$ for all $1 \leq j \leq n$.
\end{lem}
\begin{proof}
    Fix $1 \leq j \leq n$.
    The points $x_1, \dots, x_n$ are all contained in the open ball of radius $r$ around $x_j$.
    It follows that $z$ is also contained in this ball, by the convexity of open balls.
    In particular, $z \in X^r$ and $\|\pi(z) - z\| < r$.
    By the triangle inequality, we have $\|pi(x) - x_j\| < 2r$, as required.
\end{proof}

\subsection{Polar decomposition and orthogonal Procrustes problem}
\label{polar-decomposition-and-procrustes-section}

In this section we collect a few standard facts about the orthogonal Procrustes problem; a standard reference is \cite{HJ}.

Let $n \geq m \geq 1 \in \N$.
Let $N \in \R^{n \times m}$.
A \define{polar decomposition} of $N$ consists of matrices $U \in \R^{n \times m}$ and $P \in \R^{m\times m}$ such that $N = U P$, the matrix $U$ has orthonormal columns, and the matrix $P$ is positive semidefinite.
We will need the following well known fact about polar decompositions; for a reference, see \cite[Theorem~7.3.1]{HJ}.
In the statement, for $A$ a symmetric positive semidefinite matrix, $A^{1/2}$ denotes the unique positive semidefinite square root (\cite[Theorem~7.2.6]{HJ}), also called the \define{principal square root} of $A$.

\begin{lem}
    \label{unique-polar-decomposition-lem}
    Any matrix admits a polar decomposition.
    The factor $P$ is uniquely determined and satisfies $P = (N^t N)^{1/2}$.
    The factor $U$ is uniquely determined if $N$ has full rank.\qed
\end{lem}

Let $\ST{m,n} \subseteq \R^{n \times m}$ denote the subset of matrices with orthonormal columns.
Consider the map $Q : \R^{n \times m} \to \ST{m,n}$ that assigns to a matrix $N$ a matrix $U$ that is part of a polar decomposition $N = UP$.
Note that, for general matrices, this map depends on a choice of polar decomposition.

\begin{cor}
    \label{Q-is-continuous}
    The map $Q : \R^{n \times m} \to \ST{m,n}$ is uniquely determined and continuous if we restrict its domain to matrices with full rank.
\end{cor}
\begin{proof}
    By \cref{unique-polar-decomposition-lem}, if $N$ has full rank, we have $U = N P^{-1} = N \left((N^t N)^{1/2}\right)^{-1}$.
    The map $A \mapsto A^{1/2}$ that takes a symmetric positive definite matrix to its principal square root is continuous \cite[Chapter~6.2,~Problem~26]{HJ2}, and so is the map that takes an invertible matrix to its inverse.
    The result follows.
\end{proof}

For $M,N \in \R^{n\times d}$, the \define{orthogonal Procrustes problem} is the following optimization problem:
\[
    \min_{\Omega \in O(d)} \| M \Omega - N\| .
\]
It is well known (see, e.g., \cite[Section~7.4.5]{HJ}) that the solutions to the above orthogonal Procrustes problem are precisely the orthogonal matrices $U$ that are part of a polar decomposition $M^t N = U P$.
In particular, if $M^t N$ has full rank, then the above problem admits exactly one solution.
Moreover, this solution varies continuously in $M^t$ and $N$, as long as we restrict ourselves to problems for which $M^t N$ has full rank.


\subsection{Metrics on Grassmannians}

\label{metrics-on-grass-section}


In this section, we compare two metrics on the Grassmannians and we recall the computation of the reach of the Grassmannians due to Tinarrage (\cite{tinarrage}).

Recall that $\P : \ST{d,n} \to \GR{d,n}$ is $O(d)$-invariant, with $O(d)$ acting by matrix multiplication on the right.
This means that we have $\P(M) = \P(M \Omega)$ for every $\Omega \in O(d)$.
In fact, it is easy to see that, for $N,M \in \ST{d,n}$, one has $\P(N) = \P(M)$ if and only if $N \Omega = M$ for some $\Omega \in O(d)$.
This shows that, as a set, $\GR{d,n}$ is the quotient of $\ST{d,n}$ by the action of $O(d)$ by matrix multiplication on the right.
Since $O(d)$ acts by isometries, this induces a metric $\dqsa$ on $\GR{d,n}$, given by
\[
    \dq{A}{B} = \min_{\substack{N,M \in \ST{d,n}\\ \P(N) = A, \P(M) = B}} \| N - M\| .
\]
We now prove some useful comparisons between the metric and the Frobenius distance.

\begin{lem}
    \label{product-by-frame-isometry}
    Let $M \in \ST{d,n}$, $A \in \R^{d \times m}$, and $B \in \R^{n \times m}$. Then $\| MA\|  = \| A\| $ and $\| M^t B\|  \leq \| B\| $.
\end{lem}
\begin{proof}
    For the first claim, we compute
    \[
        \| MA\| ^2 = \trace((MA)^t (MA)) = \trace(A^t M^t M A) = \trace(A^t A) = \| A\| ^2.
    \]
    For the second, note that $\| M^tB\| ^2 = \trace(M M^t B B^t)$.
    Since $M M^t$ is an orthogonal projection to a subspace, there is an orthogonal change of basis such that $M M^t$ is diagonal, with diagonal entries $1$ repeated $d$ times and $0$ repeated $n-d$ times. The result follows by computing $\trace(M M^t B B^t)$ in that basis.
\end{proof}

Let $\dFrsa$ denote the metric on Grassmannians induced by the Frobenius norm.

\begin{lem}
    \label{projection-is-lipschitz}
    The map $\P : \ST{d,n} \to \GR{d,n}$ given by $\P(M) = M M^t$ is $\sqrt{2}$-Lipschitz.
    In particular, we have $\dFrsa \leq \sqrt{2} \dqsa$.
\end{lem}

Simple examples with $d = 1$ and $n=2$ show that this bound is tight.

\begin{proof}
    The second claim is a consequence the first one.
    For the first claim, we start by recalling that, for any matrix $A$, one has $\| A^t A\|  = \| A A^t\| $.
    In particular, if $A$ is square, $\| AA^t - \id\| ^2 = \| AA^t\| ^2 + \| \id\| ^2 - 2\trace(AA^t) = \| A^t A\| ^2 + \| \id\| ^2 - 2\trace(A^t A) = \| A^t A - \id\| ^2$.

    Now, let $M,N \in \ST{d,n}$.
    Since multiplication by an orthogonal matrix preserves the Frobenius norm, we may assume that $M$ and $N$ are matrices given by blocks:
    \[
        M = \begin{bmatrix}
            A \\ B
        \end{bmatrix},\;\;
        N = \begin{bmatrix}
            \id \\ 0
        \end{bmatrix},
    \]
    where the top block is $d \times d$ and the bottom one is $(n-d) \times d$.
    Note that, by assumption, we have $A^t A + B^t B = \id$.
    We now bound as follows:
    \begin{align*}
        \| MM^t - \id\| ^2 &= \| A A^t - \id\| ^2 + \| A B^t\| ^2 + \| BA^t\| ^2 + \| B B^t\| ^2\\
                         &= \| A^tA - \id\| ^2 + 2\| BA^t\| ^2 + \| B B^t\| ^2\\
                         &= 2\left(\| BA^t\| ^2 + \| B B^t\| ^2\right) = 2 \| B M^t\| ^2 = 2 \| B\| ^2\\
                         &\leq 2 \left(\| A-\id\| ^2 + \| B^t\| ^2\right) = 2\| M - \id\| ^2,
    \end{align*}
    where we used the previous observations in the second and third equality, and \cref{product-by-frame-isometry} in the fifth equality.
\end{proof}


To prove a partial converse, we need the following lemma.

\begin{lem}
    \label{square-root-psd}
    If $P$ is a symmetric and positive semidefinite $d\times d$ matrix then $\| P^{1/2} - \id\| \leq \| P - \id\|$.
\end{lem}
\begin{proof}
    Diagonalizing $P$ in an orthonormal basis, we can assume that $P$ is diagonal with non-negative entries.
    Since then the square root of $P$ corresponds to taking the square root of each diagonal entry, it is enough to prove that, for any non-negative real number $a$ one has $|\sqrt{a} - 1| \leq |a - 1|$, which is clear.
\end{proof}

The next lemma is not needed in what follows, but, since it is a direct consequence of the previous result, we give it here.

\begin{lem}
    \label{approximate-stiefel-lem}
    Let $M \in \R^{n \times d}$ and $M = UP$ be a polar decomposition. Then $\| M - U\|  \leq \| M^t M - \id\| $.
\end{lem}
\begin{proof}
    By definition, $P^2 = M^t M$. So the result follows from \cref{square-root-psd}.
\end{proof}

We are ready to prove the converse.

\begin{lem}
    \label{lem-equivalence-metrics-grassmannians}
    Let $M,N \in \ST{d,n}$.
    Then there exists $\Omega \in O(d)$ such that $\| M \Omega - N\|  \leq \| MM^t - NN^t\|$.
    In particular, $\dqsa \leq  \dFrsa$.
\end{lem}
\begin{proof}
    The second claim is a consequence of the first one.
    For the first claim, let $\Omega \in O(d)$ minimize $\| M \Omega - N\| $.
    From \cref{polar-decomposition-and-procrustes-section}, we know that $\Omega$ is part of a polar decomposition $M^t N = \Omega (N^t M M^t N)^{1/2}$.
    Now $\|M \Omega - N\| \leq \|\Omega - M^t N\| = \|\Omega - \Omega (N^t M M^t N)^{1/2}\| = \|(N^t M M^t N)^{1/2} - \id\| \leq \|(N^t M M^t N)^{1/2} - \id\|$, by \cref{square-root-psd}.
%
\end{proof}

We now recall the computation of the reach of the Grassmannians, due to Tinarrage (\cite{tinarrage}).

\begin{constr}
    \label{projection-to-grassmannian}
    Let $A \in \R^{n \times n}$.
    Define $A^s = (A + A^t)/2$, and let $A^s = \Omega D \Omega^t$ be a diagonalization of $A^s$ by an ordered orthonormal basis $\Omega \in O(n)$, where the diagonal entries of $D$ contain the eigenvalues of $A^s$ in decreasing order.
    Let $J_d$ be the diagonal $n\times n$ matrix with the first $d$ diagonal entries equal to $1$ and the rest equal to $0$.
    Let $\pi(A) = \Omega J_d \Omega^t$.
\end{constr}

\begin{lem}[{\cite[Lemma~2.1]{tinarrage}}]
    \label{tinarrages-lemma}
    If $\dFr{A}{\GR{d,n}} < \sqrt{2}/2$, then $\pi(A)$ is the unique minimizer of
    \[
        \min_{B \in \GR{d,n}} \| A - B\| .
    \]
    Moreover, we have $\reach(\GR{d,n}) = \sqrt{2}/2$.
\end{lem}

In particular, if $A \in \GR{d,n}^\epsilon$ for $\epsilon \leq \sqrt{2}/2$, then $\pi(A)$ is independent of the choice of orthonormal basis $\Omega$.
We deduce the following.

\begin{prop}
    \label{pi-is-inverse}
    Let $\epsilon \leq \sqrt{2}/2$.
    The inclusion $\GR{d,n} \subseteq \GR{d,n}^\epsilon$ is a homotopy equivalence, with inverse given by the projection $\pi$.
    This is natural in both $n$ and $\epsilon$.
    Moreover, the projections assemble into a homotopy inverse of the inclusion $\GR{d} \subseteq \GR{d}^\epsilon$.
\end{prop}
\begin{proof}
    The first claim follows from \cref{projection-well-defined} and \cref{tinarrages-lemma}.
    In the second claim, naturality means that the inclusion maps $i : \GR{d,n} \to \GR{d,n}^\epsilon$ commute with the inclusions into $\GR{d,n+1}$ and $\GR{d,n+1}^\epsilon$, which is clear, and that the projections do as well. The fact that projections commute with the inclusion $i$ follows from \cref{projection-to-grassmannian}.
    The third claim follows at once from naturality.
\end{proof}

\subsection{Metrics on orthogonal groups}
\label{metrics-on-orth-section}

In this section, we will make use of basic Riemannian geometry; see for instance \cite{jL}.
Our main goals are to calculate the reach (\cref{thickening-subsection}) of the orthogonal groups, seen as a subspace of the metric space of square matrices, with metric given by the Frobenius distance; to compare the Frobenius distance to the geodesic distance; and to calculate the systole (\cref{systole-def}) of the orthogonal groups using the geodesic distance.

The geodesic distance we consider on the orthogonal groups is the one that comes from the bi-invariant metric given by the smooth inclusion $O(d) \subseteq \R^{d\times d}$, where the inner product on $\R^{d \times d}$ is the one associated to the Frobenius norm (\cite[Example~3.16(e)]{jL}).

\begin{lem}
    \label{real-det-is-lip}
    Let $M,N \in O(d)$ such that $\det(M) = 1 = - \det(N)$.
    Then $\| M-N\|  \geq 2$.
\end{lem}
\begin{proof}
    Since $O(d)$ acts on itself by isometries, it suffices to show
    that $\| N - \id\|  \geq 2$ whenever $\det(N) = -1$.
    Since $N$ is orthogonal, there is an orthogonal change of basis that
    takes it to be diagonal by blocks, where blocks are $1 \times 1$ and equal
    $1$ or $-1$, or $2 \times 2$ and rotation matrices.
    We can thus assume that $N$ is diagonal by blocks of the form above.
    If any of the blocks is a $1\times 1$ block of the form $-1$ we are done.
    Otherwise, there must be a block that consists of a rotation with negative determinant.
    In this case,
    \[
        \| N - \id\| \geq \sqrt{2(\sin(\theta))^2 + (1-\cos(\theta))^2 + (1+\cos(\theta))^2} = 2.
        \qedhere
    \]
\end{proof}

The following lemma appears in \cite{BLW} and is a key result that lets us compare the geodesic distance of the orthogonal groups to the Frobenius distance.
Before giving the result, we recall that any metric space has an induced geodesic distance (also known as a path-length distance), where the distance between any two points is taken to be the infimum of the lengths of the continuous paths between them (\cite[Section~2]{BLW}, \cite[Section~2.3.3]{BBS}).
If the metric space is not connected, this is an extended distance.
Finally, if the metric space consists of a manifold smoothly embedded in Euclidean space, then this geodesic distance coincides with the geodesic distance associated to the Riemannian metric induced by the embedding into Euclidean space (\cite[Exercise~5.1.8]{BBS}).

\begin{lem}[{\cite[Lemma~3]{BLW}}]
    \label{geodesic-vs-embedded}
    Let $S \subseteq \R^N$ be a closed set such that $R = \reach(S) > 0$.
    Let $d_S$ denote the geodesic distance on $S$ induced by the restriction of the Euclidean distance.
    If $x,y \in S$ are such that $\|x - y\|_2 < 2R$, then $d_{S}(x,y) \leq 2\, R\, \arcsin\left(\frac{\|x - y\|}{2R}\right)$.
    \qed
\end{lem}

\begin{lem}
    \label{key-lemma}
    Let $A \in \R^{d \times d}$ and $\Omega \in O(d)$ such that $\| A - \Omega\|  < 1$, then $A$ has full rank.
\end{lem}
\begin{proof}
    It is enough to show that $\Omega^t A$ has full rank, so, since the Frobenius norm is $O(d)$-invariant, it is sufficient to prove it for the case $\Omega = \id$.
    Let $B = \id - A$.
    Since $\|B\| < 1$, and the Frobenius norm is submultiplicative, the matrix $\sum_{n \geq 0} B^n$ is well-defined.
    Finally, we have $(\id - B) \left(\sum_{n \geq 0} B^n\right) = \id$, so $\id - B = A$ is invertible, as required.
\end{proof}

\begin{lem}
    \label{reach-O-computed}
    Consider $O(d) \subseteq \R^{d \times d}$ with the Frobenius distance.
    Then $\reach(O(d)) = 1$.
\end{lem}
\begin{proof}
    Let $M_x$ denote the $d\times d$ diagonal matrix with $1$ in all diagonal entries except the first one, which is $x$.
    Since a polar decomposition of $M_0$ is given by $M_0 = \id M_0$, we have that the identity is a closest orthogonal matrix to $M_0$, by \cref{polar-decomposition-and-procrustes-section}.
    Now note that $\|M_0 - \id\| = 1 = \|M_0 - M_{-1}\|$, so $M_{-1}$ is a closest orthogonal matrix to $M_0$ too, and thus $\reach(O(d)) \leq 1$.

    To conclude, we must show that, if $\|M - \id\| < 1$, then $M$ has a unique closest orthogonal matrix.
    This is a consequence of \cref{key-lemma} and \cref{unique-polar-decomposition-lem}.
\end{proof}

Combined, \cref{geodesic-vs-embedded} and \cref{reach-O-computed} give us upper bounds for the geodesic distance of $O(d)$ in terms of the Frobenius distance.
We will need the following specific bound.

\begin{cor}
    \label{comparison-geodesic-and-frobenius}
    Let $d_G$ be the geodesic distance on $O(d)$ induced by the embedding $O(d) \subseteq \R^{d \times d}$.
    For $M,N \in O(d)$, if $\|M - N\| < 1$, then $d_G(M,N) < \sqrt{2} \pi / 4$.
\end{cor}
\begin{proof}
    By inspection, $2 \arcsin(1/2) < \sqrt{2} \pi /4$.
    Note that a slightly tighter bound is possible, but we prefer this one for readability.
\end{proof}

\begin{lem}
    \label{systole-of-O}
    Using the geodesic distance $d_G$, we have $\sysp(O(d)) = 2 \sqrt{2} \pi$.
\end{lem}
\begin{proof}
    Since $O(d)$ is a group acting on itself by isometries, to compute $\sysp(O(d))$, it suffices to consider loops on the identity matrix $\id \in O(d)$.
    We observe that the constant speed, locally length-minimizing curves from the identity to itself are in bijection with the logarithms $L \in \so(d)$ of the identity.
    More specifically, any such curve can be written as $t \mapsto \exp(t L)$ with $L \in \so(d)$ (\cite[Proposition~5.19]{jL}).
    We also observe that the speed of such a curve is $\|L\|$, the Frobenius norm of $L$.

    So let $L \in \so(d)$ be a skew-symmetric matrix such that $\exp(L) = \id$.
    Since $L$ is skew-symmetric, there exists an orthogonal change of basis such that $L$ is a block-diagonal matrix with either $1 \times 1$ blocks containing a $0$ or $2\times 2$ blocks with $0$ in the diagonal and $\lambda, -\lambda \in \R$ in the anti-diagonal.
    Since $L$ is a logarithm of the identity, the number $\lambda$ in any of these blocks must be an integer multiple of $2\pi$.
    It follows that $\|L\| = \sqrt{ 2 \sum (2 \pi n_i)^2} = 2\pi \sqrt{2} \sqrt{ \sum n_i^2 }$.
    Since we are considering non-nullhomotopic loops, at least one of the integers $n_i$ must be non-zero, and thus the smallest value of $\|L\|$ is attained when one of the $n_i$ is $1$ and the rest are zero, and in such case we have $\|L\| = 2 \sqrt{2}\pi$, concluding the proof.
\end{proof}

\subsection{Riemannian manifolds and covering spaces}

In this section, we use the systole to give a lower bound for the distance between two distinct elements in the fiber of a covering map between Riemannian manifolds, and we prove a metrically controlled lifting property for covering maps between Riemannian manifolds.

\begin{lem}
    \label{reach-and-systole}
    Let $G \to H$ be a covering map between Riemannian manifolds.
    Fix $h \in H$ and let $F \subseteq G$ be the fiber of $h$.
    Then, the infimum of the distances between distinct elements of $F$ is bounded below by $\sysp(H)$.
\end{lem}
\begin{proof}
    Let $a \neq b \in F$ and consider a geodesic between them.
    The length of this path is equal to the length of the path mapped to $H$, since $G \to H$ is a local isometry.
    The path mapped to $H$ cannot be nullhomotopic since $a$ and $b$ are distinct points of the fiber, so its length is bounded below by the systole of $H$.
\end{proof}

\begin{lem}
    \label{metric-lift}
    Let $\zeta : G \to H$ be a covering map between Riemannian manifolds, with $H$ compact.
    \begin{enumerate}
        \item Let $h,h' \in H$ be such that $d_H(h,h') < \sysp(H)/2$ and let $g \in G$ be a preimage of $h$.
        Then, there exists a unique preimage $g'$ of $h'$ such that $d_G(g, g') < \sysp(H)/2$, and $g'$ has the property that $d_G(g,g') = d_H(h,h')$.
        \item Let $U$ be locally path connected and simply connected and let $v,w : U \to H$ continuous such that, for all $z \in U$, $d_H(v(z),w(z)) < \sysp(H)/2$.
        Then, there exist lifts $\overline{v},\overline{w} : U \to G$ of $v$ and $w$ respectively such that, for all $z \in U$, $d_G(\overline{v}(z), \overline{w}(z)) = d_H(v(z), w(z))$.
    \end{enumerate}
\end{lem}
\begin{proof}
    We start with the first claim.
    To prove existence, consider a shortest geodesic from $h$ to $h'$, whose length must be strictly less than $\sysp(H)/2$.
    By lifting this path we get a preimage $g'$ of $h'$ such that $d_G(g,g') \leq d_H(h,h')$.
    To prove uniqueness, suppose that $g'' \neq g'$ is a preimage of $h'$.
    By \cref{reach-and-systole}, we have $d_G(g'',g') \geq \sysp(H)$, so $d_G(g,g'') \geq \sysp(H)/2$.
    Finally, since $G \to H$ is $1$-Lipschitz, we have $d_H(h,h') \leq d_G(g,g')$, and thus $d_G(g,g') = d_H(h,h')$.

    For the second claim, let $y \in U$ and use the first claim to choose lifts $g,g' \in G$ of $v(y)$ and $w(y)$ respectively such that $d_G(g,g') = d_H(v(y),w(y))$.
    Since $U$ is locally path connected and simply connected, there exist unique lifts $\overline{v}$ and $\overline{w}$ of $v$ and $w$ respectively such that $\overline{v}(y) = g$ and $\overline{w}(y) = g'$.
    Consider the subset $U' = \{z \in U : d_G(\overline{v}(z),\overline{w}(z)) = d_H(v(z), w(z))\} \subseteq U$.
    This subset is closed and non-empty, so it suffices to show that it is open, since $U$ is connected.
    We conclude the proof by proving this fact.

    Let $z \in U'$ and let $\epsilon = \sysp(H)/2 - d_H(v(z),w(z)) > 0$.
    Let $N$ be an open neighborhood of $z$ such that $\overline{v}(N)$ is contained in the open ball of radius $\epsilon/2$ around $\overline{v}(z)$, and $\overline{w}(N)$ is contained in the open ball of radius $\epsilon/2$ around $\overline{w}(z)$.
    If $z' \in N$, then
    \begin{align*}
    d_G(\overline{v}(z'),\overline{w}(z')) &\leq d_G(\overline{v}(z'),\overline{v}(z)) + d_G(\overline{v}(z), \overline{w}(z)) + d_G(\overline{w}(z),\overline{w}(z'))\\
    &= d_G(\overline{v}(z'),\overline{v}(z)) + d_H(v(z), w(z)) + d_G(\overline{w}(z),\overline{w}(z'))\\
    &< d_H(v(z), w(z)) + 2 \epsilon/2 = \sysp(H)/2
    \end{align*}
    From the first claim, it follows that $d_G(\overline{v}(z'), \overline{w}(z')) = d_H(v(z'), w(z'))$, so $N \subseteq U'$, and thus $U'$ is open, concluding the proof.
\end{proof}

\printbibliography

\end{document}